\documentclass[a4paper, 12pt]{amsart} %for arXiv

\usepackage[pdftex]{graphicx} %for arXiv

\usepackage{amssymb}
\usepackage{amsmath}
\usepackage{amsthm}
\usepackage{amscd}

\usepackage{comment}
\usepackage[all]{xy}
\usepackage{color}

\usepackage[utf8]{inputenc}
\usepackage{tikz}
\usetikzlibrary{positioning}
\usetikzlibrary{intersections}
\usetikzlibrary{calc, quotes, angles}

\title{Ruled surfaces}

\date{\today}
\theoremstyle{plain}
\newtheorem{thm}{Theorem}[section]

\newtheorem{lem}[thm]{Lemma}

\newtheorem{slem}[thm]{Sublemma}
\newtheorem{cor}[thm]{Corollary}
\newtheorem{prop}[thm]{Proposition}
\newtheorem{defn}[thm]{Definition}
\newtheorem{rem}[thm]{Remark}

\newtheorem{claim}[thm]{Claim}
\newtheorem{ex}[thm]{Example}

\makeatletter

\newcommand{\N}{\mathbb{N}}

\newcommand{\diam}[0]{\mathrm{diam}\,}

\newcommand{\e}[0]{\epsilon}

\newcommand{\supp}[0]{\mathrm{supp}}

\newcommand{\R}[0]{\mathbb R}
\newcommand{\pa}[0]{\partial}
\newcommand{\CAT}{\mathrm{CAT}}
\newcommand{\up}[0]{\uparrow}

\newcommand{\sing}[0]{\rm sing}

\newcommand{\inte}[0]{\mathrm{int}}
\DeclareMathOperator{\interior}{int}          
\newcommand{\ca}[0]{\mathcal}

\newcommand{\pmed}[0]{\par\medskip}
\newcommand{\psmall}[0]{\par\smallskip}
\newcommand{\pbig}[0]{\par\bigskip}
\newcommand{\n}[0]{\noindent}

\newcommand{\beq}[0]{\begin{equation}}
\newcommand{\eeq}[0]{\end{equation}}
\newcommand{\beqq}[0]{\begin{equation*}}
\newcommand{\eeqq}[0]{\end{equation*}}

\newcommand{\bali}[0]{\begin{align}}
\newcommand{\eali}[0]{\end{align}}

\newcommand{\benum}[0]{\begin{enumerate}}
\newcommand{\eenum}[0]{\end{enumerate}}

\begin{document}

\title[Two-dimensional spaces 
     with curvature bounded above]
{Two-dimensional metric spaces 
     with curvature bounded above I }

\author[K.Nagano]{Koichi Nagano}
\author[T.Shioya]{Takashi Shioya} 
\author[T.Yamaguchi]{Takao Yamaguchi}

\address{Koichi Nagano, Institute of Mathematics, University of Tsukuba,  Tsukuba 305-8571, Japan}
\email{nagano@math.tsukuba.ac.jp}

\address{Takashi Shioya, Mathematical Institute, Tohoku  University, Sendai  980-8578, Japan}
\email{shioya@math.tohoku.ac.jp}

\address{Takao Yamaguchi, Department of Mathematics, Kyoto University, Kyoto  606-8502, Japan}
%\email{takaoy@math.kyoto-u.ac.jp}
 \curraddr{Institute of Mathematics, University of Tsukuba,  Tsukuba 305-8571, Japan}
 \email{takao@math.tsukuba.ac.jp}

\thanks{This work was supported by JSPS KAKENHI Grant Numbers  18H01118, 15H05739, 19K03459,15K13436, 26610012, 21740036,17204003,18740023  }

\begin{abstract}
We determine the local geometric structure of two-dimensional metric spaces 
with curvature bounded above as the union of finitely many properly
embedded/branched  immersed Lipschitz disks.
As a result, we obtain a graph structure of the topological singular 
point set of such a singular surface. 
\end{abstract}

\maketitle

\tableofcontents

%%%%%%%%%start section {introduction}
%\intro.tex 
\section{Introduction} \label{sec:intro}

Let $X$ be a  locally compact, geodesically complete
Alexandrov space with curvature bounded above.
In this paper, we are concerned with the local structure of 
$X$. In general $X$ may have very complicated local geometry.
For instance, $X$ may have no polyhedral structure even in local.
There is such a two-dimensional space constructed by Kleiner
(cf.\cite{Ng:asymp}).
In the present paper, we completely describe the local geometry 
of such spaces in dimension two.

The study of metric spaces with curvature bounded above 
began with the work of Alexandrov \cite{Alex:Uber-eine}.
For the dimensions of such spaces $X$, 
Kleiner \cite{Kl:local} proved that the topological dimension
%$\dim_{\rm top}X$ of $X$
coincides with the maximal dimension of topological 
manifolds embedded in $X$. For  {\it geodesically complete} metric spaces $X$ with curvature bounded above, 
Otsu-Tanoue \cite{OT:riem} implicitly 
showed that the topological dimension
coincides with the Hausdorff dimension, %$\dim_HX$, 
which has been verified via a different method by a recent work
due to Lytchak-Nagano \cite{NL:geodesically}. \cite{NL:geodesically} has also clarified that the local geometric properties of 
geodesically complete metric spaces $X$ with curvature bounded above  have a lot of analogues to those of Alexandrov spaces with curvature bounded below
(see also  Remarks \ref{rem:rectifiable} and \ref{rem:metric-sphere} below).
Lytchak-Stadler \cite{LS} have recently proved that
for every convex open ball in a CAT$(\kappa)$-space
there exists a complete CAT$(-1)$-metric on the ball
that is locally bi-Lipschitz to the original CAT$(\kappa)$-metric;
in particular,
in local considerations on topological properties of CAT$(\kappa)$-spaces,
we may assume $\kappa$ to be $-1$.

For basic textbooks in this subject, there are several general references,
and we refer to 
Ballmann \cite{Ball:lectures}, Bridson-Haefliger \cite{bridson-haefliger},
Burago-Burago-Ivanov \cite{BBI},   Alexander-Kapovitch-Petrunin \cite{AKP}.

\par
Now let us consider our main concern, the two-dimensional such spaces.
The study in this particular dimension began with a classical  deep work 
due to Alexandrov-Zalgeller \cite{AZ:bddcurv} on two-dimensional topological manifolds with more general curvature bound, called the {\it bounded curvature}.  
They constructed the curvature
measure on such surfaces and established the Gauss-Bonnet theorem.
See also Reshetnyak \cite{Rsh:2mfd} for the work from an analytic point of view.
Generalizing \cite{AZ:bddcurv} and succeeding the works of
Ballmann-Buyalo \cite{BallBuy} and 
Arsinova-Buyalo \cite{ArsBuy}, Burago-Buyalo \cite{BurBuy:upperII} established the theory of two-dimensional polyhedra with curvature bounded above. 

Here it should be emphasized that 
there were  no general results determining  local structure even in 
dimension two. 
The purpose of this paper is to determine the  general local geometric structure 
of  two-dimensional geodesically complete metric spaces with curvature bounded above.
\par

Let $X$ be a two-dimensional locally compact, geodesically complete
metric space with curvature $\le\kappa$ for a constant $\kappa$.
For every $p\in X$, the space of directions $\Sigma_p=\Sigma_p(X)$  is 
the disjoint union of finitely  many points and connected finite graphs.
Since we are interested in the local structure, we assume 
the most essential case when 
$\Sigma_p$ is a connected graph, called a CAT(1)-graph
 (see Section \ref{sec:basic}).
We shall determine the geometry of the closed $r$-ball $B(p,r)$
around $p$ for small enough $r>0$ as follows.

Let $\mathcal S(X)$ denote
the set of all topological singular points in $X$.
For $\ell\ge 2\pi$ and $r>0$, we denote by $D^2(\ell; r)$ the  closed disk of radius $r$
around the vertex $O$ in the Euclidean cone over the circle of length
$\ell$. 
A map $f:D^2(\ell;r)\to B(p,r)$ is called {\em proper} if 
$f^{-1}(\partial B(p,r))=\partial D^2(\ell;r)$.
Let $\tau_p(r)$ denote a function 
depending on $p$ and $r$ satisfying $\lim_{r\to 0}\tau_p(r)=0$.
Let $S(p,r)$ denote the metric sphere $\pa B(p,r)$.

The main result in this paper is stated as follows.

\begin{thm} \label{thm:main}
For every $p\in X$ such that $\Sigma_p$ is a connected
graph, there exists a positive number $r_0$ such that
for every $0< r\le r_0$, $B(p,r)$ is a union of images ${\rm Im} f_i$ 
of finitely many proper 
Lipschitz immersions $f_i:D^2(\ell_i;r)\to B(p,r)$  
for some $\ell_i\ge 2\pi$,   possibly with branch point 
$f_i^{-1}(p)=\{ O\}$ satisfying the following $:$
\begin{enumerate}
 \item With respect to the length metric induced from $X$, 
   ${\rm Im} f_i$ are $\CAT(\kappa)$-spaces $;$
 \item Either $f_i$ is an embedding, or else 
  $f_i(\pa D^2(\ell_i;r))$ is the union of two circles of length $\ge 2\pi r$ connected by an arc, which could be a point.
   In the latter case, $\ell_i\ge 4\pi$ $;$
 \item  The bi-Lipschitz constant of $f_i$  is less than $1+\tau_p(r)$ when $f_i$ is an embedding. If $f_i$ is a branched immersion, the local bi-Lipschitz constant of $f_i$ 
except at $\{ O\}$  is less than $1+\tau_p(r)$. 
\end{enumerate}
 Moreover, $\mathcal S(X)\cap B(p,r)$ consists of finitely many 
   simple Lipschitz arcs starting from $p$ and reaching 
   $S(p,r)$.
\end{thm}

\begin{rem} \upshape
One might ask if it is possible to fill the ball $B(p,r)$ with those ${\rm Im} f_i$
that are convex in $X$ or properly embedded disks. 
However, both are impossible in general. 
For example, take the Euclidean cone $X$ over the union of two circles of length $2\pi$ joined by an arc. Note that any metric ball around the vertex of 
$X$ can not be written as a union of properly embedded disks as described in Theorem \ref{thm:main}.
For an example showing the impossibility of filling the ball
via convex properly embedded $\CAT(\kappa)$-disks, see
Example \ref{ex:3} for instance.
 \end{rem}

From the proof of Theorem \ref{thm:main}, 
we actually have the following.

\begin{cor} \label{cor:main-zeta}
Let $r=r_p$ be sufficiently small as in Theorem \ref{thm:main}. Then for any locally injective continuous map $\zeta:[a,b]\to\Sigma_p(X)$,
there is a closed subset $E$ of $X$ containing $p$
satisfying 
\begin{enumerate}
 \item $E$ is a $\CAT(\kappa)$-space with respect to 
 the length metric\,$;$ 
 \item $\Sigma_p(E)={\rm Im}(\zeta)\,;$ 
 \item $\pa E\subset S(p,r)$ possibly except the segments from $p$ directing to the endpoints of $\zeta$.
Here $\pa E$ denotes the set of points of $E$
where local geodesically completeness of $E$ fails.

\end{enumerate}
\end{cor}
%When $\zeta$ is surjective, $E$
The set $E$ is the image of a locally almost isometric, 
branched immersion, except at the branch locus $\{ p\}$,  
from 
the closed disk of radius $r$ 
around the vertex in the Euclidean cone over the 
interval of length $L(\zeta)$.
%%
%as described in Theorem \ref{thm:main}. 
When $\zeta$ is surjective in addition, this provides 
another description of $B(p,r)$.
%similar to Theorem \ref{thm:main}.

Using Theorem \ref{thm:main}, we can define a 
metric graph structure on $\ca S(X)$ in a generalized sense
(see Definition \ref{defn:metric-graph}), and we have 
\pmed

\begin{cor} \label{cor:graph}
Suppose that $\Sigma_p$ is a connected graph for every $p\in X$.
Then with respect to the induced length structure,
$\mathcal S(X)$ is isometric to a metric graph
having $($possibly locally uncountably many vertices, but$)$
the vertices of locally finite order.
\end{cor}
\pmed 

%In the proof of Theorem \ref{thm:main}(1)
%for branched immersed disks, we need to use the notion of turn, which was 
%first defined in the context of surfaces with 
%bounded curvature in \cite{AZ:bddcurv}.
%
%As a byproduct of the proof, we obtain the following.
%For the precise definition, see Definition 
%\ref{defn:turn-C}.
%
%\begin{thm} \label{thm:bdd-turn-C} 
%$\ca S(X)$ has locally finite turn variation
%in a generalized sense.
%\end{thm}

\begin{rem} \label{rem:rectifiable}  \upshape
In the general dimension, \cite{NL:geodesically} has characterized the singular set in the $k$-dimensional 
	part as a countably $(k-1)$-rectifiable set.   Corollary \ref{cor:graph} 
%and Theorem \ref{thm:bdd-turn-C} 
gives a refinement of this result in dimension two.
\end{rem}

Recall that a compact metric graph $\Sigma$ is a 
$\CAT(\mu)$-graph \,$(\mu>0)$\,
if every non-contractible loop in $\Sigma$ has length $\ge 2\pi/\sqrt{\mu}$.
\pmed
\begin{cor} \label{cor:circle}
For a given $p\in X$ such that $\Sigma_p$ is a connected graph, 
there exists a positive number $r_p$ such that 
for every $0<r\le r_p$, $S(p,r)$ with the interior metric is a 
$\CAT(\mu_\kappa(r))$-graph 
having the same homotopy type as $\Sigma_p$, where 
$\mu_\kappa(r)$ is given by the sharp constant 
 
\begin{equation*}
 \mu_\kappa(r)= \begin{cases}
  \left(\frac{\sin \sqrt \kappa r}{\sqrt \kappa}\right)^{-2}
                              &\quad  \text{if \, $\kappa>0$,} \\
        \hspace{0.7cm}  r ^{-2}   &\quad \text{if \, $\kappa=0$,}    \\
  \left(\frac{\sinh \sqrt{-\kappa} r}{\sqrt{- \kappa}}\right)^{-2}
                              &\quad \text{if  \, $\kappa<0$}.
      \end{cases}
\end{equation*}
\end{cor} 
\pmed
\begin{rem} \label{rem:metric-sphere}  \upshape
A result  in \cite{NL:geodesically} shows that for every small 
$r$,  $S(p,r)$ has the same homotopy type as $\Sigma_p$ in the general dimension.
Corollary \ref{cor:circle} 
gives a refinement of this result in dimension two.
\end{rem}

\begin{rem} \label{rem:only-local}  \upshape
 All the results in this paper are local.
Therefore they are also valid under the assumption of 
local geodesic completeness of $X$.
\end{rem}

\pmed
The idea of the proof of the main result is as follows.
We know the structure of the space $\Sigma_p$  of directions at $p$, which is completely characterized as a ${\rm CAT}(1)$-graph without endpoints.
If we rescale the metric of $X$ by the factor $1/r$, then 
$(\frac{1}{r} X,p)$ converges to the tangent cone $(K_p,o_p)$ at $p$
as $r\to 0$ with respect to the pointed Gromov-Hausdorff topology.
Let $\Sigma_p^{\rm sing}$ be a small neighborhood of the vertices of the 
graph $\Sigma_p$, and
$\Sigma_p^{\rm reg}$ the complement of $\Sigma_p^{\rm sing}$. 
Now the convergence theorem (\cite{Ng:volume}) applied to 
the unit cone $K_1(\Sigma_p^{\rm reg})$ over $\Sigma_p^{\rm reg}$ yields the existence of a 
Lipschitz domain $B^{\rm reg}(p,r)$ of $B(p,r)$ consisting
of finitely many sectors corresponding to sectors of 
$K_1(\Sigma_p^{\rm reg})$. One can consider $B^{\rm reg}(p,r)$ as a regular part of $B(p,r)$.
The main problem is to determine the structure of the singular part $B^{\rm sing}(p,r)$,
the complement of $B^{\rm reg}(p,r)$ in $B(p,r)$.
To carry out  this, we consider finitely many thin ruled surfaces, say $S_{ij}$ here,
and fill $B^{\rm sing}(p,r)$ using them.  A key is to show 
that those ruled surfaces are ${\rm CAT}(\kappa)$-spaces 
with respect to the {\it interior metrics} and are homeomorphic to a disk. 
According to Alexandrov's result in 
\cite{alexandrov-ruled},
every ruled surface in a ${\rm CAT}(\kappa)$-space is also a 
 ${\rm CAT}(\kappa)$-space  with respect to the {\it pull-back metric}.
Obviously, the interior metric and the pull-back metric 
are completely different from each other in general.  
Therefore we have to show that in our thin  ruled surfaces pull-back metrics coincide with the 
interior metrics. After achieving this,
it turns out that the topological singular point set $\ca S(X)$ locally arises from the intersections of those thin ruled surfaces 
$S_{ij}$. 
We investigate how those ruled surfaces 
meet each other to get the structure of 
$\mathcal S(X)\cap B(p,r)$ as the union of finitely many 
Lipschitz curves.  Combining the structures of both 
$B^{\rm reg}(p,r)$ and $B^{\rm sing}(p,r)$ and considering the graph structure
of $\Sigma_p$, we define the embeddings or the branched immersions
$f_i:D^2(\ell_i;r) \to B(p,r)$ as described in Theorem \ref{thm:main}.\par
\psmall
%%%%%%%%%%%%%%%%%%  Some related results %%%%%%%%%%%%%%%%%%%
As related studies on ruled surfaces, Petrunin-Stadler \cite{PS} have proved that for 
metric minimizing disks in CAT(0)-spaces, the pull-back metrics 
on the disks are CAT(0), which is a 
generalization of  Alexandrov's result \cite{alexandrov-ruled} on ruled surfaces in the CAT(0)-setting.
According to Stadler \cite[Theorem 2]{S}, for any Jordan triangle in a CAT(0)-space,
every minimal disk filling of the triangle is an embedded disk
that is CAT(0) with respect to the interior metric.

%%%%%%%%%%%%%%%%%%%%%%%%%%%%%%%%%%%%%%%%%%%%%
\psmall
The organization of the paper is as follows.

In Section \ref{sec:basic}, we recall and verify basic results for locally compact,
geodesically complete Alexandrov spaces with curvature bounded above.

In Section \ref{sec:ruled-alex},
we give basic properties of a ruled surface $S$ in a
$\CAT(\kappa)$-space.
We discuss the pullback metric, the induced metric,
the interior metric of $S$ and their relations.
In the original argument in Alexandrov \cite{alexandrov-ruled}, there are several unclear points for the authors. For instance,
there is no description  in \cite{alexandrov-ruled} about quasicontinuous
monotone representations.
We make clear all these points.

In Section \ref{sec:ruled}, which is a key section, we investigate a thin ruled surface $S$ in 
a two dimensional space, and prove that  $S$  actually admits the
induced metric
and therefore becomes a $\CAT(\kappa)$-space with respect to the interior metric.
Then we obtain the crucial property that $S$ is homeomorphic to a disk.

In Section \ref{sec:fill}, we fill $B(p,r)$ via those embedded/branched immersed disks 
using thin ruled surfaces essentially.
We prove Theorem \ref{thm:main}(1), (2) and (3) 
except (1) for branched immersed disks.

In Section \ref{sec:graph}, we describe $\mathcal S(X)\cap B(p,r)$
as a union of finitely many Lipschitz curves starting from $p$ 
and reaching points of $S(p,r)$.
The structure of generalized metric graph of $\ca S(X)$ is also discussed there.

In Section \ref{sec:approx}, we provide the proof of 
Theorem \ref{thm:main}(1) for branched immersed disks as well as Corollary \ref{cor:main-zeta}. 
%As a byproduct of the proof,
%Theorem \ref{thm:bdd-turn-C} is obtained there.

In Appendix \ref{sec:append}, following the basic idea of 
\cite{alexandrov-ruled}, we give the proof of 
Alexandov's result on ruled sufaces in $\CAT(\kappa)$-spaces
based on the results proved in Section \ref{sec:ruled-alex}.

\psmall
\n
Burago-Buyalo \cite{BurBuy:upperII} 
gave a complete characterization of two-dimensional 
polyhedra of curvature bounded  above. 
In the second part \cite{KNSYII} of our works, we show 
the following:

\psmall\n
(a)\, We provide sufficient conditions for 
two-dimensional metric spaces to have  curvature bounded above, which shows that the results in this paper
completely characterize the local structure of two-dimensional metric spaces 
with curvature bounded above.
\psmall\n
(b)\,Every pointed two-dimensional geodesically complete locally 
$\CAT(\kappa)$-space $(X,p)$ can be approximated by a sequence 
of two-dimensional  pointed geodesically complete, polyhedral locally 
$\CAT(\kappa)$-spaces $(X_n,p_n)$ having {\it  the same homotopy type as $X$}
with respect to the pointed 
Gromov-Hausdorff topology. This solves a problem raised in Burago-Buyalo \cite{BurBuy:upperII}.
\psmall\n
(c)\,We establish a Gauss-Bonnet type theorem for 
two-dimensional geodesically complete locally 
$\CAT(\kappa)$-spaces.
\psmall

Most results in the present paper were announced in \cite{Ya:upper}.

\vspace{1em}\noindent 
{\bf Acknowledgements}
First of all, the authors would like to thank Bruce Kleiner.
The outline of the results in this paper came from the discussions with him 
on the basic idea many years ago.
We would also like to thank  Alexander  Lytchak
for informing recent related results on ruled surfaces and minimal filling disks in $\CAT(0)$-spaces.
We would also like to thank Werner Ballmann, Yuri Burago,
Sergei Buyalo,  Misha Gromov for their interest to this work.
This work was partially supported by IHES, while the last named author was in
residence there, during the summer of 2005.

We would like to thank the referee for carefully 
reading our manuscript and for valuable comments and
suggestions.
Lemma \ref{lem:comparison} was suggested by the referee 
to the authors.

\pmed

\setcounter{equation}{0}
\section{Basic properties of ${\rm CAT}(\kappa)$-spaces} \label{sec:basic}
\pmed

 For some basic results in this section, we refer to 
\cite{bridson-haefliger}, \cite{BBI}.

The distance between two points $x,y$ in a metric space
$X$ is denoted by $|x,y|$ or $ |x,y|_X$, and  $d(x,y)$ or $d_X(x,y)$ sometimes.
The metric $r$-ball around $p$ is denoted by 
 $B(p,r)$. We sometimes use $B^X(p,r)$ to 
  emphasize the metric ball in $X$.
Let $X$ be a locally compact, complete geodesic space with curvature 
$\le\kappa$. By definition, for each point $p\in X$, there exists a 
positive number $r>0$ with $r\le \pi/2\sqrt{\kappa}$ when $\kappa>0$ 
such that the ball $B(p,r)$ is convex
and having the following properties:
Let $M_\kappa^2$ be the simply connected complete surface of 
constant curvature $\kappa$, called the {\it $\kappa$-plane} in short. 
For any geodesic triangle 
$\triangle xyz$ in $B(p,r)$ with vertices $x,y$ and $z$, we denote by 
$\tilde\triangle xyz$ a {\it comparison triangle} 
in $M_\kappa^2$ having  the same side lengths as $\triangle xyz$. 
Then the natural mapping $ \tilde\triangle xyz \to\triangle xyz$
is non-expanding. A convex domain with this property is 
called a {\it $\CAT(\kappa)$-domain}.
Such a space $X$ with curvature $\le \kappa$ is called a locally 
$\CAT(\kappa)$-space, and $X$ is called a $\CAT(\kappa)$-space
if $X$ itself is a $\CAT(\kappa)$-domain.
 Although all geodesics have constant speed 
by definition, most geodesics 
are assumed to have unit speed unless otherwise stated. 
For arbitrary  $x$ and $y$ in $B(p,r)$, let $\gamma_{x,y}:[0, |x,y|]\to X$ denote a 
unique minimal geodesic joining $x$ to $y$.
We say that a curve is {\it shortest} if its length is minimal 
among all curves joining the endpoints.
The angle between the geodesics $\gamma_{y,x}$ and $\gamma_{y,z}$ is 
denoted by 
$\angle xyz$, and the corresponding angle of $\tilde\triangle xyz$
by $\tilde\angle xyz$.
The space of directions and the tangent cone  of $X$ at $p$ 
are denoted by  $\Sigma_p=\Sigma_p(X)$ and $K_p = K_p(X)$ 
respectively. We shall occasionally use the identification 
$\Sigma_p = \Sigma_p\times \{1\} \subset K_p$.   
We denote by  $\dot\gamma_{x,y}(0)$ or $\uparrow_x^y$, 
 the direction at $x$ defined by $\gamma_{x,y}$.
For every $\xi\in\Sigma_p(X)$, $\gamma_{\xi}$ denotes a geodesic with 
$\dot\gamma_\xi(0)=\xi$.
For a path-connected subset $S\subset X$ and $x,y\in S$, we  denote by 
$\gamma^S_{x,y}$ a shortest curve in $S$ joining 
$x$ to $y$ if it exists.
Occasionally, we identify a geodesic with its image,
and write as $x\in \gamma$ for instance.
The length metric of $S$ induced from $X$ is denoted by $d_S$
or $|\,\,, \,\,|_S$.
 
For a closed subset $A$ of $X$ and for an accumulation point $p$ of $A$,
the set of all directions $\xi\in\Sigma_p(X)$ such that there is a 
sequence $a_i$ in $A\setminus \{ p\}$ satisfying $a_i\to p$ and 
$\dot\gamma_{p,a_i}(0) \to \xi$ is denoted by $\Sigma_p(A)$ and called 
the space of directions of $A$ at $p$.

The upper semi-continuity of angle is fundamental in the 
geometry of spaces with curvature bounded above.

\begin{lem} \label{lem:limit}
Suppose that sequences $p_i$, $q_i$ and
$r_i$ converge to $p$, $q$ and $r$ respectively in a
$\CAT(\kappa)$-domain. 
Then we have $\limsup_{i\to\infty}\angle p_iq_ir_i \le \angle pqr$.
\end{lem}

Next we shortly discuss the connectivity of a small neighborhood of a 
given point in $X$.
For each point $p\in X$, the set of components of $\Sigma_p$ are in 
one to one correspondence with the set of components of 
$B(p,r)\setminus\{ p\}$ if $B(p,r)$ is a $\CAT(\kappa)$-domain  (cf.\cite{Kr:local}).
We call the number of components of $\Sigma_p(X)$ {\it the order}
of $p$.

Now we state the gluing theorem proved by \cite{Rsh:theory}, which is convenient to construct 
spaces with curvature bounded above. The proof is also found in 
\cite[p.347]{bridson-haefliger}.

\begin{thm} \label{thm:glue}
Let $D_i$, $i=1,2$,  be a closed convex subset in an 
Alexandrov space $X_i$ with curvature $\le \kappa$. 
If there is an isometry $f:D_1\to D_2$, then the 
identification space $X_1\cup_f X_2$ is an 
Alexandrov space with curvature $\le \kappa$
with respect to the natural length metric.
\end{thm} 

%%%% basic lemma on a small ball around $p$

From now, we assume $X$ to be {\it geodesically complete}.
That is, every geodesic segment in $X$ can be extended to 
a geodesic defined on $\mathbb R$.

 The following lemma follows from \cite[Corollary 13.3]{NL:geodesically}.

\begin{lem}\label{lem:conn}
For a point $p\in X$, suppose that $\Sigma_p(X)$ has no isolated points.
Then there exists a positive number $r$ such that every point $x$ in 
$B(p,r)\setminus\{ p\}$ has order one.
\end{lem}

 Let $d_p$ denote the distance function from $p$.
For every $x\neq p$, let us denote by $(\nabla d_p)(x)$
the set of all directions $\xi\in\Sigma_x(X)$ such that 
$\angle(\xi, \up_x^p)=\pi$.  For simplicity, we set
$-(\nabla d_p)(x) = \up_x^p$.
The following lemma, which describes local geometry around a given point,
is basic in our study of local structure of surfaces with 
curvature bounded above. 

We denote by $\tau_p(\epsilon_1,\ldots, \epsilon_k)$ a function 
depending on $p$ and $\epsilon_1,\ldots, \epsilon_k$ satisfying 
$\lim_{\epsilon_1,\ldots, \epsilon_k\to 0}\tau_p(\epsilon_1,\ldots, \epsilon_k)=0$.

\begin{lem} \label{lem:base}
For every $p\in X$, there exists a positive number $r_0$ such that for every
$r$ with $0<r\le r_0$, $B(p,r)$ satisfies the following $:$
\begin{enumerate}
 \item $\diam ((\nabla d_p)(x)) < \tau_p(r)$
   for every $x\in B(p,r)\setminus\{ p\}\,;$
 \item For any two geodesics $\gamma_1$ and $\gamma_2$ starting at $p$
   with angle $\theta$, 
   and for every $s\in [0,r]$, the geodesic $\sigma_s(t)$ joining 
   $\gamma_1(s)$ to $\gamma_2(s)$ satisfies that 
$$
  |\angle(-(\nabla d_p)(\sigma_s(t)), \dot\sigma_s(t)) - \pi/2| < \tau_p(\theta,s).
$$
\end{enumerate}
\end{lem}
\begin{proof}
(1) is due to \cite{OT:riem} (see also \cite[Prop.7.3]{NL:geodesically}).
(2) easily follows from (1), and hence the proof is omitted.
\end{proof}
\pmed   
The following lemma is fundamental, and plays an important role
as in the case of Alexandrov space with curvature bounded below
(\cite{BGP}). For the proof, see \cite[Lemma 3.6]{Ng:volume}.

\begin{lem}[Jack Lemma] \label{lem:jack}
For every $p\in X$, there exists a positive number $r_0$ such that if 
$x\neq y \in B(p,r_0)$ and $q$ satisfy that 
$\tilde\angle pxq > \pi-\e$ and $|x,y|< \e\min \{ |p,x|, |q,x|\}$, then we have
\[
          |\angle pxy -\tilde\angle pxy| <\tau_p(|p,x|, \e).
\]
\end{lem} 
\pmed
%%% convergence theorem 

In the study of spaces of curvature bounded below, the theory of 
the Gromov-Hausdorff convergence has been useful.  
We apply  it in our case of curvature bounded above.

We denote by $\mathcal H^n$ the $n$-dimensional Hausdorff measure,
and set $\omega_n :=\mathcal H^n(S^n(1))$, where $S^n(1)$ is the unit $n$-sphere.

\begin{thm}[\cite{Ng:volume}, Compare \cite{Ym:conv}] \label{thm:conv}
For each positive integer $n$, there is a positive number $\epsilon_n$
satisfying the following:
Let $X_i$, $i=1,2,\ldots$, and $X$ be $n$-dimensional locally compact, 
geodesically complete, pointed Alexandrov spaces with curvature $\le\kappa$, and 
suppose that a compact $\CAT(\kappa)$-domain $U_i$ of $X_i$  
converges to a compact $\CAT(\kappa)$-domain $U$ of $X$ with respect to 
the Gromov-Hausdorff distance. Then for every compact domain 
$V$ in $\interior U$ satisfying $\mathcal H^{n-1}(\Sigma_x(X))<\omega_{n-1} +\epsilon$
with $\e\le \epsilon_n$ for all $x\in V$, 
there are a compact domain $V_i$ in $\interior U_i$ and a $\tau(\e, 1/i)$-almost isometry $\varphi_i:V_i\to V$ in the sense that 
\[
       \left|\frac{|\varphi_i(x),\varphi_i(y)|}{|x,y|}-1\right| < \tau(\e,1/i),
\]
for all $x,y\in V_i$.
\end{thm}

\begin{lem}\label{lem:comparison}
For every $p\in X$ and  arbitrary $x,y\in B(p,r)$ we have
\begin{align}\label{eq:diff-comarison-angle}
   \tilde\angle xpy -\angle xpy <\tau_p(r), \quad
   \tilde\angle pxy -\angle pxy <\tau_p(r),\quad
   \tilde\angle pyx -\angle pyx <\tau_p(r).
\end{align}
\end{lem}
\begin{proof} For the proof, it suffices to show that 
for every $\e>0$ there is an $r$ such that
for arbitrary $x,y\in B(p,r)$ we have
\eqref{eq:diff-comarison-angle} for $\e$ in place of 
$\tau_p(r)$.
Fix a constant $C\ge1$.

\pmed\n
Case 1). $C^{-1}\le |p,x|/|p,y| \le C$.
\pmed
We show \eqref{eq:diff-comarison-angle} for 
$\e=\tau_{p,C}(r)$, where $\tau_{p,C}(\,\cdot\,)$ is a function depending on $p$, $C$ with $\lim_{r\to 0}\tau_{p,C}(r) =0$. 
Suppose $|x,y|< \zeta |p,x|$ for $\zeta>0$.
Then 
Lemma \ref{lem:jack} implies 
\[
  \tilde\angle pxy - \angle pxy <\tau_p(r,\zeta),\quad
   \tilde\angle pyx - \angle pyx <\tau_p(r,\zeta).
\]   
Since  $\tilde\angle xpy<\tau(\zeta)$, 
\eqref{eq:diff-comarison-angle} holds when 
$r\le r_0$ and $\zeta\le\zeta_0$ for some
$r_0, \zeta_0$.

Next, suppose $|x,y|\ge \zeta_0 |p,x|$.
We proceed by contradiction. 
Suppose the lemma does not hold in this case.
Then there are $x_n,y_n\to p$ with $|x_n,y_n|\ge\zeta_0|p,x_n|$. Choose $z_n$ such that $x_n\in\gamma_{p,z_n}$
and $|z_n,x_n|=|p,x_n|$.
Consider the rescaling limit
\[
   \left(\frac{1}{|p,x_n|}X,p\right) \to (K_p,o_p).
\]
Since 
$\lim_{n\to\infty} (\angle px_n y_n + \angle z_nx_n y_n)=\pi$,
we have 
\[
\lim_{n\to\infty} \angle px_ny_n = \angle o_p x_\infty y_\infty=\lim_{n\to\infty} \tilde\angle px_ny_n,
\]
where $x_\infty$  and  $y_\infty$  are  the limits of
$x_n$ and $y_n$. Similarly, we have 
$
\lim_{n\to\infty} \angle py_nx_n = \lim_{n\to\infty} \tilde\angle py_n x_n.
$
Since obviously we have 
$$
\lim_{n\to\infty} \angle x_n p y_n =
\angle x_\infty o_p y_\infty= \lim_{n\to\infty} \tilde\angle x_n py_n,
$$
we derive a contradiction.

\pmed\n
Case 2). $|p,x|< C^{-1}|p,y|$.
\pmed
We show \eqref{eq:diff-comarison-angle} for 
$\e=\tau_{p}(r)+\tau(C^{-1})$. 
First note that
\begin{align} \label{eq:comp-2}
   \tilde\angle pyx<\tau(C^{-1}). 
\end{align}
Thus considering large $C$, we only have to consider
the angles at $p$ and $x$.
Take $z$ with $x\in\gamma_{p,z}$ and $|z,p|=|y,p|$.
Then we have 
\begin{align} \label{eq:comp-1}
    \tilde\angle ypz \ge \tilde\angle ypx\ge \angle ypx.
\end{align}
From Case 1), we have 
\begin{align} \label{eq:comp0}
\tilde\angle ypz- \angle ypx<\tau_{p,1}(r).
\end{align}
Combining \eqref{eq:comp-1} and   
 \eqref{eq:comp0}, we certainly have 
\begin{align} \label{eq:comp1}
  \tilde\angle ypx- \angle ypx<\tau_{p,1}(r).
\end{align}
From $|p,x|< C^{-1}|p,y|$, we have 
\begin{align} \label{eq:comp2}
   |\tilde\angle xyz -\tilde\angle pyz|<\tau(C^{-1}), \quad
  |\tilde\angle xzy -\tilde\angle pzy|<\tau(C^{-1}).
\end{align}
From \eqref{eq:comp1} and \eqref{eq:comp0}, we have 
\[
|\tilde\angle xpy -\tilde\angle zpy|<\tau_{p,1}(r).
\]
From \eqref{eq:comp-2} and the first inequality in  \eqref{eq:comp2}, we have 
\[
|\tilde\angle pyx + \tilde\angle xyz  -\tilde\angle pyz|<\tau(C^{-1}).
\]
Now  consider the quadrangle $\tilde p\tilde x\tilde z\tilde y$ on $M^2_{\kappa}$ which is the union of 
the triangles $\tilde\triangle pxy$ and $\tilde\triangle xyz$ glued along the edge $\tilde x\tilde y$ corresponding to $xy$.
We estimate the deviation of the angle of 
the quadrangle $\tilde p\tilde x\tilde z\tilde y$ at 
$\tilde x$ from $\pi$.
Combining the last two inequalities and the second inequality in \eqref{eq:comp2}, we have 
\[
|\tilde\angle pxy + \tilde\angle yxz  -\pi |<\tau_{p,1}(r)+\tau(C^{-1}).
\]
Since 
\[
|\angle pxy + \angle yxz  - \pi|<\tau_p(r),
\]
the last two inequalities yield
$\tilde\angle pxy -\angle pxy<\tau_p(r)+\tau(C^{-1})$ as required.
This completes the proof.
\end{proof}

A point $p$ in $X$ is called a {\it topological singular point} of $X$
if any neighborhood of $p$ is not homeomorphic to a disk,
and the set of all topological singular point of $X$ is denoted 
by $\mathcal S(X)$.
It is  proved in \cite{NL:geodesically} that if $\dim X=n$, then $\dim_H\mathcal S(X)\le n-1$. In particular $X\setminus\mathcal S(X)$  
has full measure with respect to 
$\mathcal H^n$ (\cite{OT:riem}).

\pmed
\n
{\bf Two-dimensional case.} 
By a result of Otsu-Tanoue \cite{OT:riem}, the Hausdorff dimension
of  every relatively compact open domain of $X$ is an integer.
See \cite{NL:geodesically} for a different proof. 
It is also known that $\Sigma_p(X)$ is a compact geodesically complete
$\CAT(1)$-space for every $p\in X$.

%%%%%  1-dimensional case %%%%%%%%
Obviously, if $X$ is $1$-dimensional, then 
it is a locally finite graph without endpoints.
Now we assume $X$ has dimension $2$. Then  any component $\Sigma$
of $\Sigma_p$ has dimension $\le 1$.
If $\dim\Sigma=1$, then $\Sigma$ has the structure of a finite graph without endpoints.
Furthermore $\Sigma$ 
is a so called {\it  $\CAT(1)$-graph} without endpoints in the sense that 
each simple closed curve in $\Sigma$ has length at least $2\pi$.
If $\dim\Sigma=0$, then $\Sigma$ is a point and the component of 
$B(p,r)\setminus\{ p\}$ corresponding to $\Sigma$ is an arc for any small enough $r$.
%%%%%%
Thus, a small neighborhood of any point $p\in X$
is the gluing at $p$ of several purely $2$-dimensional spaces
with all links conected graphs and a ball around the vertex
in the cone over finitely many points.
%%%%%%
Therefore the study of local structure around $p$ reduces 
to the case when $\Sigma_p$ is a connected $\CAT(1)$-graph without endpoints.

\begin{lem} \label{lem:disk}
A neighborhood of $p\in X$ is homeomorphic to 
a two-dimensional disk if and only if $\Sigma_p(X)$ is a circle. 
\end{lem}

\begin{proof}  
This follows from \cite[Prop.3.1,  Remark 3.4]{Ng:sphere}. 
\end{proof}

\begin{lem} \label{lem:top-vert1}
Let $p\in\mathcal S(X)$. Then
$\Sigma_p(\mathcal S(X))$ coincides with the set $V(\Sigma_p(X))$ of all 
vertices of the graph $\Sigma_p(X)$.
 \end{lem} 

\begin{proof}
For every $v\in \Sigma_p(\mathcal S(X))$, take a sequence
$x_i$ in $\ca S(X)$ converging to $p$ such that 
$\lim_{i\to\infty}\angle(\dot\gamma_{p,x_i}(0),v)=0$. If $v$ is not a vertex of 
$\Sigma_p(X)$, choose $\epsilon>0$ 
such that the $\epsilon$-neighborhood of $v$ contains no vertices of
$\Sigma_p(X)$. Let $\delta_i:=|x_i,p|$. Theorem \ref{thm:conv}
applied to the convergence $(\frac{1}{\delta_i} X, x_i)\to (K_p(X), v)$
yields that a small neighborhood of $x_i$ is almost isometric to 
a neighborhood in $\R^2$. This is a contradiction.
%%%%%%%

Conversely, suppose there is $v\in V(\Sigma_p(X))$ that is not 
contained in $\Sigma_p(\mathcal S(X))$.
Choose  $\e_0>0$ and $\delta_0>0$ such that the cone neighborhood 
\begin{align} \label{eq:cone-neighborhood}
  C(v,\delta_0,\e_0):=\{ x\,|\,
      \angle(\uparrow_p^x,v)\le\delta_0, |p,x|\le \e_0\}
\end{align}
is included in 
$\mathcal R(X)$.
Take three distinct directions $\xi_1,\xi_2,\xi_3\in\Sigma_p(X)$
having  angle $\delta_0/2$ with $v$, and set 
$x_i(\e):=\gamma_{\xi_i}(\e)$, where $\e\le\e_0$, $1\le i\le 3$.
Note that the geodesic $[x_1(\e),x_2(\e)]$ converges 
to the geodesic $[\xi_1,\xi_2]$ in $K_p(X)$ under the 
convergence 
$$
\left( \frac{1}{\e} X, p\right)  \to (K_p(X), o_p)
\qquad \text{as $\e\to 0$}.
$$
Let $y(\e)$ be a nearest point of $[x_1(\e),x_2(\e)]$ from 
$x_3(\e)$. Since $y(\e)\in\mathcal R(X)$, 
$\Sigma_{y(\e)}(X)$ must be a circle of length almost
equal to $2\pi$ with $\angle(\nabla d_p(y(\e)), -\nabla d_p(y(\e)))=\pi$. 
However, Lemma \ref{lem:base} shows that the angle
$\angle(\nabla d_p(y(\e)), \eta_i(\e))$ is almost 
$\pi/2$, where $\eta_i:=\uparrow_{y(\e)}^{x_i(\e)}$, $1\le i\le 3$, which implies $\angle(\eta_1(\e),\eta_2(\e))$ is almost 
equal to $0$. This is a contradiction since 
$\angle(\eta_1(\e),\eta_2(\e))=\pi$. 
\end{proof}

\begin{rem} \upshape
In place of the above geometric argument of the second half of the proof of Lemma \ref{lem:top-vert1},
we can also use more general topological result in \cite[Theorem 2.1] {GPW}.
\end{rem}

\begin{lem} \label{lem:near-vert}
Let $p\in\mathcal S(X)$.
For any $x\in \mathcal S(X)\cap (B(p,r)\setminus\{ p\})$, $V(\Sigma_x(X))$ is contained in the $\tau_p(r)$-neighborhood of $\{ (-\nabla d_p)(x),  (\nabla d_p)(x) \}$.

Therefore there is a positive integer $m\ge 3$ such that 
the Gromov-Hausdorff distance between $\Sigma_x(X)$ 
and the spherical suspension over $m$ points is less than
$\tau_p(r)$.
\end{lem}

\begin{proof}
This follows from \cite[Proposition 6.6, Corollary 13.3]{NL:geodesically}.
\end{proof}

As an immediate consequence of Lemmas \ref{lem:top-vert1} and \ref{lem:near-vert},
we have 

\begin{cor} \label{cor:vert}
Let $p\in\mathcal S(X)$.
For every $x\in\mathcal S(X)\cap (B(p,r)\setminus\{ p\})$,
$\Sigma_x(\mathcal S(X))$ is contained in a 
$\tau_p(r)$-neighborhood of $\{ (-\nabla d_p)(x),  (\nabla d_p)(x)\}$.
\end{cor}

\psmall
Finally in this subsection, we shortly discuss the cardinality of singular points
in a two-dimensional manifold $X$ with curvature $\le\kappa$.
Let $\e>0$. We say that $x\in X$ is an {\it $\e$-singular} point if 
$L(\Sigma_x(X))\ge 2\pi+\e$.   
We also say that $x$ is a {\it singular} point if it is $\e$-singular for some $\e>0$. 

\begin{lem} $($cf. \cite{AZ:bddcurv}, \cite[Prop.4.5]{BurBuy:upperII}$)$  \label{lem:sing-mfd}
For a domain  $D$ of a two-dimensional manifold $X$ 
with curvature $\le\kappa$, the set of all singular points contained in $D$
is at most countable.
\end{lem}
\begin{proof}
By Lemma \ref{lem:base}(1), the set of all $\e$-singular points contained in a bounded set is finite for every $\e>0$, which immediately yields the conclusion of the lemma.
\end{proof}
\psmall

\setcounter{equation}{0}

\section{Basic properties of ruled surfaces}\label{sec:ruled-alex}

We recall the notion of ruled surfaces in metric spaces
introduced by Alexandrov \cite{alexandrov-ruled}.
The metric on a ruled surface discussed in \cite{alexandrov-ruled}  is the pull-back metric defined below,
although an explicit definition was not given in \cite{alexandrov-ruled}.
See also Remark \ref{rem:def-ruled}.
In this section, we provide  some fundamental properties
of the pull-back metric, most of which are not contained in \cite{alexandrov-ruled}.
These are used in the proof of Alexandrov's result
(Theorem \ref{thm:alex-ruled2}), which is  presented in Appendix \ref{sec:append}. 
There are related results in \cite[Section 2]{PS}.
\psmall

For our purpose, it is sufficient to consider ruled surfaces 
in spaces with curvature bounded above.
Throughout this section, let $X$ be a locally compact, complete geodesic space with curvature 
$\le\kappa$ with metric $d_X$, where we do not need the dimension restriction, nor geodesic completeness.

We fix a rectangle  $R := [0,\ell] \times [0,1]$ in this section.
%%%%%%%%%%
\pmed\n
{\bf  Ruled surfaces.} 

\begin{defn} \label{defn:ruled} \upshape
A continuous map $\sigma \colon R \to X$
is called a \emph{ruled surface} in $X$
if
\begin{enumerate}
\item
for every $s \in [0,\ell]$
the $t$-curve $\lambda_s \colon [0,1] \to X$ of $\sigma$ defined as
$\lambda_s(t) := \sigma(s,t)$
is a minimal geodesic in $X$
from $\sigma(s,0)$ to $\sigma(s,1)$;
\item  for some continuous function $\xi:[0,\ell]\to [
0,1]$,   the curve $\Sigma(s)=\sigma(s,\xi(s))$ $\,(0\le s\le \ell)$
is rectifiable with respect to $d_X$.
\end{enumerate}

As usual,
the subset $S$ of $X$ defined as $S := \sigma(R)$
is also called a 
\emph{ruled surface in $X$}.
For each $s\in [0,\ell]$,
the minimal geodesic $\lambda_s \colon [0,1] \to X$ is called a 
\emph{generator} of $\sigma$, or a {\it ruling geodesic} of $\sigma$.
\end{defn}

For each $t \in [0,1]$,
the curve $\sigma_t \colon [0,\ell] \to X$ is called a 
\emph{directrix of $\sigma$ at $t$}.

\pmed\n
{\bf Pull-back metrics and induced metrics on ruled surfaces}

Let $\sigma:R\to X$ be a ruled surface in $X$ defined as above.
%%%%%%%%
We denote by ${\rm Sing}(\sigma)$
(resp. by  ${\rm Reg}(\sigma)$) the set of all 
$s\in [0, \ell]$ such that $\lambda_s$ are constant
(resp. nonconstant).
For $s\in [0,\ell]$, we set 
$$
I_s:=\{ s\}\times [0,1]\subset R.
$$

\begin{defn} \upshape
We say that a  (not necessarily continuous) map
$c:[a,b]\to R$ is {\it monotone} if 
\begin{itemize}
 \item $p_1\circ c$ is monotone 
non-decreasing or  monotone non-increasing  where $p_1:R\to [0,\ell]$ is the projection to the first factor\,$;$
 \item  if $p_1\circ c(t)=p_1\circ c(t')=s$ with $t<t'$, then  $p_2\circ c|_{[t,t']}$ is monotone, where 
 $p_2 : R \to [0,1]$ is the projection to the second factor.
\end{itemize}
Similarly, $c$ is said to be {\it strictly monotone} if $p_1\circ c$ is strictly monotone. 

We say that a monotone  map $c:[a,b]\to R$
is  a {\it quasicontinuous curve}  if the following hold:
\begin{enumerate}
\item $p_1\circ c([a,b])$ is a closed interval\,$;$
\item $c$ is continuous on the set of all $t$ with 
$p_1\circ c(t) \in {\rm Reg}(\sigma)\cup {\rm int} \,{\rm Sing}(\sigma)$.
\end{enumerate}
\end{defn}
%%%%%%%%%
We define the {\it pull-back metric} $e_{\sigma}$  on $R$ induced from $\sigma$  as
\beq \label{eq:pull-back}
e_{\sigma} (u, u') := \inf_c \, L(\sigma \circ c), 
\eeq
where $c$ runs over all quasicontinuous curves in $R$ from $u$ to $u'$,
and $L$ denotes the length of curves with respect to $d_X$.
Note that the metric $e_{\sigma}$ is certainly finite
since our ruled surface $\sigma$ has the rectifiable curve $\Sigma$.

We denote by $R_*$ the quotient metric space 
$$
       (R_*, e_\sigma):= (R, e_\sigma)/\{ e_\sigma=0\}.
$$ 
Let $\pi:R\to R_*$ be the projection.

\begin{ex} \label{ex:almost-conti}
Let $\sigma_k:[0,1/\pi]\to\R^2$ \,$(k=0,1)$ be the curve defined by
\[
      \sigma_0(s)=\left(s, -\left|s\cos \frac{1}{s}\right| \right),\qquad
       \sigma_1(s)=\left(s, \left|s\sin \frac{1}{s}\right|\right).
\]
For $R := [0,1/\pi] \times [0,1]$, define the ruled surface
$\sigma \colon R \to \R^2$ as in the above definition,
where we  have ${\rm Sing}(\sigma)=\{ 0\}$.
For $u=(0,0)$, $u'=(1/\pi, 0)$, consider the map 
$c:[0,1/\pi]\to R$ by 
\[
    c(s)=\begin{cases}
                 \sigma^{-1}(s,0),    &\quad  (0<s\le 1/\pi)  \\
                   \quad (0,0),               & \quad (s=0).
          \end{cases}
\]
Since $c$ oscillates infinitely many times near 
$\{ 0\}\times [0,1]$,
it is quasicontinuous, but realizing the distance $e_\sigma(u,u')$. 
\end{ex}

Remark that there is no continuous curve realizing 
$e_\sigma(u,u')$ in Example \ref{ex:almost-conti}.
Note also that $\pi\circ c$ is always continuous for 
every quasicontinuous curve $c$.
These are the reasons why we employ the notion of 
quasicontinuous curves in the definition \eqref{eq:pull-back} of the pull-back metric $e_\sigma$.

%%%%%%%%%%%%%%
\pmed
Obviously, $e_\sigma(u,u')=0$ implies $\sigma(u)=\sigma(u')$.
Therefore,  we can define 
a  map $\sigma_*:R_*\to X$ such that $\sigma=\sigma_*\circ\pi$. Note that 
 $\sigma_*:R_*\to X$  is continuous.
 The properties  of the projection $\pi: R \to R_*$ depend
 on those of the end $s$-curves $\sigma_0$ and $\sigma_1$.
 If $\sigma_0$ and $\sigma_1$ are Lipschitz continuous, then 
 so is $\sigma$, and hence $\pi: R \to R_*$ is continuous.
  However,  in the general case, $\pi: R \to R_*$ is not necessarily continuous (see Example \ref{ex:Koch} below).

\begin{rem} \label{rem:def-ruled} Here are some remarks on
the relation between the conditions of ruled surfaces given in 
\cite{alexandrov-ruled} and ours.
\benum
 \item 
Some ruling geodesics $\lambda_s$ of $\sigma$  may be
constant geodesics for all $s$ in an interval of $[0,\ell]$. 
This case was excluded in \cite{alexandrov-ruled}
as the conditions of ruled surfaces defined there $\,;$
 \item  More restricted property, the existence of continuous arc
 in the preimage of  any point of $R_*$ by $\pi$,  than the existence of quasicontinuous curve given in Corollary \ref{cor:esigma=0}
was assumed in \cite{alexandrov-ruled}
as one of the conditions of the
metric on the ruled surface under consideration.
\eenum
\end{rem}

 \begin{ex} \label{ex:Koch} \upshape
Let $\sigma_1(s)$ \,$(0\le s\le\ell)$ be a 
 continuous parametrization of a Koch curve on the unit sphere
 $\mathbb S^2(1)\subset\mathbb R^3$.
Letting $\sigma_0(s)=O$, we define the ruled surface
$\sigma:[0,\ell]\times [0,1]\to \mathbb R^3$ by 
$\sigma(s,t)=t\sigma_1(s)$. Note that 
\[
    e_\sigma((s,t), (s',t'))= \begin{cases}
               |t-t'|   & \text{if  \, $s=s'$}  \\
             \,\,  t+t'     & \text{if \, $s\neq s'$}.
       \end{cases}
\]
Note that $\pi:R\to R_*$ is continuous only at $\{ t=0\}$.
\end{ex}  
\psmall 
 
For $s\in [0,\ell]$, we set 
$$
   I_s^*:=\pi(I_s).
$$ 
%%%%%%%%%%

For a continuous curve $c_*:[a_0,b_0]\to R_*$
(resp. $\gamma:[a_0,b_0]\to S$), we simply say that
a quasicontinuous curve $c:[a,b] \to R$ is a {\it lift}
of $c_*$ (resp. of $\gamma$) if $c_*=\pi\circ c$ 
(resp. $\gamma=\sigma\circ c$) up to monotone
parametrization. 
%%%%%%%%
\psmall
From now, we fix arbitrary $u, u'\in R$ with  $u=(s_0,t_0),u'=(s_0',t_0')$ and $s_0 < s_0'$.
Take a sequence of quasicontinuous curves $c_n:[0,1]\to R$ from $u$ to $u'$ such that 
$e_{\sigma} (u, u')=\lim_{n\to\infty} L(\sigma\circ c_n)$, where 
we may assume that $c_n$ is monotone.
By the Arzela-Ascoli theorem, passing to a subsequence
we may assume that a Lipschitz parametrization $\gamma_n$ of $\sigma\circ c_n$ converges to a Lipschitz curve $\gamma$ in $S$ from $\sigma(u)$ to $\sigma(u')$. 
Note 
\begin{align}\label{eq:L(gamma)}
L(\gamma) \le e_{\sigma} (u, u').
\end{align}

We set
\[ \text{ $J:= [s_0, s_0']$,
 $J_{\rm reg}:=J\cap {\rm Reg}(\sigma)$,
 $J_{\rm sing}:=J\cap {\rm Sing}(\sigma)$.}
 \]
 In the following proposition, we show the equality in \eqref{eq:L(gamma)}.

\begin{prop}\label{prop:e-sigma} Under the above situation, there is a lift $c$ of $\gamma$ in $R$ from $u$ to $u'$. 

In particular,  $\pi\circ c$ provides a (continuous)
shortest curve $c_*$  in $R_*$ from  $\pi(u)$ to $\pi(u')$, and we have 
$$
      L(c_*)=L(\gamma)=e_\sigma(\pi(u),\pi(u')).
$$
\end{prop}

\begin{ex}\label{ex:trivial} \upshape
Let $\gamma:[0,1]\to X$ be a minimal geodesic between distinct two points in a $\CAT(\kappa)$-space.
Consider the ruled surface $\sigma:[0,1]\times [0,1]\to X$
defined as $\sigma(s,t)=\gamma(t)$.
Then $e_ \sigma((0,0), (1,1))=L(\gamma)$.
Note that any curve $c(t)=(x(t), y(t))$ such that $x(t)$ and
$y(t)$ are monotone from $0$ to $1$
is a lift of $\gamma$ from $(0,0)$ to $(1,1)$.
\end{ex}

The above simple example suggests that in 
Proposition \ref{prop:e-sigma},
one cannot construct a lift of the limit $\gamma$
only from $\gamma$, and one needs to take a subsequence of $c_n$ properly to obtain a limit,
which is expected as a lift of $\gamma$. 
In the proof of Proposition \ref{prop:e-sigma} below,
we proceed in this way.

\begin{proof}[Proof of Proposition \ref{prop:e-sigma}]
%%%%%%%%%
We show that %roughly speaking, 
the monotone curve $c_n$ converges to
a monotone quasicontinuous curve  $c$, up to monotone  parametrization, except on 
${\rm Sing}(\sigma)\times [0,1]$.
By the Arzela-Ascoli theorem, this is obvious if the length of $c_n|_{J_{\rm reg}\times [0,1]}$ is uniformly bounded.
However, when one of the end curves $\sigma_0(s)$ and 
$\sigma_1(s)$  is not recifiable, one can not expect 
that the length of $c_n|_{J_{\rm reg}\times [0,1]}$ is even finite. 

In the argument below, we use the idea of the proof of 
the Arzela-Ascoli theorem taking the monotonicity of $c_n$ into account.
Since each $c_n$ is continous, for any $s\in J$ there is 
$t_n(s)\in [0,1]$ satisfying $c_n(t_n(s))\in I_s$.
Let $J_0$ be a countable dense subset of $J$.
Take a subsequence $\{ m\}\subset \{ n\}$ such that 
$c_m(t_m(s))$ converges to a point $x(s)\in I_s$
for every $s\in J_0$.

Roughly speaking, the limit curve $c$ is defined via
the limit set of the sequence $\{ {\rm Im}(c_m)\}_m$.
For every $s\in J_{\rm reg}$, 
let us consider the subset $E_s\subset I_s$ defined as the set of all points $x\in I_s$ with 
$\lim_{i\to\infty}c_{m_i}(t_i)=x$ for a subsequence $\{ m_i\}\subset \{ m\}$ and $t_i\in [0,1]$.
We set 
\[
    J_{{\rm reg},1}:= \{ s\in J_{\rm reg}\,| \text{ $E_s$ is a single point}\}, \quad  
    J_{{\rm reg},2}:=J_{\rm reg}\setminus  J_{{\rm reg},1}.
\] 
\[
    J_{{\rm reg},1}^0:=J_{{\rm reg},1}\cap J_0\quad  
    J_{{\rm reg},2}^0:=J_{{\rm reg},2}\cap J_0.
\] 
For $s\in J_{{\rm reg},1}$, we define $x(s)$ by  
\[
        E_s =\{ x(s)\}.
\]
Note also that $ J_{{\rm reg},1}^0$ or $J_{{\rm reg},2}^0$
may be empty.
We begin with 

\pmed\n
1)\, $x(s)$ is continuous on $J_{\rm reg,1}$.
\pmed
This is obvious since if $s_i\in J_{\rm reg,1}$ converges
to $s\in J_{\rm reg,1}$, then any limit of $\{ x(s_i)\}$ 
must belong to $E_s=\{ x(s)\}$. 
\pmed\n
2)\, Next we show that $E_s$ is an interval  for every 
$s\in J_{\rm reg,2}$.

For arbitrary $y, y'\in E_s$, choose subsequences $\{ m_i\}$ and $\{ m_{i'}\}$ of $\{ m\}$
such that $c_{m_i}(t_i)\to y$ and $c_{m_{i'}}(t_i')\to y'$ as 
$i, i'\to\infty$ for some $t_i, t_i'\in [0,1]$.
Take $s_j\in J_{\rm reg}^0$ with $s<s_j$ converging to $s$.
Note that $x(s_j)=\lim_{i\to\infty} c_{m_i}(t_{m_i}(s_j))
=\lim_{i'\to\infty} c_{m_{i'}}(t_{m_{i'}}(s_j))$.
Passing to a subsequence, we may assume that 
$x(s_j)$ converges to a point $z\in E_s$ as $j\to\infty$.
As $i,i'\to\infty$ and then $j\to\infty$,
the arcs $c_{m_i}([t_i, t_{m_i}(s_j)])$ and 
$c_{m_{i'}}([t_{i'}, t_{m_{i'}}(s_j)])$ converge to
$[y, z]$ and $[y', z]$ respectively.
Since $[y, z]\cup [y', z]\subset E_s$, we obtain 
$[y,y']\subset E_s$.

\pmed\n
3)\, For  $s_i<s$ (resp. $s_i>s$)  with $s_i\in J_{\rm reg}$,
$s\in J_{\rm reg, 2}$, let $s_i$  converge to $s$.
In the below, we show that $x(s_i)$ converges to an endpoint of $E_s$
(resp. the other endpoint of $E_s$).

 Let $\{ y, y'\}=\pa E_s$.

\pmed\n
(a)\, We assume $s_i<s$. The other case is similar.
Suppose that $x(s_{i_k})$ converges to an interior point
$v$ of $E_s$ as $k\to\infty$,
for a subsequence $\{ i_k\}$ of $\{ i\}$.
We also have subsequences $\{ m_{\ell}\}$ and 
$\{ m_{\ell'}\}$  of $\{ m\}$
such that $c_{m_{\ell}}(t_\ell)\to y$ and 
$c_{m_{\ell'}}(t_{\ell'})\to y'$ for some $t_\ell, t_{\ell'}\in [0,1]$.
As $\ell, \ell'\to \infty$ and then $k\to\infty$, the arcs $c_{m_{\ell}}([t_{m_{\ell}}(s_{i_k}), t_{\ell}])$
and  $c_{m_{\ell'}}([t_{m_{\ell'}}(s_{i_k}), t_{\ell'}])$ converge to 
the subarcs $[v,y]$ and $[v,y']$ respectively.
Now take a sequence $s_\alpha\in J_{\rm reg}^0$ with $s_\alpha>s$ such that $x(s_\alpha)$ converges to a 
point $w\in I_s$ as $\alpha\to\infty$.
Note that
\[
x(s_\alpha)=\lim_{\ell\to\infty} c_{m_{\ell}}(t_{m_{\ell}}(s_\alpha))
     =\lim_{\ell'\to\infty} c_{m_{\ell'}}(t_{m_{\ell'}}(s_\alpha)).
\]
Then we see that 
as $\ell, \ell'\to \infty$ and then $k\to\infty$, 
the arcs $c_{m_{\ell}}([t_{m_{\ell}}(s_{i_k}), t_{m_{\ell}}(s_\alpha)])$
and  $c_{m_{\ell'}}([t_{m_{\ell'}}(s_{i_k}), t_{m_{\ell'}}(s_\alpha)])$ converge to 
the unions $[v,y]\cup [y,w]$ and $[v,y']\cup [y',w]$ respectively. However, considering $\sigma\circ c_{m_\ell}$
or $\sigma\circ c_{m_{\ell'}}$, we have a contradiction 
since $\sigma\circ c_m$ is a sequence  minimizing
$e_{\sigma}(u,u')$.

\pmed\n
(b)\,We show that as $s_\alpha<s$ converges to $s$, then 
$x(s_\alpha)$ converges to a unique endpoint of $E_s$.
Suppose that for subsequences $s_i\to s$ and $s_{i'}\to s$
with $s_i,s_{i'}<s$, $x(s_i)$ (resp. $x(s_{i'})$) converges to
$y$ (resp. to $y'$).
Choose large $i$ and $i'=i'(i)$ with $i'\gg i$. Then 
as $m\to\infty$, the arc $c_m([t_m(s_i), t_m(s_{i'})])$ 
oscillates many times near $E_s$, which implies
$\lim_{m\to\infty} L(\sigma\circ c_m) =\infty$.
This is a contradiction.

\pmed\n
(c)\,We show that as $s_i\to s$, $s_{i'}\to s$ with $s_i<s<s_{i'}$,
if $x(s_i)$ converges to $y$, then $x(s_{i'})$ converges to $y'$. Otherwise, as $m\to\infty$, $i, i'\to\infty$, the arc $c_m([t_m(s_i), t_m(s_{i'})])$ converges to the union 
$[y,y']\cup [y',y]$, which is a contradiction to the 
hypothesis that $\sigma\circ c_m$ is a minimizing sequence.
\pmed

\pmed\n
4)\, 
We show that $J_{\rm reg,2}$ is at most countable,
and 
\[
   \sum_{s\in J_{\rm reg,2}} L(\sigma(E_s))  \le L(\gamma).
\]
For arbitrary finite set $s_1<s_2<\cdots <s_k$ of $J_{\rm reg, 2}$, the argument in 3)-(c) shows that 
some subarcs of $c_m$ are so close to the union
$E_{s_1}\cup \ldots \cup E_{s_k}$ for any large $m$.
Thus,  $\sigma(E_{s_1})\cup \ldots \cup \sigma(E_{s_k})$
is the union of finite subarcs of $\gamma$. Therefore we 
have
\[
      \sum_{i=1}^k L(\sigma(E_s))  < L(\gamma).
\]
The conclusion follows immediately.
\pmed

\pmed\n
5)\,
For $s\in J_{\rm sing}$, let $x(s):=(s, a)\in I_s$  for any 
fixed constant $a\in [0,1]$, for instance.
Let $L_0$ be the total sum of $L(\sigma(E_s))$ for all
$s\in J_{\rm reg,2}$.
 Now we consider the collection $\mathcal C$ 
 consisting of points 
 $\{ x(s)\,|\, s\in J_{\rm sing} \cup J_{\rm reg,1}\}$ and the intervals  $E_s$ for all $s\in J_{\rm reg,2}$.
In view of 1) $\sim$ 4), it is possible to parametrize $\mathcal C$ as a quasicontinuous curve $c:[s_0, s_0'+L_0]\to R$
  from $u$ to $u'$. For the detail, see 2) in the proof of 
  Proposition \ref{cor:e-sigma2}.
   
From construction,  we see that $c$ is a lift of $\gamma$. 

 The second half of the assertion of the proposition is
 immediate. This completes the proof of Proposition \ref{prop:e-sigma}. 
\end{proof}  

As an immediate consequence of Proposition \ref{prop:e-sigma}, we have 

\begin{cor} \label{cor:esigma=0}
If $e_\sigma(u,u')=0$, then there is a strictly monotone quasicontinuous  curve
$c:[0,1]\to R$ joining $u$ to $u'$ such that 
$\pi(c)=\pi(u)=\pi(u')$.

 In particular, if ${\rm Sing}(\sigma)$ is empty,
$\pi^{-1}(\pi(u))$ is a strictly monotone $($continuous$)$ curve.
\end{cor}

The following example shows that it is impossible to 
take a monotone (continuous) curve $c$ in Corollary \ref{cor:esigma=0}
as well as in Proposition \ref{prop:e-sigma}.

\begin{ex}\label{ex:osx} \upshape
Let $X$ be the one point union of two copies, say $\mathbb R^2_0$
and $\mathbb R^2_1$,  of $\mathbb R^2$
at the origin $O$. Let  $\sigma_k(u)$ be straight lines on 
$\mathbb R^2_k$ with $\sigma_k(0)=O$\, $(k=0,1)$. Consider strictly monotone continuous parametrizations
$\sigma_k(\varphi_k(s))$ of $\sigma_k(t)$ with $\varphi_k(0)=0$.
Joining $\sigma_0(\varphi_0(s))$ and $\sigma_1(\varphi_1(s))$
by the minimal geodesics, we define a ruled surface
$\sigma:\mathbb R\times [0,1]\to X$.
Note that ${\rm Sing}(\sigma)=\{ 0\}$.
For each $s\in\mathbb R\setminus \{ 0\}$, let $t(s)\in (0,1)$ be
such that $\lambda_s(t(s))=O$. Thus we have
\[
     \sigma^{-1}(O)=\{ (s, t(s))\,|\, s\in\mathbb R\setminus \{0\}\} \cup I_0.
\]
Now choosing the two parameters $\varphi_0(s)$ and $\varphi_1(s)$ properly,
we can let the function $t(s)$ oscillate as $s\to 0$.
In that case, for arbitrary  $u,u'\in \mathbb R\times [0,1]$ with 
$\sigma(u)=\sigma(u')=O$ and $p_1(u)<0<p_1(u')$, 
there is no continuous curve in $\sigma^{-1}(O)$ joining $u$ and $u'$
but quasicontinuous one.
\end{ex}

Next using the procedure in the proof of Proposition 
\ref{prop:e-sigma}, 
we provide a condition for a curve
$c_*$ in $R_*$ to have a lift $c$ in $R$. 

\begin{defn} \label{defn:s(x)} \upshape
For $x\in R_*$, we set
\begin{align*}
 &s(x):=\{ s\in [0,\ell]\,|\, x\in I_s^*\}=p_1(\pi^{-1}(x)),\\
  &   s_{\min}(x) :=\min s(x), \quad s_{\max}(x) :=\max s(x).
\end{align*}
We write  $s(x)< s(y)$ when $s_{\max}(x) <s_{\min}(y)$.
For a subset $A$ of $R_*$, we also write
$s(A) :=\{ s\in [0,\ell]\,| A\cap I_s^*\neq\emptyset\})=p_1(\pi^{-1}(A))$.
\end{defn}

We need a lemma.

\begin{lem}\label{lem:p1-interval}
For any  continuous curve $c_*:[a,b]\to R_*$,
$s(c_*([a,b]))$ is connected.
\end{lem}
\begin{proof}
Choose $u\in \pi^{-1}(c_*(a))$ and $u'\in \pi^{-1}(c_*(b))$.
We may assume $p_1(u)<p_1(u')$.
Let $\gamma:=\sigma_*\circ c_*$.
Since $\pi^{-1}(c_*([a,b]))=\sigma^{-1}(\gamma([a,b]))$, 
$p_1(\pi^{-1}(c_*([a,b])))$ is closed.
Suppose that $p_1(\pi^{-1}(c_*([a,b])))$ is not connected.
Then there are some $s_-<s_+$ in $[p_1(u),p_1(u')]$ satisfying
$$
\pi^{-1}(c_*([a,b]))\subset [0,s_-]\times [0,1]\cup [s_+,\ell]\times [0,1].
$$
Set $R_{-}:=[0,s_-]\times [0,1]$, 
$R_{+}:=[s_+,\ell]\times [0,1]$.   In view of Corollary \ref{cor:esigma=0},  we may assume that $\pi^{-1}(c_*(a))\subset R_-$ and 
$\pi^{-1}(c_*(b))\subset R_+$.
Let us consider 
\[
      t_-:=\sup \{ t\,|\,\pi^{-1}(c_*([a,t]))\subset R_-\}. 
\]
Note that $\pi^{-1}(c_*(t_-))\subset R_-$.
Take  $t_n>t_-$ with $t_n\to t_-$ such that
$\pi^{-1}(c_*(t_n))\subset R_+$. 
Choose a point 
$x_n\in \pi^{-1}(c_*(t_n))$.
Passing to a subsequence, we may assume that $x_n$ 
converges to a point $x_\infty\in R_+$. 
This is a contradiction since $x_\infty \in \pi^{-1}(c_*(t_-))$.
\end{proof}

%%%%%%%%%%%%%%

\begin{defn} \upshape
For a continuous curve $c_*$ in $R_*$,
we say that a subset $A_*\subset I_s^*$ is {\it 
$c_*$-convex} if $c_*(t), c_*(t')\in A_*$ with $t\le t'$, then
$c_*([t,t'])\subset A_*$.
For a continuous curve $\gamma$ in $S$, the notion of 
{\it $\gamma$-convexity} of  a subset $\Lambda\subset
\lambda_s$  is similarly defined.
\end{defn}

Let $c_*:[a,b]\to (R_*,e_\sigma)$ be a continuous curve
from $\pi(u)$ to $\pi(u')$,  and put $s_0=p_1(u)$,  $s_0'=p_1(u')$.  

For any $s\in [s_0, s_0']$, 
we consider   
\[
    E_s^*:= I_s^*\cap c_*([a,b]),
\]
which is nonempty by Lemma \ref{lem:p1-interval}.

\begin{lem} \label{lem:*-monotone}
For a continuous curve  $c_*:[a,b]\to (R_*,e_\sigma)$
with $s_{\min}(c_*(a))\le s_{\max}(c_*(b))$,
suppose that $E_s^*$ is 
$c_*$-convex for every $s\in p_1(c_*([a,b]))$.
Then we have the monotonicity for all $t<t'$ in $[a,b]$,
\[
       s_{\min}(c_*(t))\le s_{\max}(c_*(t'))
\]
\end{lem}
\begin{proof}
Suppose that there are $t_1<t_2$ such that
$ s_{\min}(c_*(t_1)) > s_{\max}(c_*(t_2))$. 
If $s_{\max}(c_*(b))\ge s_{\min}(c_*(t_1))$, then 
Lemma \ref{lem:p1-interval} shows the existence of
$t_3\in [t_2, b]$ such that $s(c_*(t_3))$ meets
$s(c_*(t_1))$. This contradicts the $c_*$-convexity
of $I_s^*$ for $s\in s(c_*(t_1))\cap s(c_*(t_3))$,
since $c_*(t_1), c_*(t_3)\in I_s^*$ and $c_*(t_2)\notin I_s^*$
.
If $s_{\max}(c_*(b)) <  s_{\min}(c_*(t_1))$, then we have
$s(c_*(b))<s(c_*(t_1))$ with 
$s_{\min}(c_*(a))\le s_{\max}(c_*(b))$. Therefore similarly,  
we have a contradiction.
\end{proof}

\begin{prop}\label{cor:e-sigma2}
Let $c_*:[a,b]\to (R_*,e_\sigma)$ be a continuous curve
with $L(c_*)<\infty$ from $\pi(u)$ to $\pi(u')$.
We assume the following$:$\,For each  
$s\in s(c_*([a,b]))$, 
\begin{enumerate}
\item $E_s^*$ is $c_*$-convex\,$;$
\item the restriction $c_*|_{E_s^*}$ is monotone for every 
  $E_s^*$ that is an interval\,$;$
\end{enumerate}
Then there is a lift of $c_*$ in $R$ from $u$ to $u'$.
\end{prop}

\begin{proof} 
1)\,Since we only need to construct a lift $c$ on ${\rm Reg}(\sigma)\times [0,1]$, we may assume ${\rm Sing}(\sigma)$ is empty.
If $s(c_*(a))$ meets $s(c_*(b))$, then 
$c_*$ is a geodesic subarc of $I_s^*$  for 
$s\in s(c_*(a))\cap s(c_*(b))$, and hence certainly
has a lift in $R$ by  Corollary  \ref{cor:esigma=0}.

Thus we may assume
$s(c_*(a))< s(c_*(b))$.
We denote by $\ca E^*_+$ (resp. $\ca E^*_0$) the collection of all 
$E_s^*$ having positive length (resp. zero length, that is points).
Since $L(c_*)<\infty$,  $\ca E^*_+$ is at most countable, and 
\[
    L_0:=  \sum_{{E_s^*}\in \ca E^*_+} L(E_s^*) \le L(c_*).
\]
For each $E_s^*\in \ca E_0^*$, by Corollary \ref{cor:esigma=0}, 
$\pi^{-1}(E_s^*)$ is a continuous strictly monotone arc, denoted by $c_{E_s^*}$.

For $E_s^*\in\ca E_+^*$ with endpoints $c_*(t), c_*(t')$ 
\,$(t<t')$, from the convexity condition together with
Lemma \ref{lem:*-monotone}, we have 
\begin{align} \label{eq:s-monotone}
         s_{\rm min}(c_*(t))\le s_{\rm max}(c_*(t')).
\end{align}
 Let $a(t), b(t')$ be the endpoints of 
 $E_s:=\pi^{-1}(E_s^*)\cap I_s$ corresponding
 to $c_*(t), c_*(t')$ respectively.
 Let $a_{\rm min}(t)\in \pi^{-1}(c_*(t))$ and 
 $b_{\rm max}(t')\in \pi^{-1}(c_*(t'))$ be such that
 $p_1(a_{\rm min}(t))=s_{\rm min}(c_*(t))$ and 
$p_1(b_{\rm max}(t'))=s_{\rm max}(c_*(t'))$.
Then let us denote by $c_{E_s^*}$ the union of the subarc of $\pi^{-1}(c_*(t))$ from $a_{\rm min}(t)$ to $a(t)$, $E_s^*$ and the subarc of $\pi^{-1}(c_*(t'))$ from $b(t')$ to  $b_{\rm max}(t')$.

Let $\ca E^*$ be the union of the collections 
$\ca E_0^*$ and $\ca E_+^*$.
Note that from construction, the family of $p_1$-images
$\{ p_1(c_{E_s^*})\, |\, E_s\in\ca E^*\}$ is pairwise disjoint,
and all the union coincides with $[s_0,s_0']$. 
In particular, we can define the natural order
on the set $\ca E^*$.

\par\n
2)\,We are now ready to parametrize the union of all those arcs 
$c_{E_s^*}$ \, $(E_s^*\in\ca E^*)$, to construct a lift 
$c:[s_0, s_0'+L_0]\to R$ of $c_*$.
For each $E_s^*\in \ca E^*$, let $\ca E^*_{+}(s)$ denote
the set of all $E_{s'}^*\in \ca E_+^*$ with $E_{s'}^*<E_s^*$.
We set
\[
      \ell(E_s^*):= \sum_{E_{s'}^*\in \ca E^*_{+}(s)}
             L(E_{s'}^*).
\]
%%%%%%%
For  $E_s^*\in \ca E_0^*$, 
let $a, b$ be the endpoints of the arc $c_{E_s^*}$ with
$p_1(a)\le p_1(b)$.
We  parametrize 
$c_{E_s^*}$ on $[\ell(E_s^*)+p_1(a), \ell(E_s^*)+p_1(b)]$
by  the condition
$$
       p_1(c_{E_s^*}(\ell(E_s^*)+t))=t \quad \text{for $t\in [p_1(a), p_1(b)]$}.
$$

For $E_s^*\in\ca E_+^*$ with endpoints $c_*(t), c_*(t')$\,
$(t<t')$, let $a(t), b(t')\in \partial E_s^*$,
$a_{\rm min}(t), b_{\rm max}(t')$ be defined as in the previous paragraph. Then we  parametrize 
$c_{E_s^*}$ on 
$[\ell(E_s^*)+p_1(a_{\rm min}(t)), \ell(E_s^*)+L(E_s^*)+p_1(b_{\rm max}(t))]$
by  the conditions that 
\[
  p_1(c_{E_s^*}(\ell(E_s^*)+t))=t 
\]
for $t\in [p_1(a_{\min}(t)), p_1(a(t))]\cup [L(E_s^*)+p_1(b(t')), L(E_s^*)+p_1(b_{\rm max}(t'))]$  and
\[
c_{E_s^*}(\ell(E_s^*)+p_1(a(t))+t)=E_s^*(t)
\]
for  $t\in[0, L(E_s^*)]$,
where $E_s^*(t)$ is the arc-length parameter
from $a(t)$ to $b(t')$.

Finally we observe the continuity of the family 
 $\{ c_{E_s^*}\,|\, E_s^*\in\ca E^*\}$ in the following sense: Let  $\{ E_{s_i}^*\}\in \ca E_0^*$ be a Cauchy sequence
 in $R_*$ satisfying $E_{s_i}^*<E_s^*$
  (resp.   $E_{s_i}^* >E_s^*$) such that its limit
 meets $E_s^*$. Let 
 $a,b$ be the initial and terminal 
 points of $c_{E_{s}^*}$ respectively.
 Then $c_{E_{s_i}^*}$  converges to $a$ (resp. to $b$). 
 This follows from the conditions
 (1), (2) and \eqref{eq:s-monotone}, and the detail is omitted here.
  
Thus  we can define the curve 
$c:[s_0, s_0+L_0]\to R$ as the union of 
all $c_{E_s^*}$\,$(E_s^*\in \ca E^*)$.
It is easy to see that $c$ is a
continuous and monotone lift of $c_*$.
This completes the proof.
\end{proof}

\begin{rem} \upshape
To consider the problem of lifting a curve $\gamma$ in $S$, we need an
extra condition on $\sigma$ or  $\gamma$,
which will be discussed later in 
Proposition \ref{prop:e-sigma3}. 
\end{rem} 

By  Proposition \ref{cor:e-sigma2},
we immediately have the following.

\begin{prop}\label{prop:e-sigma-shortest}
Let $c_*:[a,b]\to (R_*,e_\sigma)$ be a shortest curve
from $\pi(u)$ to $\pi(u')$.
Then there is a lift $c$ of $c_*$ from $u$ to $u'$.
\end{prop}

%$p_1(b_{\rm max}(t'))=s_{\rm max}(c_*(t'))$.
%Then let us denote by $c_{E_s^*}$ the union of the subarc of $\pi^{-1}(c_*(t))$ from $a_{\rm min}(t)$ to $a(t)$, $E_s$ and the subarc of $\pi^{-1}(c_*(t'))$ from $b(t')$ to  $b_{\rm max}(t')$.
%Let $\ca E^*$ be the union of the collections 
%$\ca E_0^*$ and $\ca E_+^*$.
%Note that from construction, the family of $p_1$-images
%$\{ p_1(c_{E_s^*})\, |\, E_s\in\ca E^*\}$ is pairwise disjoint,
%and all the union coincides with $[s_0,s_0']$. 
%In particular, we can define the natural order
%on the set $\{ c_{E_s^*}\,|\, E_s^*\in\ca E^*\}$.
%
%We are now ready to parametrize the union of all the arcs 
%$c_{E_s^*}$ \, $(E_s^*\in\ca E^*)$, to construct a lift 

%

In \cite[Theorem 2]{alexandrov-ruled}, Alexandrov proved the following result,
which plays a crucial role in the present paper.

\begin{thm} [\cite{alexandrov-ruled}] \label{thm:alex-ruled2}
Let $S$ be a ruled surface in a $\CAT(\kappa)$-space
$X$ with parametrization $\sigma \colon R \to X$.
Then $(R_*,e_\sigma)$ is a $\CAT(\kappa)$-space.
\end{thm}

The proof  of  Theorem \ref{thm:alex-ruled2}
is differed  to Appendix \ref{sec:append}.

\psmall
One might expect to define the induced ``metric''
$d_{\sigma}$ on $S$ along $\sigma$  as
\[
d_{\sigma}(x, y) := 
\inf \left\{ \, e_{\sigma} (u, v) \mid 
\text{$\sigma(u) = x$ and $\sigma(v) = y$} \, \right\}.
\]
However, $d_{\sigma}$ does not neessarily satisfy the triangle inequality.
See Remark \ref{rem:non-disk:non-metric}.
Even if $(S, d_\sigma)$ becomes a metric space, 
in certain cases, it could be far from the notion of ``induced metric'', as described in the following example.

\begin{ex}   \label{ex:non-induced-metric} \upshape
Let us consider the following curve  $\alpha:[0,5\pi] \to \mathbb C$ on $\mathbb C=\mathbb R^2$ defined as 
\begin{equation*}
\alpha(s)=
 \begin{cases}
 \hspace{0.2cm} 
   e^{\sqrt{-1} s} & \quad (0\le s\le \pi/2), \\
     (0,2s/\pi)       & \quad (\pi/2\le s\le \pi),\\
    (0, 4- 2s/\pi)   & \quad (\pi\le s\le3\pi/2),  \\
 \hspace{0.2cm}  
  e^{\sqrt{-1}(s-\pi)}  & \quad (3\pi/2\le s\le 5\pi).   
\end{cases}
\end{equation*}
We define the ruled surface $\sigma:[0,5\pi]\times [0,1]\to\mathbb R^3$ by
$\sigma(s,t)=(\alpha(s), t)$.
In this case, $d_\sigma$ is a distance on the image $S$ of 
$\sigma$. Actually $d_\sigma$ coincides with the interior metric 
of $S$ defined in Definition \ref{defn:int-metric}.

On the other hand, if we consider the restriction $\sigma'$ of $\sigma$ to 
$[0,3\pi]\times [0,1]$,
then $d_{\sigma'}$ is not the distance on the ruled surface $S'$ defined by $\sigma'$.
\end{ex}

\begin{lem}\label{lem:non-metric}
Suppose that we have for all $u,v\in R$, 
\beq \label{eq:sigma=uv}
      \sigma(u)=\sigma(v)  \iff e_\sigma(u,v)=0.
\eeq
Then $(S, d_\sigma)$ is a metric space, and $\sigma_*:(R_*,e_\sigma)\to (S, d_\sigma)$ 
is an isometry.
\end{lem}
\begin{proof}
First note that  $e_\sigma(u,u')=0$ implies $\sigma(u)=\sigma(v)$.
Suppose \eqref{eq:sigma=uv} holds for all $u,v\in R$.
Then we have $d_\sigma(x,y)=e_\sigma(u,v)$ for all
$x,y\in S$ and $u\in\sigma^{-1}(x)$, $v\in\sigma^{-1}(y)$.
This implies that $d_\sigma$ is a metric on $S$.
It is also obvious that $\sigma_*:(R_*,e_\sigma) \to (S, d_\sigma)$ is an isometry.
\end{proof}

\begin{defn} \label{defn:induced-metricS}\upshape
We say that $S$ has the {\it induced metric} from $\sigma$
if $\sigma_*:R_*\to S$ is injective.
This is the case when \eqref{eq:sigma=uv} holds   
for all $u,v\in R$, and therfore 
$\sigma_*:(R_*, e_\sigma) \to (S,d_\sigma)$ is an isometry by Lemma \ref{lem:non-metric}.
In this case, $d_\sigma$ is called the induced metric from $\sigma$.
\end{defn}

\begin{cor}\label{cor:alex-ruled}
Let $S$ be a ruled surface in a $\CAT(\kappa)$-space
$X$ with parametrization $\sigma \colon R \to X$.
If $S$ has the induced metric from $\sigma$, then $(S, d_{\sigma})$ is a $\CAT(\kappa)$-space.
\end{cor}

From now, in the rest of this section,  we consider
curves $\gamma$ in $S$ with respect to the topology of $S$ induced from $X$.
 
\begin{lem} \label{lem:S-conn}
If $S$ has the induced metric from $\sigma$, then
$s(\gamma([a,b]))$ is an interval for any continuous curve 
$\gamma:[a,b]\to S$.
\end{lem}
\begin{proof} Let $J:=[a,b]$.
If the conclusion does not hold, we have $s_-<s_+$ in  $s(\gamma(J))$
such that $(s_-,s_+)$ does not meet $s(\gamma(J))$.
Set $R_-:= [0,s_-]\times [0,1]$, 
$R_+:= [s_+,\ell]\times [0,1]$.
Let $J_+$ and $J_-$ be the set of all $t\in J$
such that the arc $\sigma^{-1}(\gamma(t))$ is 
contained in $R_-$ and $R_+$ respectively.
%%%%%
Since $S$ has the induced metric from $\sigma$, Corollary \ref{cor:esigma=0} implies that 
$J=J_+\cup J_-$.
We show that $J_-$ and $J_+$ are open, yielding a contradiction. Suppose $J_-$ is not open for instance,
and choose  $t\in J_-\setminus {\rm int} J_-$ and a sequence $t_n$
in $J_+$ converging to $t$.
%%%%%%%
Choose any $x_n\in \sigma^{-1}(\gamma(t_n))$
converging to a point $x_\infty\in R_+$.
Since $\sigma(x_n)=\gamma(t_n)\to \sigma(x_\infty)$ as 
$n\to\infty$, we have $\gamma(t)=\sigma(x_\infty)$.
It turns out that $\sigma^{-1}(\gamma(t))\in R_+$.
This is a contradiction to $t\in R_-$.
%%%%%%%
%Then the Euclidean distance $d(\sigma^{-1}(\gamma(t_n)), \sigma^{-1}(\gamma(t)))$ must go to $0$.
%Otherwise 
%However from $\sigma(x_n)\to \sigma(x_\infty)$, we have
%$x_\infty\in \sigma^{-1}(\gamma(t))\subset R_-$.
%This is a contradiction.
%Again we have a contradiction since $\sigma^{-1}(\gamma(t_n))\subset R_+$ and $\sigma^{-1}(\gamma(t))\subset R_-$.
\end{proof}

\psmall

\pmed\n
{\bf Interior metrics on ruled surfaces}

Let $S$ be a ruled surface in $X$ with parametrization
$\sigma \colon R \to X$.

\begin{defn}   \upshape \label{defn:int-metric}
We denote by $d_S$ the {\it interior metric} on $S$ 
associated with $d_X$
defined as
\[
d_S(x_0,x_1) := 
\inf \left\{ \,
L(\gamma) \mid 
\text{$\gamma$ is a curve in $S$ from $x_0$ to $x_1$}
\right\}.
\]
\end{defn}

 Due to the Arzela-Ascoli theorem,
$(S, d_S)$ is a geodesic space.

We discuss the problem of lifting curves in $S$.
For a subset $A\subset S$, we set
\[
  s(A):=\{ s\,|\,\lambda_s\cap A\neq\emptyset\}.
\]
Note that  $s(A)=p_1(\sigma^{-1}(A))$.
 In particular, for every $x\in S$, 
 we define $s(x)$, $s_{\max}(x)$, $s_{\min}(x)$ in this way 
 as in Definition \ref{defn:s(x)}.

\begin{prop}\label{prop:e-sigma3}
Let $\gamma:[a,b]\to S$ be a continuous curve of finite
length from $\sigma(u)$ to $\sigma(u')$ with 
$p_1(u)<p_1(u')$.
Set $J:=[p_1(u),p_1(u')]$, and 
\[
    \Lambda_s^*:=\sigma_*^{-1}(\gamma([a,b]))\cap I_s^*
      \quad (s\in J).
\]
We assume the following:
\begin{enumerate}
\item For arbitrary $t<t'$ in $[a,b]$, $s(\gamma([t,t']))$ is connected\,$;$
\item $\sigma_*(\Lambda_{s}^*)$ is 
   $\gamma$-convex for each $s\in J$\,$;$
\item $\gamma$ is monotone on 
   $\sigma_*(\Lambda_{s}^*)$ that is an interval.
\end{enumerate}
Then there is a lift of $\gamma$ in $R$ from $u$ to $u'$.
\end{prop}

\begin{lem} \label{lem:*-monotoneS}
For a continuous curve  $\gamma:[a,b]\to S$ with
$s_{\min}(\gamma(a))\le s_{\max}(\gamma(b))$,
suppose (1), (2) of Proposition \ref{prop:e-sigma3} for 
$\gamma$.
Then we have the monotonicity for all $t<t'$ in $[a,b]$,
\[
       s_{\min}(\gamma(t))\le s_{\max}(\gamma(t')).
\]
\end{lem}
\begin{proof}
Using the condition (1) of Proposition \ref{prop:e-sigma3}
 in place of Lemma \ref{lem:p1-interval},
we can proceed in the same manner as  the proof of Lemma \ref{lem:*-monotone} 
in our setting, to get the conclusion.
\end{proof}
%%%%%%%%

\begin{proof}[Proof of Proposition \ref{prop:e-sigma3}]
In view of the conditions (2),(3) and Lemma \ref{lem:*-monotoneS}, 
using $\Lambda_s^*$ in place of 
$E_s^*$, we construct the family of continuous arcs
$c_{\Lambda_s^*}$ by the same manner as in 
Proposition \ref{cor:e-sigma2}. Then  parametrize them  and take the union of those arcs to obtain a lift
of $\gamma$ in $R$. Since the procedure is the same, 
we omit the detail.
\end{proof} 

%%%%%%

\begin{thm}  \label{thm:int=ind}
Let $S$ be  a ruled surface in a  $\CAT(\kappa)$-space $X$
with parametrization $\sigma:R\to X$.
If $S$ has the induced metric from $\sigma$, 
then we have $d_S=d_\sigma$, and 
$(S,d_S)$ is  a $\CAT(\kappa)$-space.
\end{thm}
\begin{proof} 
Since $d_S \le d_{\sigma}$,
to see  $d_S=d_\sigma$,
it suffices to show 
$d_S(x,x')\ge d_\sigma(x,x')$ for arbitrary $x, x'\in S$.
Take a $d_S$-shortest curve
$\gamma:[a,b]\to S$ from $x$ to $x'$.  
Since $\gamma$ is $d_S$-shortest, the conditions (2),(3) in Proposition \ref{prop:e-sigma3} certainly hold for $\gamma$.
By Lemma \ref{lem:S-conn},
$s(\gamma([t_1,t_2]))$ is an interval,
and the condition (1) in Proposition \ref{prop:e-sigma3} holds too.
Therefore by Proposition \ref{prop:e-sigma3},
we have a lift of $\gamma$ in $R$. Thus we have $d_S(x,x')=L(\gamma)\ge d_\sigma(x,x')$.
Finally Corollary  \ref{cor:alex-ruled} implies that
$(S,d_S)$ is a $\CAT(\kappa)$-space. 
This completes the proof.
\end{proof}

\setcounter{equation}{0}

\section{Thin ruled surfaces} \label{sec:ruled}
Let $X$ be a locally compact, geodesically complete two-dimensional space with curvature $\le\kappa$, 
and fix $p\in X$.
It is known that $\Sigma_p(X)$ is a finite metric graph without endpoints.
For a vertex $v$ of $\Sigma_p(X)$, take  $\nu_1, \nu_2\in\Sigma_p(X)$ 
with equal distance to $v$ 
such that $\angle(\nu_1, v)+\angle(v,\nu_2)=\angle(\nu_1,\nu_2)$ and 
$v$ is the unique vertex contained in the shortest geodesic joining $\nu_1$ and $\nu_2$ in $\Sigma_p(X)$. We set 
\begin{align}  
       \delta := \angle(\nu_1, v) = \angle(\nu_2, v), \label{eq:delta}
\end{align}
where $\delta$ is assumed to be small enough and will be determined later on in Section \ref{sec:ruled}. 
Let $\alpha_i:[0,\ell]\to X$ be geodesics in the directions $\nu_i$, $i=1,2$.
Joining $\alpha_1(s)$ to $\alpha_2(s)$ by the minimal geodesic $\lambda_s:[0,1] \to X$, 
we have a ruled surface $S$ in $X$. 
Let $B(p,r)$ be a small ball, and we assume  $\ell =2r$.
Set $R=[0,\ell]\times [0,1]$.
Let $\sigma:R\to S$ be the map that defines $S$:
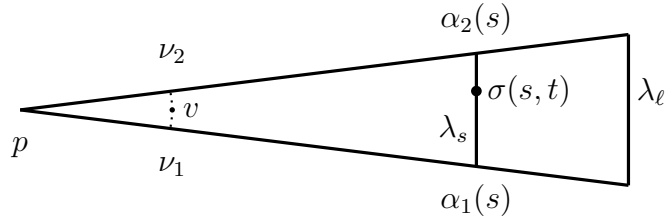
\[
       \sigma(s,t) = \lambda_s(t).
\]

%%%%%tikz picture%%%%%
\begin{center}
\begin{tikzpicture}
[scale = 1]
\draw [-, very thick] (-6,0)--(2,1);
\draw [-, very thick] (-6,0)--(2,-1);
\draw [-, very thick] (2,1)--(2,-1);
\draw [-, very thick] (0,0.75)--(0,-0.75);
\draw [dotted, thick] (-4,0) arc(0:15:1.1);
\draw [dotted, thick] (-4,0) arc(0:-15:1.1);
\draw (-6,-0.5) node[circle] {$p$};
\draw (0,-1.2) node[circle] {$\alpha_1(s)$};
\draw (0,1.2) node[circle] {$\alpha_2(s)$};
\draw (-0.3,-0.25) node[circle] {$\lambda_s$};
\draw (2.3,0.25) node[circle] {$\lambda_{\ell}$};
\fill (0,0.25) coordinate (A) circle (2pt) node [right] {$\sigma(s,t)$};
\draw (-4,-0.75) node[circle] {$\nu_1$};
\draw (-4,0.75) node[circle] {$\nu_2$};
\fill (-4,0) coordinate (A) circle (1pt) node [right] {$v$};
\end{tikzpicture}
\end{center}
%%%%%%%%%%

We define the {\it boundary} and the {\it interior} of $S$ as  
\beq \label{eq:bdyS}
    \pa S:= \alpha_1\cup\alpha_2\cup\lambda_{\ell},  \qquad 
     {\rm int}\,S:=S\setminus\pa S.
\eeq
\pmed
The purpose of this section is to prove the following

\begin{thm} \label{thm:ruled}
There exists an $r_p>0$ such that for every $r\in (0, r_p]$, $S$ with length metric is a ${\rm CAT}(\kappa)$-space
homeomorphic to a two-disk.
\end{thm}

The proof of Theorem \ref{thm:ruled} is completed in Subsection \ref{ssec:monotonicity}.
%%%%%
As shown in the following example, Theorem \ref{thm:ruled} does 
not hold for a general ruled surface even in a two-dimensional 
ambient space.

\begin{ex}  \label{ex:non-disk}  \upshape
For any $0<a<\pi/2$, let $X_0$ be the complement of the domain 
$\{ (x,y)\,|\, |y| < (\tan a) x\}$ on the $xy$-plane.
For $b$ with 
\beq  \label{eq:ab-ineq}
           a<b<a+\pi/4,
\eeq
consider the Euclidean cone $K(I)$ 
over a closed interval $I$ of length $2b$.
Let $X_1$ be the gluing of $X_0$ and $K(I)$ along their boundaries, where the origin $o$ of $X_0$ is identified 
with the vertex of $K(I)$.
Let $\xi$ be the midpoint of $I$ and $\gamma_{\xi}$ denote
the geodesic ray of $X_1$ from $o$ in the derection $\xi$.
Next consider the Euclidean cone $K(J)$ 
over an interval $J$ of length $\theta$ with $\pi-(b-a)\le \theta<\pi$.
Let $X$ be  the gluing of $X_1$ and $K(J)$ in such a way 
that $\partial K(J)$ is identified with $\gamma_{\xi}$ and 
$L:=\{ (x,0)\,|\,x\le 0\}\subset X_0$ in an obvious way.
It is easy to see that $X$ is 
a locally compact, geodesically complete, 
two-dimensional $\CAT(0)$-space.
Let $p=(0, -10)\in X_0\subset X$,
and let $\sigma_+$ (resp.$\sigma_-$) be the geodesic ray 
starting from $o$ defined by the ray
$y=(\tan a)x$ (resp. by the ray $y=-(\tan a)x$).
Note that the geodesic in $X$ joining $p$ and $\sigma_+(1)$
intersect $\sigma_-\setminus \{ o\}$ because of 
\eqref{eq:ab-ineq}.
Let $\ell:=2d(p, \sigma_+(1))$,
and let $\alpha_1:[0,\ell]\to X$ be the geodesic starting from $p$ through $\sigma_+(1)$.
Let $q_1$ be the intersection point of $\alpha_1$ with $\gamma_{\xi}$.
Let $q_2$ be the point of $L$ such that
$d(p,q_1)=d(p,q_2)$.
Letting  $\alpha_2:[0,\ell]\to X$ be the geodesic starting from $p$
through $q_2$,
consider 
the ruled surface $S=S(\alpha_1,\alpha_2)$ in $X$.
Let  $\Delta_1$ (resp.$\Delta_2$) be the geodesic 
triangle region in $X_1$ (resp in $K(J)$) with vertices $p$, 
$\alpha_1(\ell)$ and $\alpha_2(\ell)$ (resp.
$o$, $q_1$ and $q_2$).
Obviously, $S$ is the gluing of $\Delta_1$ and $\Delta_2$ 
along the geodesic segments $oq_1$ and $oq_2$.
In particular 
$S$ is not homeomorphic to a disk.
\end{ex}

%%%%%tikz picture%%%%%  p.13
\begin{center}
\begin{tikzpicture}
[scale = 1]
\draw [-, very thick] (-4.5,-1)--(1,1.75);
\draw [-, very thick] (-4.5,-1)--(0.25,-1);
\draw [-, very thick] (1.25,-1)--(3,-1);
\draw [-, very thick] (0,0)--(-0.5,1);
\draw [-, very thick] (0,0)--(0.25,-1);
\draw [-, very thick] (0,0)--(0.75,-0.25);
\draw [-, very thick] (0.62,-0.5)--(1.25,-1);
\draw [-, dotted, very thick] (3,-1)--(3.5,-1);
\draw [-, dotted, very thick] (1,1.75)--(1.5,2);
\draw [-, dotted, very thick] (0,0)--(0.62,-0.5);
\draw [-, dotted, very thick] (0.75,-0.25)--(1.2,-0.4);
\draw [-, dotted, very thick] (-0.5,1)--(-1,2);
\draw [-, very thick] (-0.5,1)--(0.75,-0.25);
\draw [-, very thick] (0.25,-1)--(0.75,-0.25);
\draw [-, very thick] (0.75,-0.25)--(1.25,-1);
\draw (-4.8,-1) node[circle] {$p$};
\draw (-0.2,-0.2) node[circle] {$o$};
\draw (1,0) node[circle] {$q_1$};
\draw (-0.9,1.1) node[circle] {$q_2$};
\draw (3,-0.7) node[circle] {$\alpha_1$};
\draw (0.8,2) node[circle] {$\alpha_2$};
\draw (0.6,0.7) node[circle] {$K(J)$};
\draw (-1.5,-0.5) node[circle] {$S(\alpha_1,\alpha_2)$};
\draw (-3,0.5) node[circle] {$X$};
\filldraw [fill=gray, opacity=.3] 
(-0.5,1) -- (0,0) -- (0.75,-0.25) -- cycle;
\filldraw [fill=gray, opacity=.1] 
(-4.5,-1) -- (0.25,-1) -- (0,0) -- (-0.5,1) -- cycle;
\filldraw [fill=gray, opacity=.1] 
(0,0) -- (0.25,-1) -- (0.75,-0.25) -- cycle;
\filldraw [fill=gray, opacity=.1] 
(0,0) -- (1.25,-1) -- (0.75,-0.25) -- cycle;
\filldraw [fill=gray, opacity=.1] 
(-0.5,1) -- (0.75,-0.25) -- (1.25,-1) -- (3.5,-1) -- (1.5,2) -- cycle;
\end{tikzpicture}
\end{center}
%%%%%%%%%%

\begin{rem}  \label{rem:non-disk:non-metric} \upshape
(1)\,Note that in Example \ref{ex:non-disk},
$\diam((\nabla d_p)(o))=\pi$, which never happens 
in a small neighborhood of $p$ by Lemma \ref{lem:base}.
This suggests the validity of Theorem \ref{thm:ruled}, 
which is verified in the argument below.
\par\n 
(2)\, In Example \ref{ex:non-disk}, take two points $x, y$ from the distinct components of $S\setminus \Delta_2$ 
respectively.
Then  if $x$ and $y$ are sufficiently close to the point $o$, then we have $d_\sigma(x,y)> d_\sigma(x,o)+d_\sigma(o,y)$. Thus $d_\sigma$ is not a distance 
for Example \ref{ex:non-disk}.
\end{rem}

\medskip
\subsection{Behavior of ruling geodesics}
\psmall

In this subsection, we start the study of the behavior of ruling geodesics of $S$.
We begin with two examples, which help us 
to understand the argument in the rest of the paper.
  
\begin{ex} [Kleiner] $($cf. \cite{Ng:volume}$)$ \label{ex:2}\upshape
First consider a smooth nonnegative function $f:\R\to\R_+$ 
such that $\{ f=0\} =\{ 1/n\,|\,n=1,2,\ldots\}\cup [1,\infty)\cup (-\infty,0]$.
Let $\Omega:= \{ (x,y)\,|\,|y|\le f(x), x\in\R\}$, equipped with the natural length
metric induced from that of $\R^2$. 
We set $I_n^+:=\{ (x,+f(x))\,|\,1/(n+1)\le x\le 1/n\}$, 
$I_n^-:=\{ (x,-f(x))\,|\,1/(n+1)\le x\le 1/n\}$,
and 
$L_+:=\{(x,0)\,|\,x\ge 1\}$, $L_-:=\{(x,0)\,|\,x\le 0\}$. 
Let $\ell_n$ denote the length of $I_n^\pm$, and 
let $\kappa_n$ be the maximum of absolute geodesic curvature of $I_n^\pm$.
We choose $f$
satisfying 
\begin{align}  \label{eq:sum-ell-kappa}
    \sum \ell_n<\infty,   \quad \sum \kappa_n\ell_n< 2\pi. 
\end{align}
By these conditions, one can take a closed domain $H$ in $\R^2$
such that 
\begin{enumerate}
 \item $\partial H$ is smooth, connected and concave in the sense that 
    the geodesic curvature  is nonpositive everywhere;
 \item there are consecutive points $p_1,p_2,\ldots, $ on $\partial H$ such that 
    the subarc $K_n$ between $p_n$ and $p_{n+1}$ of $\partial H$ has 
    length equal to $\ell_n$;
 \item if we denote $p_{\infty}$ the limit of $p_n$, the closure of the complement 
    of the arc
    between $p_1$ and $p_{\infty}$ in $\partial H$ consists of two geodesic rays,
    say $R_+$ and $R_-$, in $\R^2$ with $p_1\in R_+$ and $p_{\infty}\in R_-$;
 \item the absolute geodesic curvature of $K_n$ is greater than or equal to $\kappa_n$
    everywhere.
\end{enumerate}
Take four copies $H_1,\ldots,H_4$ of $H$, and denote
$K_n, R_{\pm}\subset \pa H_{\alpha}$ by $K_n^{(\alpha)}$, $R_{\pm}^{(\alpha)}$\,
$(1\le \alpha\le 4)$ respectively.
We put
\begin{align*}
  \partial_+\Omega:& =\left(\bigcup_{n=1}^{\infty} I_n^+\right)\cup L_+\cup L_-,     \\
  \partial_-\Omega:& =\left(\bigcup_{n=1}^{\infty} I_n^-\right)\cup L_+\cup L_-.
\end{align*}
Now glue $H_1$, $H_2$ and $\Omega$ along their boundaries $\partial H_1$,
$\partial H_2$, $\partial_+ \Omega$ in such a way that 
$I_n$, $L_+$ and $L_-$ are glued with $K_n^{(\alpha)}$, 
$R_+^{(\alpha)}$ and $R_-^{(\alpha)}$\,$(\alpha=1,2)$  respectively
in an obvious way. Similarly glue $H_3$, $H_4$ and $\Omega$ along their boundaries  $\partial H_3$,
$\partial H_4$ and $\partial_-\Omega$.

Let  $X$ be the result of these gluings equipped with natural length metric, which is a two-dimensional locally compact, geodesically complete space.
Let $\iota:\Omega\to X$ be the natural inclusion, and let
$O=(0,0)\in \Omega$. 
Note that no neighborhood of $p:=\iota(O)$ in $X$ has triangulation.
Approximating $f$ by functions $f_k$\,$(k=1,2,\ldots)$\, that are  $0$ near $0$, we have polyhedral spaces $X_k$ in 
a similar way which approximate $X$ in the sense of Gromov-Hausdorff
distance.
Applying a result in \cite{BurBuy:upperII}, we see that 
 $X_k$ are $\CAT(0)$-spaces. Thus the limit space $X$ is also 
a $\CAT(0)$-space.
Note that $\mathcal S(X)$ consists of the two curves
$\iota(\partial_+\Omega)$ and $\iota(\partial_-\Omega)$.
\end{ex}

%%%%%%%%%%%%%%%%%%%%
%%%%%%%%%%%%%%%%%%%%
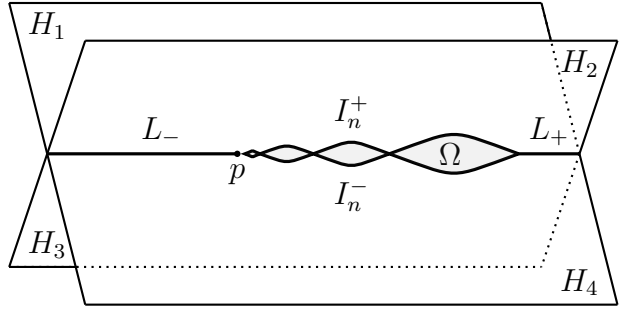
\begin{frame}%Page 14

\begin{center}
\begin{tikzpicture}
[scale = 1]
\draw [-, thick] (-3,2)--(4,2);
\draw [-, thick] (-2,1.5)--(5,1.5);
\draw [-, dotted, thick] (-3,-1.5)--(4,-1.5);
\draw [-, thick] (-2,-2)--(5,-2);
\draw [-, thick] (-3,2)--(-2.5,0);
\draw [-, thick] (-3,-1.5)--(-2.5,0);
\draw [-, thick] (-2,1.5)--(-2.5,0);
\draw [-, thick] (-2,-2)--(-2.5,0);
\draw [-, dotted, thick] (4.5,0)--(4,2);
\draw [-, dotted, thick] (4.5,0)--(4,-1.5);
\draw [-, thick] (4.5,0)--(5,1.5);
\draw [-, thick] (4.5,0)--(5,-2);
\draw [-, thick] (4,2)--(4.125,1.5);
\draw [-, thick] (-3,-1.5)--(-2.1,-1.5);
\draw [-, very thick] (3.7,0)--(4.5,0);
\draw [-, very thick] (-2.5,0)--(0,0);
\fill (0,0) coordinate (A) circle (1.2pt) node [below] {$p$};
\draw (2.8,0) node[circle] {$\Omega$};
\draw (1.5,0.55) node[circle] {$I_n^+$};
\draw (1.5,-0.55) node[circle] {$I_n^-$};
\draw (4.1,0.3) node[circle] {$L_+$};
\draw (-1,0.3) node[circle] {$L_-$};
\draw (-2.5,1.7) node[circle] {$H_1$};
\draw (4.5,1.2) node[circle] {$H_2$};
\draw (-2.5,-1.2) node[circle] {$H_3$};
\draw (4.5,-1.7) node[circle] {$H_4$};
\coordinate (P1) at (3.7,0);
\coordinate (P2) at (2,0);
\coordinate (P3) at (1,0);
\coordinate (P4) at (0.3,0);
\coordinate (P5) at (0.1,0);
\draw [very thick]
(P1) .. controls +(160:1cm) and +(20:1cm) ..
(P2) .. controls +(340:1cm) and +(200:1cm) .. (P1); 
\filldraw [fill=gray, opacity=.1]
(P1) .. controls +(160:1cm) and +(20:1cm) ..
(P2) .. controls +(340:1cm) and +(200:1cm) .. (P1); 
\draw [very thick]
(P2) .. controls +(160:0.6cm) and +(20:0.6cm) ..
(P3) .. controls +(340:0.6cm) and +(200:0.6cm) .. (P2); 
\filldraw [fill=gray, opacity=.1]
(P2) .. controls +(160:0.6cm) and +(20:0.6cm) ..
(P3) .. controls +(340:0.6cm) and +(200:0.6cm) .. (P2); 
\draw [very thick]
(P3) .. controls +(160:0.4cm) and +(20:0.4cm) ..
(P4) .. controls +(340:0.4cm) and +(200:0.4cm) .. (P3); 
\filldraw [fill=gray, opacity=.1]
(P3) .. controls +(160:0.4cm) and +(20:0.4cm) ..
(P4) .. controls +(340:0.4cm) and +(200:0.4cm) .. (P3); 
\draw [very thick]
(P4) .. controls +(160:0.15cm) and +(20:0.15cm) ..
(P5) .. controls +(340:0.15cm) and +(200:0.15cm) .. (P4); 
\filldraw [fill=gray, opacity=.1]
(P4) .. controls +(160:0.15cm) and +(20:0.15cm) ..
(P5) .. controls +(340:0.15cm) and +(200:0.15cm) .. (P4); 
\end{tikzpicture}
\end{center}

\end{frame}

%%%%%%%%%%%%%%%%%%%

\begin{ex} \label{ex:3} \upshape
This example is based on Example \ref{ex:2}.
The point is we make different gluings.
This time we glue $H_1, H_2, H_3, H_4$ and $\Omega$ along their boundaries as follows:
\begin{enumerate}
 \item $K_n^{(1)}$ is glued with $I_n^+$ for all $n\,;$
 \item $K_n^{(2)}$ is glued with $I_n^+$ if $n\not\equiv 2\,\, ({\rm mod}\,4)$ and with $I_n^-$ if $n\equiv 2\,\, ({\rm mod}\,4)\,;$
 \item $K_n^{(3)}$ is glued with $I_n^+$ if $n\equiv 0,1\,\, ({\rm mod}\,4)$ and with $I_n^-$ if $n\equiv 2,3\,\, ({\rm mod}\,4)\,;$
 \item $K_n^{(4)}$ is glued with $I_n^+$ if $n\equiv 3\,\, ({\rm mod}\,4)$ and with $I_n^-$ if $n\not\equiv 3\,\, ({\rm mod}\,4)$. 
\end{enumerate}
Here, $R_+^\alpha$ and $R_-^\alpha$\,$(1\le\alpha\le 4)$ are 
glued with $L_+$ and $L_-$ respectively in those gluings.
The result $Y$ of these gluings equipped with natural length metric is a 
two-dimensional locally compact, geodesically complete $\CAT(0)$-space.
Let $\iota:(\amalg_{i=1}^4 H_i)\amalg\Omega\to Y$ be the identification map.
Note that 
\begin{align}\label{eq:choice-glueing}
\text{for all $1\le\alpha\neq\beta\le 4$,\, 
     $\iota(K_n^\alpha)=\iota(K_n^\beta)$\,\, for some $n$.}
    \hspace{1.5cm}
\end{align}
Let $p:=\iota(O)$, where $O$ is the origin of $\Omega$,
and let $v$ denote the direction at $p$ defined by the union of 
all $I_n^{\pm}$\,$(n=1,2,\ldots)$.
For small $\epsilon>0$, take sufficiently small $r>0$ and 
choose $a_i\in S(p,r)\cap\iota(H_i)$, $1\le i\le 4$,
such that $\angle(\dot\gamma_{p,a_i}(0),v)=\epsilon$.
Let $S(a_i,a_j)$ be the ruled surface defined 
by the geodesic segments $\gamma_{p,a_i}$ and 
$\gamma_{p,a_j}$.
Then it follows from \eqref{eq:choice-glueing} that 
$S(a_i,a_j)$ are not convex 
in $Y$ for ll $i\neq j$.
\end{ex}

\begin{rem}  \label{rem:ex-convex} \upshape
(1)\,In Example \ref{ex:2}, if we take $a_i$ in a way similar to Example \ref{ex:3}, then 
$S(a_i, a_j)$  for $i=1,2$ and $j=3,4$ are convex in $X$ while 
$S(a_1, a_2)$ and $S(a_3, a_4)$ are not convex. 
Considering the other vertex of $\Sigma_p(X)$, it is possible 
to fill a neighborhood of the singular set $B(p,r)\cap \ca S$ via those convex ruled surfaces. 
This is not the case of Example \ref{ex:3}. 
\par\n
(2)\, In Example \ref{ex:3}, it is impossible to fill the ball
$B(p,r)$ for any $r>0$ via a properly embedded convex disks.
More strongly, there is no such a  convex disk 
properly embedded in $B(p,r)$. If there is such a 
convex disk $D$, from the convexity of $D$, 
we can take some $a_i\neq a_j$ in $\pa D$.  
The convexity of $D$ also implies that $S(a_i,a_j)\subset D$,
and hence $S(a_i,a_j)$ must be convex.
However this is impossible as indicated in Example \ref{ex:3}.
\end{rem}
\pmed

For $x\in S$ with $x=\lambda_s(t)$, from Lemma \ref{lem:base},
we have
\begin{align} \label{eq:lambda-gradient}
  |\angle(\pm\dot\lambda_s(t), (\nabla d_p)(x))-\pi/2|<
 \tau_p(|p,x|,\delta).
\end{align}

For $x\in S$, let $\Sigma_x(S)$ denote the set of all directions $\xi\in\Sigma_x(X)$
such that $\xi=\lim_{i\to\infty} \up_x^{x_i}$ for some sequence $x_i\in S$ with $|x, x_i|_X\to 0$, as in Section \ref{sec:basic}.
We call  $\Sigma_x(S)$ the {\it extrinsic space of directions of $S$ at $x$}. 

%%%%%tikz picture%%%%%
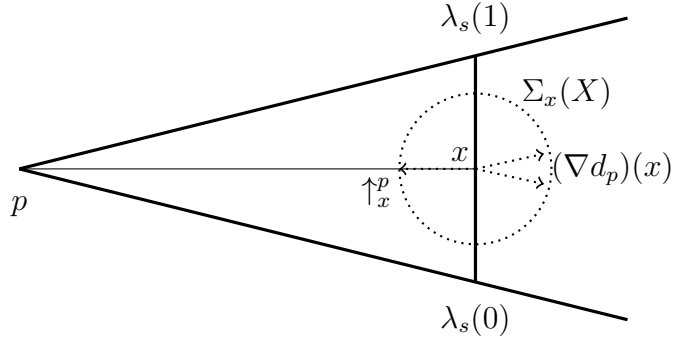
\begin{center}
\begin{tikzpicture}
[scale = 1]
\draw [-] (-6,0)--(0,0);
\draw [-, very thick] (-6,0)--(2,2);
\draw [-, very thick] (-6,0)--(2,-2);
\draw [-, very thick] (0,1.5)--(0,-1.5);
\draw [dotted, thick] (0,0) circle [radius=1];
\draw [->, dotted, thick] (0,0)--(0.9,0.2);
\draw [->, dotted, thick] (0,0)--(0.9,-0.2);
\draw [->, dotted, thick] (0,0)--(-1,0);
\draw (-6,-0.5) node[circle] {$p$};
\draw (-0.2,0.2) node[circle] {$x$};
\draw (0,2) node[circle] {$\lambda_s(1)$};
\draw (0,-2) node[circle] {$\lambda_s(0)$};
\draw (-1.3,-0.3) node[circle] {$\uparrow_x^p$};
\draw (1.8,0) node[circle] {$(\nabla d_p)(x)$};
\draw (1.2,1) node[circle] {$\Sigma_x(X)$};
\end{tikzpicture}
\end{center}
%%%%%%%%%%

In this paper, we use the following terminology.
We call a direction $\xi\in \Sigma_x(X)$
\begin{align*} 
 \begin{cases}
 {\it horizontal} \,\,\,&{\rm if}\,\, \angle(\xi, \pm (\nabla d_p)(x))\le 3\pi/10, \hspace{2cm}\\    
  {\it vertical} \,\,\,&{\rm if}\,\, \angle(\xi, \pm (\nabla d_p)(x))\ge  \pi/5,\\
  {\it medial}  \,\,\,&{\rm if \,\,it \,\,is \,\, horizontal\,\, and\,\, vertical}.\end{cases}
\end{align*}
We also call a direction $\xi\in\Sigma_x(X)$ 
\begin{align*} 
 \begin{cases}
  {\it negative} \,\,\,&{\rm if}\,\,  \angle(\xi,  - (\nabla d_p)(x))<\pi/2  \hspace{2cm}\\   
{\it positive} \,\,\,&{\rm if}\,\,   \angle(\xi,  (\nabla d_p)(x)) <\pi/2.
\end{cases}
\end{align*}
We say that an open subset $\Omega\subset\Sigma_x(X)$ is  in the {\it positive side} 
(resp. {\it negative side}) of  $\Sigma_x(X)$ if every element of $\Omega$ is positive (resp. negative).

Assume that a Lipschitz curve  $c:[a,b] \to B(p,r)\setminus \{ p\}$  has the right and left directions $\dot c_+(t)$ and $\dot c_-(t)$ respectively at every $t\in [a,b]$. 
We say that such a curve  $c$ is {\it vertical} (resp. {\it horizontal} or {\it medial}\,) if both  $\dot c_+(t)$ and 
$\dot  c_-(t)$ 
are vertical  (resp. horizontal or medial) for every $t\in [a,b]$. 
\pmed

Recall that for every  $x\in S$, 
\begin{align*}
 & s(x) = \{ \, s\in [0,\ell]\,|\, x\in \lambda_s \}, \qquad\\
 &  s_{\rm max} (x) = \max s(x), \quad
   s_{\rm min} (x) = \min s(x).
 \end{align*}
For every  $x\in {\rm int}\,S$, we set
 \[  
 +\dot {\boldmath \lambda}(x) := \{\, \up_x^{\lambda_s(1)}\,|\, s   \in s(x)\,\}, \quad
  -\dot {\boldmath \lambda}(x) := \{\, \up_x^{\lambda_s(0)}\,|\, s\in s(x)\,\}.
\]
We show that $s(x)$ is a closed interval later in Lemma \ref{lem:point-crossing}.

%%%%%tikz picture%%%%%
\begin{center}
\begin{tikzpicture}
[scale = 1]
\draw [->, dotted, thick] (0,0)--(-1,0);
\draw [-, very thick] (0,2)--(0,-2);
\draw [dotted, thick] (0,0) circle [radius=1];
\draw [-, very thick] (-3,2)--(3,2);
\draw [-, very thick] (-3,-2)--(3,-2);
\draw [very thick] (-1.15,2) arc(60:-60:2.3);
\draw [very thick] (-0.55,2) arc(30:-30:4);
\draw [very thick] (0.55,2) arc(150:210:4);
\draw [very thick] (1.15,2) arc(120:240:2.3);
\draw (0.3,0) node[circle] {$x$};
\draw (-1.5,0) node[circle] {$\uparrow_x^p$};
\draw (0,2.5) node[circle] {$\alpha_2(s(x))$};
\draw (0,-2.5) node[circle] {$\alpha_1(s(x))$};
\draw (-1.8,2.3) node[circle] {{\footnotesize $\alpha_2(s_{\min}(x))$}};
\draw (1.8,2.3) node[circle] {{\footnotesize $\alpha_2(s_{\max}(x))$}};
\draw (-1.8,-2.3) node[circle] {{\footnotesize $\alpha_1(s_{\min}(x))$}};
\draw (1.8,-2.3) node[circle] {{\footnotesize $\alpha_1(s_{\max}(x))$}};
\filldraw [fill=gray, opacity=.2] 
(-1.15,-2) arc(-60:60:2.3) -- (-1.15,2) -- (1.15,2) arc(120:240:2.3) (1.15,-2) -- cycle;
\end{tikzpicture}
\end{center}
%%%%%%%

\begin{lem} \label{lem:s(x)-diam}
For every  $x\in {\rm int}\, S$, we have
\benum
 \item  $\diam(s(x))/|p,x| < \tau_p(|p,x|);$
 \item$\diam(\pm  \dot {\boldmath \lambda}(x)) <   
\tau_p(|p,x|)).$
\eenum
\end{lem}

\begin{proof}
Suppose that the conclusion does not hold. Then we have a sequence
$x_i\in\inte\, S$ and a positive constant $c$ such that one of the following holds:
\benum
 \item[$(i)$] $\diam s(x_i)/|p,x_i| \ge c;$
 \item[$(ii)$] $\diam (\pm\dot\lambda(x_i))\ge c$.
\eenum
Let $x_i=\lambda_{s_i}(t_i)$. Note that $0<t_i<\ell$.

\pmed\n

%contradiction to either $(1)$ or $(2)$ according to the case when it happens.
%\pmed\n
%Case II)\,\,$\lim_{i\to\infty} \min\{t_i, 1-t_i\}=0$.
%\psmall\n
We may assume $t_i=\min\{t_i, 1-t_i\}$ since the other case is similar.
For any other $s_i'\in s(x_i)$, from $|x_i,p|\to 0$, 
we have 
\[
\lim_{i\to\infty} \,\angle x_i \lambda_{s_i}(0) p=\pi/2 -\delta, \quad
\lim_{i\to\infty} \,\angle x_i \lambda_{s_i'}(0) p=\pi/2 -\delta.
\]
We may assume that $s_i' < s_i$ without loss of generality.
Note also that 
\begin{align*}
   & \lim_{i\to\infty} \angle x_i\lambda_{s_i'}(0)\lambda_{s_i}(0)  = \pi/2+\delta, \\
  & \lim_{i\to\infty} ( \tilde\angle\lambda_{s_i}(0) x_i \lambda_{s_i'}(0)
    + \tilde\angle x_i \lambda_{s_i}(0)  \lambda_{s_i'}(0) 
+ \tilde\angle\lambda_{s_i}(0) \lambda_{s_i'}(0) x_i )  = \pi.
\end{align*}
It follows that
\begin{align*}
  \lim_{i\to\infty} \angle\lambda_{s_i}(0) x_i \lambda_{s_i'}(0)
 &\le  \lim_{i\to\infty} \tilde\angle\lambda_{s_i}(0) x_i \lambda_{s_i'}(0) \\
 & = \lim_{i\to\infty} (\pi -  \tilde\angle x_i \lambda_{s_i}(0)  \lambda_{s_i'}(0) 
      -  \tilde\angle\lambda_{s_i}(0) \lambda_{s_i'}(0) x_i ) \\
 & = 0.
\end{align*}
Thus we conclude that $\diam(-\dot\lambda(x_i))\to 0$. 
Therefore the assumption $(ii)$ does not hold. 
Note that 
\[
   |s_i-s_i'|\le |p,x_i|t_i \tilde\angle\lambda_{s_i}(0) x_i \lambda_{s_i'}(0),
\]
Therefore, from $\lim_{i\to\infty} \tilde\angle\lambda_{s_i}(0) x_i \lambda_{s_i'}(0)=0$,
we see that the assumption $(i)$ does not hold either.
\end{proof}

\pmed
Now, for any $x\in S$ and $s\in (0,\ell)$ with $x\notin \lambda_s$,
let $y \in\lambda_s$ be such that $|x,y|=|x,\lambda_s|$.
By Lemma \ref{lem:base}, we have either
\psmall
\[
\text{
$\angle(\uparrow_y^x, -\nabla d_p)<\tau_p(\delta,r)$ \, or \,
 $\angle(\uparrow_y^x, \nabla d_p)<\tau_p(\delta,r)$}.
\]
\psmall

\begin{rem} \label{rem:tau_p} \upshape
From the proof of Lemma \ref{lem:s(x)-diam} by contradiction,  one might think the function  $\tau_p(\,\cdot\,)$ depends also 
on $S$. However the number of possible $S$ at $p$
is finite. Therefore it is possible to find such a function   depending only on $p$. From now on, we use this convention.
\end{rem}

\begin{lem} \label{lem:foot}
Under the above situation, 
if $\angle(\uparrow_y^x, -\nabla d_p)<\tau_p(\delta,r)$
$($resp. $\angle(\uparrow_y^x, \nabla d_p)<\tau_p(\delta,r))$,
then we have
$$\angle(\uparrow_x^y, \nabla d_p)<\tau_p(\delta,r)\,\,
(\text{resp}.\,\, \angle(\uparrow_x^y, -\nabla d_p)<\tau_p(\delta,r)).
$$
\end{lem} 
\begin{proof} 
We assume 
$\angle(\uparrow_y^x, -\nabla d_p)<\tau_p(\delta,r)$.
The proof of the other case is similar.
By Lemma \ref{lem:base}, we have 
$\angle pyx <\tau_p(\delta,r)$. Since
$\angle xpy <2\delta$, it follows from Lemma
\ref{lem:comparison} that
$\angle pxy > \pi - \tau_p(\delta,r)$,
which implies 
$\angle(\uparrow_x^y, \nabla d_p)<\tau_p(\delta,r)$.
\end{proof}

\begin{lem} \label{lem:basic}
 For $x\in  {\rm int}\,S$, fix $s_0$ and $t_0$ with $x=\lambda_{s_0}(t_0)$.  
Then for every $u\in\Sigma_x(S)$ with 
\beq
                       \angle(u, \pm\dot\lambda_{s_0}(t_0))\ge\pi/3,  \label{eq:horiz}
\eeq
there exists a shortest path 
$\xi_{\infty}:[-1, 1]\to \Sigma_x(X)$ 
satisfying
\beq  
\left\{
 \begin{aligned} \label{eq:infinite-app}
     &u\in\xi_{\infty}([-1,1]), \\ 
   & \angle_X(\xi_{\infty}(\pm 1), \pm\dot\lambda_{s_0}(t_0))<\tau_p(r), \hspace{3cm}\\ 
                %\label{eq:infinite-app}
  &  \xi_{\infty}([-1, 1])\subset \Sigma_x(S).
\end{aligned}
\right.
\eeq
\end{lem}

%%%%%tikz picture%%%%%
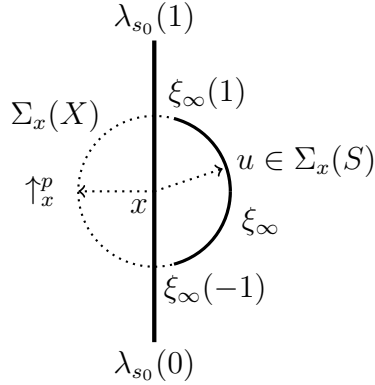
\begin{center}
\begin{tikzpicture}
[scale = 1]
\draw [->, dotted, thick] (0,0)--(-1,0);
\draw [-, ultra thick] (0,2)--(0,-2);
\draw [dotted, thick] (0,0) circle [radius=1];
\draw [->, dotted, thick] (0,0)--(0.9,0.3);
\draw [very thick] (1,0) arc(0:75:1);
\draw [very thick] (1,0) arc(0:-75:1);
\draw (-0.2,-0.2) node[circle] {$x$};
\draw (0,2.3) node[circle] {$\lambda_{s_0}(1)$};
\draw (0,-2.3) node[circle] {$\lambda_{s_0}(0)$};
\draw (-1.5,0) node[circle] {$\uparrow_x^p$};
\draw (2,0.4) node[circle] {$u \in \Sigma_x(S)$};
\draw (1.4,-0.4) node[circle] {$\xi_{\infty}$};
\draw (0.7,1.3) node[circle] {$\xi_{\infty}(1)$};
\draw (0.8,-1.3) node[circle] {$\xi_{\infty}(-1)$};
\draw (-1.3,1) node[circle] {$\Sigma_x(X)$};
\end{tikzpicture}
\end{center}
%%%%%%%%%%%%%%%%%%%%%%%%%%%%%%%%
We need a sublemma.

\begin{slem} \label{slem:trivial}
There is a uniform positive constant $c$ satisfying the following$:$
 For any $x\in S$ and any horizontal direction $\xi\in\Sigma_x(S)$,
let $y_i\in S$ be a sequence such that $y_i\to x$ and $\uparrow_x^{y_i}\to \xi$.
There is an $i_0$ such that if  $\lambda_{s_i}$ meets $y_i$ for some $i\ge i_0$, then we have 
\[
        |x, \lambda_{s_i}|\ge c|x, y_i|.
\]
\end{slem}

\begin{proof}
Note that $\lambda_{s_i}$ does not pass through $x$
for any large enough $i$ since $\xi$ would be a 
direction tangent to $\lambda_{s_i}$ at $x$ otherwise,
which is a contradiction.
Take $z_i\in\lambda_{s_i}$ such that $|x, z_i|=|x, \lambda_{s_i}|$. 
By Lemmas \ref{lem:base} and \ref{lem:near-vert}, we have 
\[
      |\angle x z_i y_i - \pi/2| <\tau_p(\delta, r).
\]
It follows that 
$\angle (\uparrow_ {z_i}^x, (-\nabla d_p)) <\tau_p(\delta, r)$
or $\angle (\uparrow_ {z_i}^x, \nabla d_p) <\tau_p(\delta, r)$.
We assume the former since the latter case is similar. 
It follows from Lemma \ref{lem:foot} that 
$\angle (\uparrow_ x^{z_i},  \nabla d_p) <\tau_p(\delta, r)$.
Lemma \ref{lem:jack} implies that 
\beq \label{eq:pzy}
  |\tilde\angle pz_i y_i -\pi/2|<\tau_p(\delta, r).
\eeq
By Lemma \ref{lem:comparison}, we have
\beq \label{eq:pzx}
  \tilde\angle pz_i x <\tau_p(\delta, r).
\eeq
Let $x_i$ be the point on the geodesic $z_i p$ satisfying $|z_i,x_i|=|z_i,x|$.
By \eqref{eq:pzy}, we have $|\tilde\angle x_iz_i y_i -\pi/2|<\tau_p(\delta, r).$
It follows from \eqref{eq:pzx} that 
\beq \label{eq:xzy}
  |\tilde\angle xz_i y_i -\pi/2|<\tau_p(\delta, r) .
\eeq
Now let us consider the convergence 
\[
       \left(\frac{1}{|x, z_i|} X, x\right )   \rightarrow (K_x(X), o_x),
\]
as $i\to\infty$. Let $z_\infty\in \Sigma_x\subset K_x(X)$
be the limit of $z_i$ under this convergence.
Since $\angle(z_\infty, \nabla d_p)<\tau_p(\delta, r)$
and since $\uparrow_x^{y_i}\to v$,
the limit $y_\infty$ of $y_i$ under the above convergence
certainly exists, and we have
\beq \label{eq:yoz}
     \angle y_\infty o_x z_\infty <\pi/2-\pi/3 +\tau_p(\delta,r)
              < \pi/5.
\eeq
By \eqref{eq:xzy}, we also have 
\beq \label{eq:ozy}
  |\angle o_x z_\infty y_\infty - \pi/2| <\tau_p(\delta, r),
\eeq
\eqref{eq:yoz} and \eqref{eq:ozy} imply that 
$|o_x,z_\infty|\ge c|o_x, y_\infty|$ for some uniform 
constant $c>0$.
This yields the conclusion of the lemma via contradiction.
\end{proof}

%%%%%%%%%%%%%%%%%%%%%%%%%%%
\begin{proof}[Proof of Lemma \ref{lem:basic}]
Let $y_i\in S$ be such that 
$|x, y_i|_X\to 0$ and $\uparrow_x^{y_i}$ converges to $u$.
Take $s_i \in (0,\ell)$, $t_i\in (0,1)$ such that $ y_i=\lambda_{s_i}(t_i)$. 
Let 
 $\e_i:=|x, y_i|_X$, and consider the convergence
\[
       \left(\frac{1}{\e_i} X, x\right )   \rightarrow (K_x(X), o_x),
\]
as $i\to\infty$.
Note that the minimal geodesic $\hat\lambda_{s_i}(t):=\lambda_{s_i}(t_i+\e_i t)$,  $(-t_i/\e_i < t < (1-t_i)/\e_i)$, 
has a uniformly bounded speed for $\frac{1}{\e_i} X$ 
independent of $i$. Therefore passing to a subsequence, we may assume that 
 $\hat\lambda_{s_i}(t)$ converges to a minimal geodesic $\hat\lambda_{\infty}(t)$ in $K_x(X)$
defined on $(-\infty, \infty)$, where this convergence is uniform on every 
bounded interval. 
Note that $\hat\lambda_{\infty}(0)=u$.
From Sublemma \ref{slem:trivial} and \eqref{eq:horiz},  
the geodesic $\hat\lambda_{\infty}$ does not pass through 
$o_x$.
Consider the curve
\beq \label{eq:hlambda}
   \hat\xi_{\infty}(t):=\hat\lambda_{\infty}(t)/|\hat\lambda_{\infty}(t)|.
\eeq
Obviously, $\hat\xi_{\infty}$ is a shortest path
in $\Sigma_x(X)$ and $\hat\xi((-\infty,\infty))\subset\Sigma_x(S)$.
Let $\xi_\infty:[-1,1]\to \Sigma_x(S)$ be a 
reparametrization of the extension $\check\xi_\infty:[-\infty, \infty]\to \Sigma_x(S)$
of $\hat\xi_\infty$.

Take an arbitrary $w_+\in (\nabla d_p)(x)$ and set 
$w_-=-(\nabla d_p)(x)$. 
Consider the two sets
$\{ w_+,  \dot\lambda_{s_0}(t_0),
w_-,  -\dot\lambda_{s_0}(t_0)\}$
and 
$\{ w_+,  \xi_{\infty}(1),
w_-,  \xi_{\infty}(-1)\}$.  They 
are on a circle $C$ in $\Sigma_x(X)$ in these orders,
where
\beq
   |L(C)-2\pi|<\tau_p(r). \label{eq:lengthC}
\eeq
Since $|\angle(\pm w, \xi_\infty(\pm 1)) -\pi/2|<\tau_p(\delta,r)$ and 
$|\angle(\pm w, \pm\dot\lambda_{s_0}(t_0)) -\pi/2|<\tau_p(\delta,r)$, 
we have the conclusion  \eqref{eq:infinite-app}.

This completes the proof.
\end{proof}

\pmed 
A direction $\xi\in\Sigma_x(X)$ is called {\it regular}
if $\xi\notin \mathcal S(\Sigma_x(X))$.

\begin{lem} \label{lem:pass}
For $x\in{\rm int}\, S$, let $\xi_1, \xi_2\in\Sigma_x(S)$ be positively
horizontal (resp. negatively horizontal). 
Assume that $\xi_1$ is regular. Take an $X$-geodesic
$\gamma_1$ such that $\dot\gamma_1(0)=\xi_1$, and 
a sequence $x_i\in S$ such that 
$|x,x_i|_X\to 0$ and $\uparrow_x^{x_i}\to \xi_2$.
Then for $s_0$ with $\lambda_{s_0}\ni x$, there exists an 
$\e>0$ such that if some ruling geodesic $\lambda_s$ with $|s-s_0|<\e$ passes through $x_i$ for a sufficiently large $i$, 
then it passes through $\gamma_1$, too.
\end{lem}

%%%%%tikz picture%%%%%
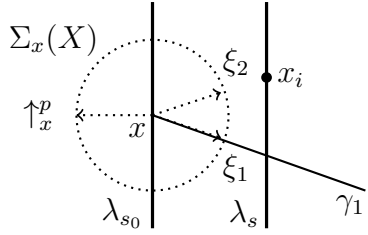
\begin{center}
\begin{tikzpicture}
[scale = 1]
\draw [->, dotted, thick] (0,0)--(-1,0);
\draw [-, very thick] (1.5,1.5)--(1.5,-1.5);
\draw [-, very thick] (0,1.5)--(0,-1.5);
\draw [dotted, thick] (0,0) circle [radius=1];
\draw [->, dotted, thick] (0,0)--(0.9,0.3);
\draw [->, dotted, thick] (0,0)--(0.9,-0.3);
\draw [-, thick] (0,0)--(2.8,-1.0);
\draw (-0.2,-0.2) node[circle] {$x$};
\draw (-0.4,-1.3) node[circle] {$\lambda_{s_0}$};
\draw (1.2,-1.3) node[circle] {$\lambda_{s}$};
\draw (-1.5,0) node[circle] {$\uparrow_x^p$};
\draw (1.1,0.7) node[circle] {$\xi_2$};
\draw (1.1,-0.7) node[circle] {$\xi_1$};
\draw (2.6,-1.2) node[circle] {$\gamma_1$};
\draw (-1.3,1) node[circle] {$\Sigma_x(X)$};
\fill (1.5,0.5) coordinate (A) circle (2pt) node [right] {$x_i$};
\end{tikzpicture}
\end{center}
%%%%%%%%%%

\begin{proof}
Suppose that the conclusion does not hold.
Then there exists a sequence $s_i\to s_0$ such that 
$\lambda_{s_i}$ meets $x_i$ while 
$\lambda_{s_i}$ does not pass through $\gamma_1$.
Applying Lemma \ref{lem:basic} to $x=\lambda_{s_0}(t_0)$ and 
$\xi_2$,
we have a shortest arc 
$\xi_{\infty}:[-1, 1]\to \Sigma_x(X)$ joining two points 
close to $\pm\dot\lambda_{s_0}(t_0)$
such that $\xi_2\in\xi_{\infty}([-1,1])$.
Similarly, applying Lemma \ref{lem:basic} to 
$\xi_1$,
we have a shortest arc 
$\bar\xi_{\infty}:([-1,1])\to \Sigma_x(X)$ joining two points 
close to $\pm\dot\lambda_{s_0}(t_0)$
such that $\xi_1\in\bar\xi_{\infty}([-1,1])$.

Note that  both $\xi_{\infty}$  and  $\bar\xi_{\infty}$  pass through the
positive side of $\Sigma_x(X)$ and connect points close 
to $\pm\dot\lambda_{s_0}(t_0)$.
From construction,  $\xi_{\infty}$ and $\bar\xi_{\infty}$ 
pass through horizontal directions $\xi_2$ and $\xi_1$
respectively. 
Furthermore both intersections 
$\xi_{\infty}([-1, -1+\e_1])\cap \bar\xi_{\infty}([-1, -1+\bar \e_1])$ and  
$\xi_{\infty}([1-\e_2, 1])\cap \bar\xi_{\infty}([1-\bar\e_2, 1])$
are not empty for some small $\e_i, \bar \e_i >0$ \,$(i=1,2)$\, 
since they are in the regular parts of $\Sigma_x(X)$ by Corollary \ref{cor:vert}.
%containing $\pm\dot\lambda_{s_0}(t_0)$.
Therefore by the uniqueness of geodesics in the $\CAT(1)$-space $\Sigma_x(X)$, 
we conclude that  
$\xi_{\infty}([-1+\e_3, 1-\e_3])= \bar\xi_{\infty}([-1+\bar\e_3, 1-\bar\e_3])$
some small $\e_3, \bar \e_3 >0$,  
and in particular
$\xi_{\infty}$ and $\bar\xi_{\infty}$ pass through both $\xi_1$ and 
$\xi_2$.

Take $\xi_3, \xi_4\in\Sigma_x(X)$ close to $\xi_1$ such that 
every element of the arc $[\xi_3,\xi_4]$ in $\Sigma_x(X)$ is regular and 
$\xi_1$ is the midpoint of $[\xi_3,\xi_4]$.
Let $\gamma_i:[0,\e]\to X$ be $X$-geodesics with 
$\dot\gamma(0)=\xi_i$\, $(i=3,4)$, and $\gamma_5$ the $X$-minimal 
geodesic joining $\gamma_3(\e)$ to $\gamma_4(\e)$.
%Note that $\gamma_i$ are horizontal for $i=1,2$.
If $\e>0$ is small enough, then the triangle $\triangle(\gamma_3, \gamma_4,\gamma_5)$ bounds a domain in $X$ homeomorphic to a two-disk $D$.
Let ${\rm int} \,D$ denote the interior of the disk $D$.
Note that ${\rm int} \,D$ is open in $X$ and that
$\gamma_1([0,\e_1])\subset D$ for a small $\e_1<\e$.
Since ${\rm int} \,D$ is open in $X$ and since
$\xi_{\infty}$  is constructed by \eqref{eq:hlambda}, 
$\lambda_{s_i}$ really passes through $\gamma_1$ for large $i$,
which is a contradiction. This completes the proof.
\end{proof}

\subsection{Canonical balls}

In this subsection, we introduce the notion of canonical balls,
which turns out to be useful to have better understanding of the behavior 
of ruling geodesics of $S$.

We denote by $\mathcal R(X)$ the set of topological regular points, 
$\mathcal R(X)=X\setminus \mathcal S(X)$.

\begin{defn} \upshape
For $x\in B(p,r)$, a ball $B(x, \e)$ is called {\it canonical} if for every $y \in B(x,\e)\setminus\{ x\}$ with vertical $\uparrow_x^y$, we have 
$y \in\mathcal R(X)$.
\end{defn}

\begin{lem} \label{lem:exist-canonical}
There exists an $r=r_p>0$ such that there is a
canonical ball around every point  in $B(p,r)\setminus \{ p\}$. 
\end{lem}

Lemma \ref{lem:exist-canonical} is a direct consequence of 
the following Lemma \ref{lem:sing-nabla}, which is immediate from Corollary \ref{cor:vert}.

\begin{lem} \label{lem:sing-nabla}
For every $x\in \ca S(X)\cap B(p,r)\setminus \{ p\}$, we have 
\begin{align*}
 & \sup\{\angle(\xi,\nabla d_p)\,|\,\xi\in\Sigma_x(\mathcal S(X))\text{ is positive}\}
         <\tau_p(|p,x|), \\
& \sup\{\angle(\xi,-\nabla d_p)\,|\,\xi\in\Sigma_x(\mathcal 
S(X))\text{ is negative}\}
         <\tau_p(|p,x|).
\end{align*}
\end{lem}
\pmed

\begin{defn} \upshape
For $x\in {\rm int}\, S$, let $B(x, \e)$ be a canonical ball. We set 
\begin{align*}
      U_+(x,\e) &:= \{ y \in B(x,\e)\,|\, \angle(\uparrow_x^y, \dot {\boldmath \lambda}(x))<\pi/4\,\},\\
      U_-(x,\e) &:= \{ y \in B(x,\e)\,|\, \angle(\uparrow_x^y,-\dot {\boldmath \lambda}(x) )<\pi/4\,\}.
\end{align*}
Note that both $U_+(x,\e)$ and $U_-(x,\e)$ are convex in $X$ for small $\e>0$.
\end{defn}
\pmed

In Lemma \ref{lem:inS}, we show that $U_{\pm}(x,\e)\subset S$ for a small $\e>0$.

We denote by $|A|$ the cardinality of a set $A$.

%%%%%tikz picture%%%%%
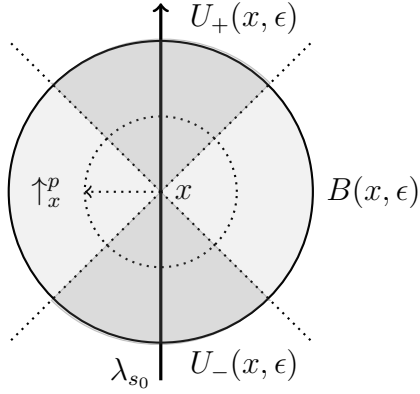
\begin{figure}
\begin{tikzpicture}
[scale = 1]
\draw [->, dotted, thick] (0,0)--(-1,0);
\draw [->, very thick] (0,-2.5)--(0,2.5);
\draw [-, dotted, thick] (0,0)--(2,2);
\draw [-, dotted, thick] (0,0)--(-2,2);
\draw [-, dotted, thick] (0,0)--(2,-2);
\draw [-, dotted, thick] (0,0)--(-2,-2);
\draw [thick] (0,0) circle [radius=2];
\draw [dotted, thick] (0,0) circle [radius=1];
\draw (0.3,0) node[circle] {$x$};
\draw (-0.4,-2.4) node[circle] {$\lambda_{s_0}$};
\draw (-1.5,0) node[circle] {$\uparrow_x^p$};
\draw (1.1,2.3) node[circle] {$U_+(x,\epsilon)$};
\draw (1.1,-2.3) node[circle] {$U_-(x,\epsilon)$};
\draw (2.8,0) node[circle] {$B(x,\epsilon)$};
\filldraw [fill=gray, opacity=.2] 
(0,0) -- (1.44,1.44)  arc (45:135:2) -- (-1.44,1.44) -- (0,0) -- cycle;
\filldraw [fill=gray, opacity=.2] 
(0,0) -- (-1.44,-1.44)  arc (225:315:2) -- (1.44,-1.44) -- (0,0) -- cycle;
\filldraw [fill=gray, opacity=.1] (0,0) circle [radius=2];
\end{tikzpicture}
\caption{A canonical ball around $x$}
\label{fig:figcanonical}
\end{figure}
%%%%%%%%%%

\par
\medskip

\begin{lem} \label{lem:Sfinite}
Let $\gamma:[0,1]\to X$ be a vertical $X$-geodesic in $B(p,r)$.
Then we have
 $|\gamma\cap \mathcal S(X)|<\infty$
\end{lem}

\begin{proof}
Suppose that the lemma does not hold. Then we would have an accumulation 
point $x=\gamma(t_0)$ of $\gamma\cap \mathcal S(X)$. It turns out that 
either $\dot\gamma(t_0)$ or $-\dot\gamma(t_0)$ is in $\Sigma_x(\mathcal S(X))$,
which is a contradiction to the existence of a canonical ball around $x$. 
\end{proof}

\begin{rem} \upshape
At this stage, we do not know yet a uniform bound 
on  $|\gamma\cap \mathcal S(X)|$ for all 
the vertical geodesics $\gamma$. 
 In Section \ref{sec:graph}, we give
a uniform bound (see Sublemma \ref{slem:number-C}).
\end{rem}
\pmed

%\subsection{Key Lemmas}

The following is a key lemma.

\begin{lem} $(${\rm no-return lemma}$)$ \label{lem:noreturn}
For every $s_0$, there exists an $\e>0$ such that for any $s_1\in (s_0-\e, s_0)$ $($resp.  $s_1\in (s_0,  s_0+\e))$, 
there are no $t_0, t_1\in [0,1]$ satisfying that 
 $\uparrow_{\lambda_{s_0}(t_0)}^{\lambda_{s_1}(t_1)}$ is positively horizontal
 $($resp.  negatively horizontal $)$ of 
$\Sigma_{\lambda_{s_0}(t_0)}(X)$.
\end{lem}

\begin{proof}
Suppose that the conclusion does not hold.  
Then we have some sequence $s_i<s_0$ with $\lim_{i\to\infty} s_i= s_0$ such that 
\begin{align} \label{eq:reverse}
  \uparrow_{\lambda_{s_0}(t_i)}^{\lambda_{s_i}(u_i)}  \text{ is  positively horizontal
                                                     for some} \,\,  t_i, u_i \in (0,1).
\end{align}
(see Figure \ref{fig:figreverse}).
 We show that both $\lambda_{s_i}((0,u_i))$ and $\lambda_{s_i}((u_i, 1))$ meet  $\lambda_{s_0}$,  which yields a contradiction
to the minimality of $\lambda_{s_0}$.

%%%%%tikz picture%%%%%
\begin{figure}
\begin{tikzpicture}
[scale = 1]
\draw [-, very thick] (0,2)--(0,-2);
\draw [-, very thick] (-0.5,2) to [out=270, in=120] (-0.1,1);
\draw [-, very thick] (0.1,0.8) to [out=300, in=90] (0.5,0.2);
\draw [-, very thick] (0.5,0.2) -- (0.5,-0.2);
\draw [-, very thick] (0.1,-0.8) to [out=60, in=270] (0.5,-0.2);
\draw [-, very thick] (-0.5,-2) to [out=90, in=240] (-0.1,-1);
\draw [->, dotted, thick] (0,0)--(0.45,0.05);
\draw (0.4,-1.8) node[circle] {$\lambda_{s_0}$};
\draw (-0.9,-1.8) node[circle] {$\lambda_{s_i}$};
\fill (0.5,0.1) coordinate (A) circle (2pt) node [right] {$\lambda_{s_i}(u_i)$};
\fill (0,0) coordinate (A) circle (2pt) node [left] {$\lambda_{s_0}(t_i)$};
\end{tikzpicture}
\caption{$\uparrow_{\lambda_{s_0}(t_i)}^{\lambda_{s_i}(u_i)}$
is positively horizontal}
\label{fig:figreverse}
\end{figure}
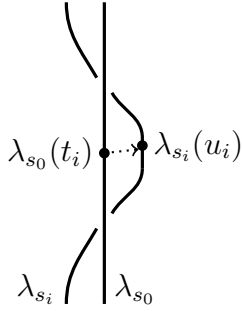
%%%%%%%%%%

From Lemmas \ref{lem:exist-canonical} and  \ref{lem:Sfinite},  it is possible to cover $\lambda_{s_0}$
by finitely many canonical balls  $B(x_{\alpha}, \e_{\alpha})$, $1\le\alpha\le N$, 
where 
 $x_{\alpha} = \lambda_{s_0}(t_{\alpha})$ and $t_{\alpha}<t_{\alpha+1}$.
Taking smaller $\e_{\alpha}$ if necessary, we may further assume that for any large $i$

\begin{enumerate}
 \item $\lambda_{s_i}\subset \bigcup_{\alpha=1}^N B(x_\alpha, \e_\alpha);$
 \item  
    $\displaystyle{
   B(x_\alpha, \e_\alpha)\cap B(x_{\alpha+1}, \e_{\alpha+1})\cap \lambda_{s_i} \subset U_+(x_{\alpha}, \e_{\alpha})\cap U_-(x_{\alpha+1}, \e_{\alpha+1})}$ \\
for each $\alpha$.
\end{enumerate}
Note that $\lambda_{s_0}\cap \mathcal S_X\subset \{ x_{\alpha}\}_{\alpha=1}^N$, and that 
$ U_+(x_{\alpha}, \e_{\alpha})\cap U_-(x_{\alpha+1}, \e_{\alpha+1})$
is convex in $X$ and 
homeomorphic to a disk. 
Suppose that   $\lambda_{s_i}((0, u_i))$ does not meet  $\lambda_{s_0}$. Take a maximal interval $I_{\alpha}^i$ in $[0,1]$
such that $\lambda_{s_0}(I_{\alpha}^i)\subset B(x_\alpha, \e_\alpha)$,
and set 
\[
      \xi_{\alpha}^i(t) := \uparrow_{x_{\alpha}}^{\lambda_{s_i}(t)} \quad (t\in I_{\alpha}^i).
\]
From the assumption, $\xi_{\alpha}^i(I_{\alpha}^i)$ is in either the negative side
or the positive side of $\Sigma_{x_{\alpha}}(X)$.
Note that 
$ \xi_{1}^i(I_1^i)$ is in the negative side of $\Sigma_{x_1}(X)$.
Let $\alpha_0=\alpha_0(i)$ be such that 
$\lambda_{s_0}(t_i)\in B(x_{\alpha_0}, \e_{\alpha_0})$.
From \eqref{eq:reverse}, 
 $ \xi_{\alpha_0}^i(I_{\alpha_0}^i)$ is in the positive side of 
$\Sigma_{x_{\alpha_0}}(X)$. 
Therefore for some $\alpha\le\alpha_0$,
$\xi_{\alpha-1}^i(I_{\alpha-1}^i)$ is in the negative side of 
$\Sigma_{x_{\alpha-1}}(X)$ and 
$ \xi_{\alpha}^i(I_{\alpha}^i)$ is in the positive side of  
$\Sigma_{x_{\alpha}}(X)$.
Now $\lambda_{s_0}$ divides the disk domain 
$U_+(x_{\alpha-1}, \e_{\alpha-1})\cap U_-(x_{\alpha}, \e_{\alpha})$
into two disk domains $D_-$ and $D_+$, where 
we may assume that $\lambda_{s_i}(t_-)\in D_-$  and 
$\lambda_{s_i}(t_+)\in D_+$ for some $t_-\in I_{\alpha-1}$
and  $t_+\in I_{\alpha}$.
Thus $\lambda_{s_i}([t_-, t_+])$ must meet $\lambda_{s_0}$.

Similarly, we would have another intersection point of   $\lambda_{s_i}((u_i, 1))$ and $\lambda_{s_0}$.
This completes the proof.
\end{proof}

The following lemma is a global version of Lemma \ref{lem:noreturn}.

\begin{lem}  \label{lem:noreturn2}
For arbitrary $s_1<s_2$, there are no $t_1, t_2\in [0,1]$ satisfying that 
 $\uparrow_{\lambda_{s_1}(t_1)}^{\lambda_{s_2}(t_2)}$ 
$($resp.  $\uparrow_{\lambda_{s_2}(t_2)}^{\lambda_{s_1}(t_1)}$ $)$ 
is negatively horizontal
 in $\Sigma_{\lambda_{s_1}(t_1)}(X)$
$($resp. positively horizontal
 in $\Sigma_{\lambda_{s_2}(t_2)}(X))$.
\end{lem}

\begin{proof} 
Let $I(s_1)$ be the set of all $s\in (s_1, s_2]$ such 
that there are no $t_1, t\in [0,1]$ satisfying that 
 $\uparrow_{\lambda_{s_1}(t_1)}^{\lambda_{s}(t)}$ is negatively horizontal
 in $\Sigma_{\lambda_{s_1}(t_1)}(X)$.
By Lemma \ref{lem:noreturn}, $(s_1,s_0)\subset I(s_1)$
for some $s_0\in (s_1, s_2)$.
Let $u$ be the supremum of those $s_0$.
From the continuity of the map $\sigma:R\to S$,
$(s_1,s_2]\setminus I(s_1)$ is open in $(s_1,s_2]$.
It follows  that $u\in I(s_1)$.  Suppose that $u<s_2$.
Then we have a sequence of positive numbers $\e_i$ with
$\e_i\to 0$ such that $u_i:=u+\e_i\notin I(s_1)$.
Namely we have sequences $t_i$ and $t_i'$ satisfying that 
  $\uparrow_{\lambda_{s_1}(t_i)}^{\lambda_{u_i}(t_i')}$ is negatively horizontal  in $\Sigma_{\lambda_{s_1}(t_i)}(X)$.
Set
$x_i:=\lambda_{u_i}(t_i')$, and let $y_i:= \lambda_{s_1(t_i)}$.
Take $z_i\in \lambda_{u_i}$ and $w_i\in \lambda_{u}$
such that 
$$
       |y_i, z_i|=|y_i, \lambda_{u_i}|, \qquad 
            |z_i, w_i|=|z_i, \lambda_{u}|.
$$
Since $\uparrow_{y_i}^{x_i}$ is horizontal, we have $y_i\neq z_i$.
By \eqref{eq:lambda-gradient}, we obtain  

\[ 
\text{
$\angle(\uparrow_{z_i}^{y_i}, \nabla d_p)<\tau_p(\delta,r)$
or
$\angle(\uparrow_{z_i}^{y_i}, -\nabla d_p)<\tau_p(\delta,r)$}.
\]
We show that 
\beq \label{eq:angle-zyd}
   \angle(\uparrow_{z_i}^{y_i}, \nabla d_p)<\tau_p(\delta,r).
\eeq
Otherwise, we have $\angle(\uparrow_{z_i}^{y_i}, -\nabla d_p)<\tau_p(\delta,r)$.
In view of Lemma \ref{lem:foot}, it turns out that
\[
               \angle x_i y_i z_i> 2\pi/3-\tau_p(\delta,r),
\]
and hence 
\begin{align*}
     \angle x_i z_i y_i  < \pi - \angle  x_i y_i z_i -
                                 \angle  y_i x_i z_i +\tau_p(r) < \pi/3+\tau_p(\delta,r).
\end{align*}
This is a contradiction to the choice of $z_i$.

Next note that $w_i\neq z_i$. Because if $w_i=z_i$,
then $\uparrow_{y_i}^{w_i}$ must be negatively horizontal
by \eqref{eq:angle-zyd}, 
which contradicts $u\in I(s_1)$.
Now by Lemma \ref{lem:noreturn}, 
$\uparrow_{w_i}^{z_i}$ is positively horizontal.
In view of Lemma \ref{lem:foot},  we have 
\beq \label{eq:angle-zwd}
  \angle(\uparrow_{z_i}^{w_i},- \nabla d_p)<\tau_p(\delta,r),
\eeq
It follows from \eqref{eq:angle-zyd} and \eqref{eq:angle-zwd} that 
$\angle y_iz_i w_i >\pi -\tau_p(\delta,r)$,
which implies 
$\angle(\uparrow_{y_i}^{w_i},- \nabla d_p)<\tau_p(\delta,r)$.
In particular $\uparrow_{y_i}^{w_i}$ is negatively horizontal.
This contradicts $u\in I(s_1)$. Thus we conclude $u=s_2$.

Similarly we see that 
there are no $t_1, t_2$ satisfying that 
 $\uparrow_{\lambda_{s_2}(t_2)}^{\lambda_{s_1}(t_1)}$
is positively horizontal.
This completes the proof.
\end{proof}

\begin{lem} \label{lem:inS}
For every $x\in\inte \,S$, there exists an $\e>0$
such that 
\[
     U_+(x,\e)\subset S, \quad U_-(x,\e)\subset S.
\]  
\end{lem}

\begin{proof}  
Let $B(x,\e_0)$ be a canonical ball.
Take the positively horizontal $v_+\in \Sigma_x(X)$
(resp. negatively horizontal $v_-\in \Sigma_x(X)$) such that 
$\angle(v_{\pm}, \dot\lambda(x)) =\pi/4$.
Let $\gamma_{\pm}$ be $X$-geodesics starting from $x$ with 
$\dot\gamma_{\pm}(0)=v_{\pm}$.

\begin{slem}\label{slem:mean}
For any $0<\e <\e_0$,
there are $s_-\in (0, s_{\min}(x))$ and $s_+\in (s_{\max}(x), r)$
such that $\lambda_{s_{\pm}}$ pass through 
$\gamma_{\pm}((0,\e])$ respctively.
\end{slem}

%%%%%tikz picture%%%%%
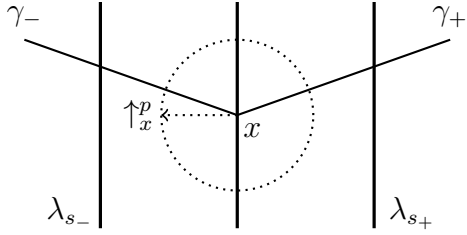
\begin{center}
\begin{tikzpicture}
[scale = 1]
\draw [->, dotted, thick] (0,0)--(-1,0);
\draw [-, very thick] (1.8,1.5)--(1.8,-1.5);
\draw [-, very thick] (0,1.5)--(0,-1.5);
\draw [-, very thick] (-1.8,1.5)--(-1.8,-1.5);
\draw [dotted, thick] (0,0) circle [radius=1];
\draw [-, thick] (0,0)--(2.8,1.0);
\draw [-, thick] (0,0)--(-2.8,1.0);
\draw (0.2,-0.2) node[circle] {$x$};
\draw (2.3,-1.3) node[circle] {$\lambda_{s_+}$};
\draw (-2.2,-1.3) node[circle] {$\lambda_{s_-}$};
\draw (-1.3,0) node[circle] {$\uparrow_x^p$};
\draw (2.8,1.3) node[circle] {$\gamma_+$};
\draw (-2.8,1.3) node[circle] {$\gamma_-$};
\end{tikzpicture}
\end{center}
%%%%%%%%%%

\begin{proof}
Suppose that there is no such $s_-<s_{\min}(x)$. 
Then we have a sequence $s_i<s_{\min}(x)$ converging to
$s_{\min}(x)$ such that 
$\lambda_{s_i}$ does not pass through $\gamma_-((0,\e])$
for some $\e>0$ and all $i$. 
As in the proof of Lemma \ref{lem:basic}
together with Lemma \ref{lem:noreturn}, the curves
$\xi_i(t)=\uparrow_x^{\lambda_{s_i}(t)}$\,$(t\in [0,1])$\,
in $\Sigma_x(X)$
pass through $\dot\gamma_-(0)$ for all $i$.
This shows in particular that $\dot\gamma_-(0)\in\Sigma_x(S)$.
Since $\dot\gamma_-(0)$ is regular, it follows from 
Lemma \ref{lem:pass} that $\lambda_{s_i}$ meets 
$\gamma_-((0,\e])$ for every large enough $i$.
This is a contradiction.

Similarly, we see that $\lambda_s$ meets $\gamma_+((0,\e])$
for any $s>s_+$ sufficiently close to $s_+$.
\end{proof}

Take a sufficiently small $0<\e_1<\e_0$ such that 
\beq \label{eq:span-triangle}
\text{the triangle $\Delta\gamma_+(\e_1)x\gamma_-(\e_1)$
spans a disk domain in $X$.}
\eeq
Let $s_{\pm}$ be as in 
Sublemma \ref{slem:mean} for $\e_1$, and set $I:=[s_-, s_+]$
and $\gamma_{\e_1}=\gamma_+([0, \e_1])\cup \gamma_-([0, \e_1])$.
It follows from the contiuity of $\sigma$, Lemma \ref{lem:noreturn}
and \eqref{eq:span-triangle} that 
$\lambda_s$ meets $\gamma_{\e_1}$ for all $s\in I$.
Now define $\varphi:I\to \gamma_{\e_1}$
by $\varphi(s)=\lambda_s\cap\gamma_{\e_1}$.
Let $\gamma_{\pm}(\e_{\pm}):=\varphi(s_{\pm})$.
Since $\varphi$ is continuous, the intermediate-value theorem
implies that 
$$
              \gamma_{\e_2} \subset {\rm Im}\,\varphi\subset S,
$$
where $\e_2:=\min\{ \e_+,\e_-\}$.

For any $0<\e\le\e_2$, let $\mu_\e$ be the $X$-geodesic
joining $\gamma_-(\e)$ to $\gamma_+(\e)$.
Put
\[
       \hat\gamma_{\e,\e_2} := \gamma_-([\e,\e_2])\cup
\mu_\e\cup\gamma_+([\e,\e_2]).
\]
Similarly, we can define the map $\psi:I\to \hat\gamma_{\e,\e_2}$ by 
 $\psi(s)=\lambda_s\cap\hat\gamma_{\e,\e_2}$.
Again, since $\psi$ is continuous, 
the intermediate-value theorem implies that
$\psi$ is surfjective, and hence  
$\mu_\e\subset S$. 
Now we can take $\e_3^+>0$ such that      
\[
     U_+(x,\e_3^+)\subset \bigcup_{0\le\e\le e_2}\mu_\e \subset S.
\]
Similarly we have
$U_-(x,\e_3^-)\subset S$ for some $\e_3^->0$. 
This completes the proof of Lemma \ref{lem:inS}.
\end{proof} 

\pmed
\subsection{Spaces of directions}
\psmall
 In this subsection, we determine the structure of the space of 
directions of $S$ at each point of $S$.

\begin{lem} \label{lem:regular-dir}
For every $x\in S$, let $\xi\in\Sigma_x(S)$ be regular in $\Sigma_x(X)$, and let $\gamma$ be an 
$X$-geodesic with $\dot\gamma(0)=\xi$. Then $\gamma([0,\e])\subset S$
for a small $\e>0$. Furthermore, $\e$ can be taken locally uniformly for $\xi$.
\end{lem}
\begin{proof}First assume $x\in\inte\,S$, and let $x=\lambda_s(t)$.
From Lemma \ref{lem:inS}, we may assume that $\angle(\xi,\pm\dot\lambda_s(t))\ge\pi/3$. If $\xi$ is positive, Lemmas \ref{lem:pass} and \ref{lem:noreturn} imply that for some $s_1> s_{\max}(x)$,
$\lambda_s$ meets $\gamma$ at, say $\gamma(t(s))$ for every $s\in [s_{\max}(x), s_1]$.
Since $\xi$ is horizontal, $t(s)$ is unique and continuous in $s$,  and $t(s)>0$ if $s>s_{\max}(x)$.
Therefore $\gamma([0, t(s_1)])\subset S$. 
From this argument, the local uniformness of $t(s_1)$ for $\xi$ is clear.
The case when $\xi$ is negative is similar. 
If $x\in\partial S\setminus \{p\}$, the proof is similar.

Finally we consider the case $x=p$.
For small enough $\e>0$, let $\xi_{\pm}\in \Sigma_p(X)$ be such that 
\[
\angle(\xi_+, \xi_-)=\angle(\xi_+,\xi)+\angle(\xi,\xi_-) =2\angle(\xi_+,\xi)=2\e,
\]
and let $\gamma_{\pm}$ be $X$-geodesics with $\dot\gamma_{\pm}(0)=\xi_{\pm}$.
For a small $\eta>0$, let $U(\eta)$ denote the 
domain bounded by $\gamma_{\pm}$ and $S(p,\eta)$.
If $\eta$ is  small enough, then $U(\eta)$
is homeomorphic to a disk.
Take $x_i\in S$ with $x_i\to p$ such that $\uparrow_p^{x_i}\to \xi$.
Then for a large $N$, we have $x_i\in U(\eta)$ for all $i\ge N$.
If $x_i=\lambda_{s_i}(t_i)$, $\lambda_{s_i}$ must meet
$\gamma_{\pm}$. The intermediate-value theorem then yields that 
the subdomain of $U(\eta)$ bounded by 
$\gamma_{\pm}$ and $\lambda_{s_N}$ is contained in $S$.
In particular $\gamma([0,\e_1])\subset S$ for small $\e_1>0$.
\end{proof}

\psmall
\begin{rem} \label{rem:singular-nonconvex} \upshape
If $\xi\in\Sigma_x(S)$ is singular in $\Sigma_x(X)$, Lemma \ref{lem:regular-dir} does not hold in general. See Examples \ref{ex:2} and \ref{ex:3}.
\end{rem}

\psmall

\begin{lem} \label{lem:S-circle}
Let $x\in S$.
\benum
 \item If $x\in {\rm int} S$, then $\Sigma_x(S)$ is a circle of length $<2\pi +\tau_p(r)\,;$ 
\item If $x\in\pa S$, then $\Sigma_x(S)$ is an arc.
\end{enumerate}
\end{lem}

\begin{proof} $(1)$
First we show that $\Sigma_x(S)$ contains a circle $C$.
Take an  $s_0\in s(x)$ and $t_0$ with $x=\lambda_{s_0}(t_0)$. Obviously
\[
C_0 := 
\{ \xi\in \Sigma_x(X)\,|\, \angle(\pm\dot\lambda_{s_0}(t_0),\xi) \le \pi/3\,\}
\]
consists of two arcs in the regular part of $\Sigma_x(X)$.
It follows  from Lemma \ref{lem:inS} that  
$C_0$ is contained in $\Sigma_x(S)$.
For a  positively horizontal direction $v_+\in \Sigma_x(S)$, 
we apply Lemma \ref{lem:basic} to obtain a minimal arc $C_+$ in $\Sigma_x(S)$ 
joining two points close to 
$\pm\dot\lambda_{s_0}(t_0)$ and containing $v_+$.
%
%$\uparrow_x^{\lambda_{s_+}(0)}$ and 
%$\uparrow_x^{\lambda_{s_+}(1)}$ 
%take an $S$-geodesic $\gamma:(-\e,\e)\to S$ with $\gamma(0)=x$ such that $\dot\gamma(0)$ is
%For a sequence $t_i\to 0$ of positive numbers, 
%take $s_i\in s(\gamma(t_i))$. We may assume that
%$s_i\to s_+\in s(x)$.
Similarly, for a  negatively horizontal direction $v_-\in \Sigma_x(S)$, 
we apply Lemma \ref{lem:basic} to obtain a minimal arc $C_-$ in $\Sigma_x(S)$ 
joining two points close to 
$\pm\dot\lambda_{s_0}(t_0)$ and containing $v_-$.
Obviously the union of $C_0$, $C_+$ and $C_-$ forms a circle 
$C$ in $\Sigma_x(S)$. 
It follows from Lemma \ref{lem:s(x)-diam}  and \eqref{eq:infinite-app} that
\[
        |L(C_{\pm})-\pi|<\tau_p(r),\,\, L(C\setminus (C_+\cup C_-))<\tau_p(r),
\]
which implies $|L(C)-2\pi|<\tau_p(r)$.

Suppose next that 
$\Sigma_x(S)\setminus C$ is not empty, and take a $w$ in $\Sigma_x(S)\setminus C$. 
Since $\angle(w, \pm\dot\lambda_{s_0}(t_0))\ge \pi/3$, 
we can apply Lemma \ref{lem:basic} to obtain a minimal arc 
$C_1$ in $\Sigma_x(S)$ 
joining two points close to 
$\pm\dot\lambda_{s_0}(t_0)$ and containing $w$.
Note that the complement $C_1'$ of a small neighborhood of 
$\pm\dot\lambda_{s_0}(t_0)$ in $C_1$ is contained in $C$,
and $w$ must be contained in $C_1'$, which is a contradiction.

$(2)$ If $x=\lambda_s(0)$ with $ 0<s<\ell$ (resp.  $s=\ell$), then 
$\Sigma_x(S)$ is an arc with endpoints $\pm\dot\alpha_1(s)$
(resp. $-\dot\alpha_1(\ell)$ and $\dot\lambda_\ell(0)$) 
through $\dot\lambda_s(0)$ (recall \eqref{eq:bdyS}).
The case $x=\lambda_s(1)$ with $ 0<s\le\ell$
is similar.
Next consider the case $x=p$.
Let $v$ and $\nu_1, \nu_2$ be as in \eqref{eq:delta}.
We show that $\Sigma_p(S)$ coincides with the arc $[\nu _1,\nu_2]$
in $\Sigma_p(X)$. 
Let $\eta_i$ be any interior point of $[\nu_i,v]$,
and let $\sigma_i$ be $X$-geodesics with $\dot\sigma_i(0)=\eta_i$.
If $s>0$ is small enough, then $\lambda_s$ meets both
$\sigma_1$ and $\sigma_2$.
This implies that $[\nu_1,\eta_1]\cup [\eta_2,\nu_2]$ is contained in 
$\Sigma_p(S)$. Letting $\eta_1, \eta_2\to v$, we obtain
that $[\nu_1,\nu_2]\subset \Sigma_p(S)$.
Conversely, for any $\xi\in\Sigma_p(S)$, take $x_i\in S$ with 
$|p,x_i|\to 0$ and $\uparrow_p^{x_i}\to \xi$. Since $x_i$ can be written as
$x_i=\lambda_{s_i}(t_i)$ with $s_i\to 0$, it is obvious that 
$\xi\in [\nu_1,\nu_2]$. Thus we have $\Sigma_p(S)=[\nu_1,\nu_2]$.
\end{proof}

\begin{defn}  \label{defn:intrinsic-direct}\upshape 
For $x\in S$, let  $\Sigma_x(S^{\rm int})$ denote the {\it intrinsic space of directions}
 of $S$ at $x$, which is defined as the completion of 
the set of all equivalence classes of $S$-geodesics starting from $x$
equipped with the upper angle $\angle^S$ for the induced interior metric of $S$.
\end{defn}

\begin{lem} \label{lem:ext=int}
$\Sigma_x(S)$ is isometric to $\Sigma_x(S^{\rm int})$.
\end{lem}

\begin{proof}  First assume $x\in \inte\, S$.
Let 
\[
    \Omega:=\Sigma_x(S) \cap {\mathcal S}(\Sigma_x(X)).
\]
Note that $|\Omega|<\infty$.
We first show that each component $\Sigma$ of $\Sigma_x(S)\setminus \Omega$
is isometrically embedded in $\Sigma_x(S^{\rm int})$. Take $\xi_1$ and $\xi_2$
from $\Sigma$ with $|\xi_1,\xi_2|<\pi$. Let $\mu_i:[0,\e]\to X$ be an $X$-geodesic with
$\dot\mu_i(0)=\xi_i$. Then for small $\e$, we have  from Lemma \ref{lem:regular-dir} 
\begin{enumerate}
 \item $\mu_i\subset S$
 \item every $X$-geodesic joining $\mu_1(t)$ and $\mu_2(t)$ is contained in $S$
  for every $t\in [0,\e]$. 
\end{enumerate}
Thus we conclude that $\angle^X(\xi_1,\xi_2)= \angle^S(\xi_1,\xi_2)$.

Next, for any $v\in\Omega$, take 
$\xi_3, \xi_4\in \Sigma_x(S)\setminus \Omega$ close to $v$ such that 
$\xi_3,  v$ and $\xi_4$ are in this order on the circle $\Sigma_x(S)$.
Take $X$-geodesics  $\gamma_i$\, $(i=3,4)$\, in the direction $\xi_i$.
By Lemma \ref{lem:pass}, we can find a sequence $s_i$ 
such that $s_i\to s_0\in s(x)$ and  $\lambda_{s_i}$ meets both 
$\gamma_3$ and $\gamma_4$.
This implies that $\angle^X(\xi_3,\xi_4)= \angle^S(\xi_3,\xi_4)$.
This completes the proof.

The case $x\in \pa S$ is similar, and  hence we omit the proof.  
\end{proof}

\pmed
\subsection{Proof of Theorem \ref{thm:ruled}} \label{ssec:monotonicity}

\pmed
 In this subsection, we  first prove Theorem \ref{thm:ruled}.
Then we control the difference between the geometries 
of $S$ and $X$.

\begin{lem} \label{lem:point-crossing}
For every $x\in S$, we have 
\begin{enumerate}
 \item $s(x)$ is either a point or a closed interval $;$
 \item $\sigma^{-1}(x)$ is a strictly monotone arc in $R$.
\end{enumerate}
\end{lem}

\begin{proof}
Suppose that the conclusion (1) does not hold. Then we would have $s_-<s_+$ such that 
$s_{\pm}\in s(x)$ and $(s_-, s_+)$ does not meet $s(x)$.
By Lemma \ref{lem:base}, we may assume 
$x\neq p$.
Choose an $S$-geodesic $\gamma:[0,a)\to S$ starting from $x$ 
such that $\dot\gamma(0)$ is  a positive, horizonal and regular direction.
Let us denote 
\begin{align*}
  I & =\{ s\in (s_-, s_+)\, |\,\lambda _s \,\, {\rm passes \,\,through}\,\, \gamma\setminus \{ x\} \,\}. %\\
 \end{align*}
From Lemmas \ref{lem:noreturn} and \ref{lem:pass}, there is  $\e_0>0$ such that 
 $(s_-, s_- +\e) \subset I$  for every $0<\e\le\e_0$.
 Since $x\in\lambda_{s_+}$, this is a contradiction to 
 Lemma \ref{lem:noreturn2}.  

(2) follows immediately from (1) and the injectivity of $\sigma|_{I_s}$ for each $s\in (0,\ell]$.
\end{proof}

\begin{proof}[Proof of Theorem \ref{thm:ruled}]
By Lemma \ref{lem:point-crossing}, we have \eqref{eq:sigma=uv} for all $u,v\in R$.
Thus $S$ has the induced metric from  $\sigma$.
Theorem \ref{thm:int=ind} then implies that 
$(S, d_S)$ is a $\CAT(\kappa)$-space.

\psmall
We set $S^{\rm int}:=(S, d_S)$.

\begin{lem} \label{lem:g-complete}
${\rm int}\, S^{\rm int}$ is locally geodesically complete.
\end{lem}

\begin{proof}  
This is immediate from Lemma \ref{lem:S-circle}
in a straightforward way.
  See \cite[Proposition II.5.12]{bridson-haefliger},   \cite[Theorem 1.5]{LSr} together with  \cite[Theorem A]{Kr:local}
 for general considerations.
\end{proof}

We prove that $S^{\rm int}$ is a topological 
two-manifold with boundary.
In view of Lemmas \ref{lem:disk}, 
\ref{lem:S-circle},  \ref{lem:ext=int}  and \ref{lem:g-complete}, it suffices to show that a 
small $S$-ball around any point $x\in\pa S$ is 
homeomorphic to a half disk. Suppose $x=p$. 
The other cases are similar. 
The argument is standard. Logically,
we proceed as follows.
For a positive integer $m$ with $m\ge [\pi/2\delta]+1$, 
gluing $m$ copies of $S$ in order around $p$, we have a sector $T$ with sector angle $\ge \pi$ at $p$,
which is a $\CAT(\kappa)$-space by Theorem \ref{thm:glue}. Glue two copies of $T$ along their edges
to obtain a $\CAT(\kappa)$-space $W$ for which 
$L(\Sigma_p(W))\ge 2\pi$. Then Lemma \ref{lem:disk}
shows that $p$ has an open disk neighborhood,
which implies that $p$ has a half-disk neighborhood
in $S$.

Finally, from the $\CAT(\kappa)$-property of $S$, 
the contractibility of $S$ is immediate since we may assume
that the diameter of $S$ for the metric $d_S$ is 
less than  $\pi/\sqrt{\kappa}$ when $\kappa>0$. 
This completes the proof of Theorem \ref{thm:ruled}.
\end{proof}

\medskip
In the rest of this section, we present a few results
that control the difference between the geometries 
of $X$ and $S$. These will be needed in Sections \ref{sec:fill} and \ref{sec:approx}.

\begin{lem} \label{lem:S-grad}
For arbitrary distinct $x,y\in S$, let $(-\nabla^S d_x)(y)$ denote 
$\dot \gamma_{y,x}^S(0)$, 
where $\gamma_{y,x}^S$ is the $S$-geodesic from $y$ to $x$.  Then we have
\begin{enumerate}
  \item $\angle(\dot\gamma^S_{x,y}(0), \dot\gamma^X_{x,y}(0)) <  \tau_x(|x,y|_X)\,;$
  \item $\angle((-\nabla^S d_x)(y), (-\nabla d_x)(y)) < \tau_x(|x,y|_X)$.
\end{enumerate}
\end{lem}

For the proof, we need a sublemma.

\begin{slem}   \label{slem:ratio-SX}
For every $x\in S$, we have 
\[
       \sup_{y\in B^S(x,s)\setminus\{ x\}} \,\frac{|x,y|_S}{|x,y|_X} < 1+\tau_x(s).
\]
\end{slem}
\pmed
When $x\in S\setminus \ca S(X)$, 
Sublemma \ref{slem:ratio-SX} and Lemma \ref{lem:S-grad}
are clear. 

\begin{proof}[Proof of Sublemma \ref{slem:ratio-SX}]
If the sublemma does not hold, there would exist a sequence $x_n$
in $S$ converging to $x$ such that 
\beq
         \frac{|x,x_n|_S}{|x,x_n|_X} >1+c >1,  \label{eq:S-irr}
\eeq
for some constant $c>0$ independent of $n$.
Passing to a subsequence, we may assume that 
$\uparrow_x^{x_n}$ converges to a direction $v\in \Sigma_x(X)$. It is easily seen 
from \eqref{eq:S-irr} that 
$v$ is a vertex of $\Sigma_x(X)$.
Take a small enough $\e>0$ compared with $c$ and an $s_n\in s(x_n)$.
Let $y_n$ be an element of $\lambda_{s_n}$ with $|x_n, y_n|=\e|x,x_n|_X$.
From Lemma \ref{lem:S-circle}, the $X$-geodesic joining $x$ and $y_n$ 
is contained in $S$ for any large $n$.
It follows from triangle inequality that 
\[
 \frac{|x,x_n|_S}{|x,x_n|_X} \le  \frac{|x,y_n|_S +|y_n, x_n|_S}{|x,x_n|_X} 
   \le \frac{|x,x_n|_X + 2|y_n, x_n|_X}{|x,x_n|_X}   
    = 1+2\e <1+c,
\]
which is a contradiction.
\end{proof}

\begin{proof}[Proof of Lemma \ref{lem:S-grad}]
If Lemma \ref{lem:S-grad} does not hold, there would be a sequence $x_i$ 
of $S$ converging to $x$ such that 
\begin{align}  
 &\angle(\dot\gamma^S_{x,x_i}(0), \dot\gamma^X_{x,x_i}(0))  > c>0, \,\,\text{or} \label{eq:XS-gradp} \\
 &\angle((-\nabla^S d_x)(x_i), (-\nabla d_x) (x_i)) > c>0,
               \label{eq:XS-gradx}
\end{align}
where $c$ is a uniform positive constant. 
From now, we assume $x\in {\rm int}\,S$.
The other case is similar.
%%%%%%%%%%%%%
We may assume that 
$\xi^X_i:=\dot\gamma^X_{x,x_i}(0)$  and 
$\xi^S_i:=\dot\gamma^S_{x,x_i}(0)$ converge to 
$\xi^X\in \Sigma_x(X)$ and 
$\xi^S\in \Sigma_x(S)\subset \Sigma_x(X)$ respectively.
Note that $\xi^X\in \Sigma_x(S)$.

(1)\, First we assume \eqref{eq:XS-gradp}.
Then we have $\angle(\xi^X,\xi^S)\ge c$.
We show 
$\xi^X\in \Sigma_x(S)\cap V(\Sigma_x(X))$.
Actually, by Lemma \ref{lem:regular-dir}, if  $\xi^X\in \Sigma_x(S)\setminus V(\Sigma_x(X))$,
we have $\e>0$ such that 
$\gamma^X_{\xi^X_i}([0,\e]) \subset S$ for any large $i$.
It turns out $\gamma^X_{x,x_i}\subset S$, which is a contradiction to \eqref{eq:XS-gradp}.
Similarly, we have $\xi^S\in \Sigma_x(S)\cap V(\Sigma_x(X))$.
Since $\Sigma_x(S)\cap V(\Sigma_x(X))$ is a point, it follows 
that $\xi^X=\xi^S$. This is a contradiction.

%%%%%%%%%%%

(2)\, Next assume \eqref{eq:XS-gradx}.
We set $\xi:=\xi^X=\xi^S$, and $\Sigma:=\Sigma_x(S)$.
From the above argument of (1) and
\eqref{eq:XS-gradx}, we have 
$\xi\in \Sigma\cap V(\Sigma_x(X))$.
Letting $t_i^X:=|x_i,x|_X$, consider the convergence
$(\frac{1}{t_i^X} X,x_i)\to (K_x(X),\xi^X)$. 
Similarly, letting $t_i^S:=|x_i,x|_S$, from Lemma \ref{lem:ext=int},
we have the convergence
$(\frac{1}{t_i^S}S,x_i)\to (K(\Sigma),\xi^S)$.
%%%%%
Let $\mu_1, \mu_2$ be elements of $\Sigma\setminus  V(\Sigma_x(X))$ near $\xi$
such that $\xi$ is in the interior of the shortest 
arc between $\mu_1$ and $\mu_2$.
%we may assume $\xi^X \notin \{ \nu_1,\nu_2\}$.
Take any $s_i\in s(x_i)$, and let  
$y_i$, $z_i$ be the intersections of $\lambda_{s_i}$ and
$\gamma_{\mu_1}$, $\gamma_{\mu_2}$ respectively.
%%%%%
Let $y_\infty\in K(\Sigma)$ and $z_\infty\in K(\Sigma)$ be the limit of $y_i$ and 
$z_i$ under the above rescaling limit respectively.
By Lemma \ref{lem:limit}, we have 
\begin{align*}
    &\limsup_{i\to\infty}\angle^X y_ix_i x\le \angle y_\infty\xi o_x, \,\, 
                         \limsup_{i\to\infty}\angle^X z_ix_i x \le \angle z_\infty\xi o_x,\\
   &\limsup_{i\to\infty}\angle^S y_ix_i x \le \angle y_\infty\xi o_x, \,\, 
                         \limsup_{i\to\infty}\angle^S z_ix_i x \le \angle z_\infty\xi o_x.
\end{align*}
It follows from 
\begin{align*}
& \angle^X y_ix_i x + \angle^X z_ix_i x\ge \pi, \,\, \angle^S y_ix_i x + \angle^S z_ix_i x\ge\pi, \\
    &\hspace{2cm}\angle y_\infty\xi o_x + \angle z_\infty\xi o_x = \pi 
\end{align*}
that
\beq \label{eq:XSyxp}
\begin{aligned}
  &\lim_{i\to\infty}\angle^X y_ix_i x=\angle y_\infty\xi o_x=                                       
                   \lim_{i\to\infty}\angle^S y_ix_i x, \\
 &\lim_{i\to\infty}\angle^X z_ix_i x=\angle z_\infty\xi o_x=                                       
                   \lim_{i\to\infty}\angle^S z_ix_i x.
\end{aligned}
\eeq
Now let $w_i\in\Sigma_{x_i}(S)$ be the  nearest point of $\Sigma_{x_i}(S)$ from $\dot\gamma^X_{x_i,x}(0)$.
If $\dot\gamma^X_{x_i,x}(0)\in \Sigma_{x_i}(S)$, then  \eqref{eq:XSyxp} implies that 
$\angle(\dot\gamma^X_{x_i,x}(0), \dot\gamma^S_{x_i,x}(0))\to 0$
as $i\to\infty$.
This is a contradiction to \eqref{eq:XS-gradx}.
Suppose $\dot\gamma^X_{x_i,x}(0)\notin \Sigma_{x_i}(S)$.
Then $w_i\in V(\Sigma_{x_i}(X))$. It follows from 
Lemma \ref{lem:near-vert} that 
$\angle(\dot\gamma^X_{x_i,x}(0), w_i)<\tau_x(|x,x_i|_X)$.
In what follows, we may assume that 
$\angle(\dot\gamma_{x_i,y_i}(0), w_i) \ge \angle(\dot\gamma_{x_i,y_i}(0), \dot\gamma^S_{x_i,x}(0))$ without loss of generality by replacing $y_i$ by $z_i$ if necessary.
Then using Lemma \ref{lem:S-circle} and \eqref{eq:XSyxp},  we obtain
\begin{align*}
& \angle(\dot\gamma^X_{x_i,x}(0), \dot\gamma^S_{x_i,x}(0)) = 
     \angle(\dot\gamma^X_{x_i,x}(0), w_i) +  \angle(w_i, \dot\gamma^S_{x_i,x}(0)) \\
& =   \angle(\dot\gamma^X_{x_i,x}(0), w_i)  + 
    \angle(w_i, \dot\gamma_{x_i,y_i}(0)) -
\angle(\dot\gamma^S_{x_i,x}(0), \dot\gamma_{x_i,y_i}(0)) \\
& \le 2\angle(\dot\gamma^X_{x_i,x}(0), w_i)  + 
    \angle(\dot\gamma^X_{x_i,x}(0), \dot\gamma_{x_i,y_i}(0)) -
\angle(\dot\gamma^S_{x_i,x}(0), \dot\gamma_{x_i,y_i}(0)) \\
& <  \tau_x(|x,x_i|_X) +o_i,
\end{align*}
where $\lim_{i\to\infty}o_i=0$.
This is a contradiction to \eqref{eq:XS-gradx}.
\end{proof}

\begin{lem} \label{lem:verticalS}
 For $x,y\in S$, suppose that the $S$-geodesic 
 $\gamma_{x,y}^S:[0,1]\to S$ from $x$ to $y$ is vertical. Then 
 $\gamma_{x,y}^S$ is an $X$-geodesic.
 \end{lem}
 \begin{proof} For any $t\in [0,1]$, let $\e>0$ be chosen as in 
 Lemma \ref{lem:inS} for $z:=\gamma_{x,y}^S(t)$.
 Choose $t_n\to t$, and set $z_n:=\gamma_{x,y}^S(t_n)$.
 Let $\gamma_n^X$ be the $X$-geodesic from $z$ to $z_n$.
 In view of Lemmas \ref{lem:S-circle} and \ref{lem:inS}, we have
 $$
         \gamma_n^X\subset U_{\pm}(z,\e)\subset S.
 $$
 Thus  $\gamma_n^X$ must be a subarc of $\gamma_{x,y}^S$, and hence
 $\gamma_{x,y}^S$ is an $X$-geodesic.
 \end{proof}

\begin{lem}  \label{lem:x-geod-general}
For  $x, y\in S$ with  $x\in\alpha_1$ and $y\in\alpha_2$ satisfying 
\[
    | |p,x|_X- |p,y|_X| < |x,y|_X/100,
\]
the $S$-geodesic joining $x$ and $y$ is an $X$-geodesic.
\end{lem}

\begin{proof}
It follows from the assumption that
\[
    | |p,x|_S- |p,y|_S| < |x,y|_X/100\le |x,y|_S/100.
\]
Using Lemma \ref{lem:jack} in $S$, we have
\[
    |\angle p\gamma_{x,y}^S(t)x -\pi/2|<\pi/3, \,\,|\angle p\gamma_{x,y}^S(t)y -\pi/2|<\pi/3,
\]
for all $t\in (0,1)$, where $\gamma_{x,y}^S:[0,1]\to S$ is the $S$-geodesic joining $x$ to $y$.
This implies that $\gamma_{x,y}^S$ is vertical.
The lemma then follows from Lemma \ref{lem:verticalS}.
\end{proof}

In a way similar to Lemma \ref{lem:x-geod-general}, we have the following.

\begin{lem}  \label{lem:x-geod-general'}
For arbitrary  $x, y\in S$ such that
\[
    | |p,x|_S- |p,y|_S| < |x,y|_S/100,
\]
the $S$-geodesic joining $x$ and $y$ is an $X$-geodesic.
\end{lem}

\setcounter{equation}{0}

\section{Filling via  ${\rm CAT}(\kappa)$-disks} \label{sec:fill}

Let $v$ be a vertex of $\Sigma_p$ of order $N$, and 
let $\nu_1,\ldots,\nu_N$ be 
the set of all points of $\Sigma_p$ with $d(\nu_i,v)=\delta$ for 
a sufficiently small positive number $\delta$. 
Take a small enough  $r>0$ and 
points $a_1,\ldots, a_N$ of $S(p,2r)$ with 
$\dot\gamma_{p,a_i}(0)=\nu_i$ and $r\le r_p$.
For simplicity, we denote by $S(a_i,a_j)$ the ruled surface 
$S(\gamma_{p,a_i},\gamma_{p,a_j})$ spanned by $\gamma_{p,a_i}$ and $\gamma_{p,a_j}$.
Let $V(\Sigma_p)$ be the set of all vertices of the 
graph $\Sigma_p$.
Since  $V(\Sigma_p)$ is finite,  
we have a positive number $r_p$ such that for any
$0<r \le r_p$, all the $S(a_i,a_j)$, when $v$ runs over 
$V(\Sigma_p)$, 
satisfy the conclusion of Theorem \ref{thm:ruled}. 
Then obviously $\mathcal S(X)\cap B(p,r)$ is contained in 
the union of all $S(a_i,a_j)$ when $v$ runs over 
$V(\Sigma_p)$.
\pmed
\n
{\bf Sector correspondence}. 
We fix $S:=S(a_i,a_j)$ for a moment, and set  
\[
       \Omega(S,r)^X :=B^X(p,r)\cap S, \,\,\,\,
        C^X:=S^X(p,r)\cap S.
\] 
From here on,  we use the symbols $B^X(p,r)$ and $S^X(p,r)$ to emphasize the metric ball and the metric circle {\it in $X$}.
Note that $\Omega(S,r)^X$ is bounded by the two geodesics
 $\gamma_{p,a_i}$, $\gamma_{p,a_j}$ and $C^X$.

\psmall 
To show Theorem \ref{thm:main}(3), we need the following lemma.

\begin{lem} \label{lem:bilip-sector}
For any small enough $r\le r_p$, 
the sector  $\Omega(S,r)^X$ is $\tau_p(r)$-almost isometric to a Euclidean sector.
\end{lem}

\begin{proof}
By Lemma \ref{lem:S-grad}, for every $x\in C^X$, we have 
$\angle(\dot\gamma^X_{x,p}(0),\dot\gamma^S_{x,p}(0))<\tau_p(r)$.
Since $\angle(\dot\gamma^X_{x,p}(0), \dot C^X)=\pi/2$, 
it follows that 
\beq \label{eq:Sgeode-C}
   |\angle(\dot\gamma^S_{x,p}(0), \dot C^X)-\pi/2|<\tau_p(r).
\eeq 
Consider the rescaling limit of the $\CAT(\kappa)$-space:
 $(\frac{1}{r} S, p)\to (K_p(S), o_p)$ as 
$r\to 0$. By Theorem \ref{thm:conv}, we have a $\tau_p(r)$-almost isometry 
$\varphi:\Omega(S,r)^X \to {\rm image}(\varphi)\subset\mathbb R^2$. 
It suffices to show that ${\rm image}(\varphi)$ is  $\tau_p(r)$-almost isometric to a Euclidean sector.
Although the argument below is elementary and standard, 
we present a proof for completeness since we do  
not find a reference.

We may assume $\varphi(p)=(0,0)=O$.
Let $L_k$ be the line segment from $O$ to
$\varphi(\gamma_k(r))$ \,$(k=1,2)$.
We express $L_k$ %and $\varphi(\gamma_k(t))$ 
in the polar coordinates as  
\[
     L_k(x)=(x, \theta_k)\, (0\le x\le x_k(r)), \,\, 
\]
where $\theta_k$ is a constant and 
$x_k(r):=|\varphi(\gamma_k(r)), O|$.
Let $\theta_0$ be the direction representing the midpoint of 
$\dot L_1(0)$ and $\dot L_2(0)$.
We may assume that 
$\theta_1 <\theta_0=0<\theta_2$, and let $L_0$ be the line 
segment from $O$ in the direction $\theta_0$ : $L_0(x)=(x,0)$.
Let $\varphi(C^X)$ intersect $L_0$ with $(r_0, 0)$.
Set $q_k:=L_k(x_k(r))$.

Let $U_k$ (resp. $D_k$) be the domain bounded by 
$L_0$, $\varphi(\gamma_k)$ and $\varphi(C^X)$
 (resp. by $L_0$, $L_k$ and $\varphi(C^X)$ ).
%%%%%%%%%%%%%%%%%%
Let $\Omega(L_1, L_2;r)$ denote the Euclidean  
sector bounded by the rays in the derections $L_1$, $L_2$  and the circle of radius $r$. 
In the first step, we deform  ${\rm image}(\varphi)=U_1\cup U_2$ to
$D_1\cup D_2$ via a $\tau_p(r)$-almost isometry.
In the second step,  we deform  $D_1\cup D_2$ to
$\Omega(L_1, L_2;r)$ via a $\tau_p(r)$-almost isometry.
%%%%%%%%%%%%%%%%%%
\pmed\n
Step 1).\,
Choose a point $q_0\in L_0$  such that 
$\angle \pi/4 \le Oq_k q_0\le \pi/3$ for $k=1,2$.
Note that $[q_k, q_0]\subset U_k$.
Let $J_k$ denote the union $[O, q_0]\cup [q_0,q_k]$.
Let $\hat U_k$ (resp. $\hat D_k$) be the domain bounded by 
$L_0$, $\varphi(\gamma_k)$ and $[q_k, q_0]$
 (resp. by $L_0$, $L_k$ and $[q_k, q_0]$).
%%%
We first show that $\hat U_k$ is $\tau_p(r)$-almost isometric
to $\hat D_k$.

Let $J_k(x)$\,$(0\le x\le L(J_k))$ be the arc-length parameter 
of $J_k$ with $J_k(0)=O$.
For every $x\in [0,L(J_k)]$, let $\zeta_k(x, s)$\, $(0\le s\le 2)$ be
the segment such that 
\begin{itemize}
 \item $\zeta_k(x,0)=J_k(x)$, $\zeta_k(x,1)\in L_k\,;$
 \item $|O,\zeta_k(x,s)|=|O,\zeta_k(x,0)|$ for all 
 $s\in [0,2]\,;$ 
 \item $s\mapsto \zeta_k(x,s)$ is proportional to arc-length.
\end{itemize}
Then $\zeta_k(x, s)$\, $(0\le x\le L(J_k), 0\le s\le 1)$
defines a parametrization of $\hat D_k$, and differentiable except at 
$x=x_0$, where $J_k(x_0)=q_0$.
Take a unique $t_k(x)\in (0,2)$ such that 
\[
      \zeta_k(x, t_k(x))\in {\rm Im} \,(\varphi_k\circ\gamma_k).
\]
Now, we define  $\psi_k: \hat U_k \to  \hat D_k$\, $(k=1,2)$ by 
\beqq
   \psi_k(\zeta_k(x,s)):= \zeta_k\left( x, \frac{s}{t_k(x)}\right).
\eeqq
Obviously  $t_k(x)$ is locally Lipschitz, and hence differentiable 
on a set $\Omega\subset [0, L(J_k)]$ with full measure since 
$\zeta_k(x,s)$ defines a locally bi-Lipschitz embedding.
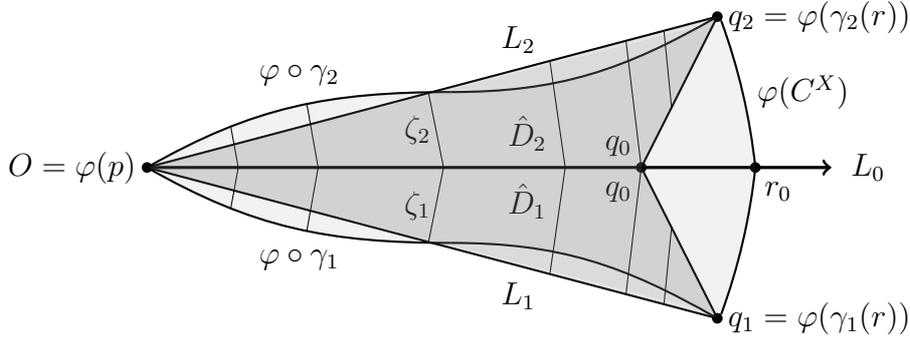
\pmed

%%%%%%%%%%%%%%%%%%%%
\begin{frame}%Page 12

\begin{center}
\begin{tikzpicture}
[scale = 1]
\draw [->, very thick] (0,0) -- (9,0);
\coordinate (O) at (0,0);
\coordinate (P1) at (7.5,-2);
\coordinate (P2) at (7.5,2);
\draw [thick] (O) -- (P1);
\draw [thick] (O) -- (P2);
\fill (0,0) coordinate (O) circle (2pt) node [left] {$O=\varphi(p)$};
\draw [-, thick] (0,0) to [out=30, in=180] (4,1);
\draw [-, thick] (4,1) to [out=0, in=210] (7.5,2);
\filldraw[fill=gray, opacity=.1] 
(0,0)  to [out=30, in=180] (4,1) to [out=0, in=210] (7.5,2)
-- (6.5,0) -- cycle;
\draw [-, thick] (0,0) to [out=330, in=180] (4,-1);
\draw [-, thick] (4,-1) to [out=0, in=150] (7.5,-2);
\filldraw[fill=gray, opacity=.1] 
(0,0)  to [out=330, in=180] (4,-1) to [out=0, in=150] (7.5,-2)
-- (6.5,0) -- cycle;
\draw (2,0.5) node[circle] [above] {$\varphi \circ \gamma_2$};
\draw (2,-0.5) node[circle] [below] {$\varphi \circ \gamma_1$};
\draw (9,0) node[circle] [right] {$L_0$};
\draw (4.4,-1.7) node[circle] [right] {$L_1$};
\draw (4.4,1.7) node[circle] [right] {$L_2$};
\draw (7.8,1) node[circle] [right] {$\varphi(C^X)$};
\draw (3.1,0.5) node[circle] [right] {$\zeta_2$};
\draw (3.1,-0.5) node[circle] [right] {$\zeta_1$};
\fill (7.5,-2) coordinate (P1) circle (2pt) node [right] {$q_1 = \varphi(\gamma_1(r))$};
\fill (7.5,2) coordinate (P2) circle (2pt) node [right] {$q_2 = \varphi(\gamma_2(r))$};
\fill (8,0) circle (2pt); 
\draw (8.3,0.1) node[circle] [below] {$r_0$};
\fill (6.5,0) circle (2pt);
\draw (6.2,0.1) node[circle] [below] {$q_0$};
\draw (6.2,0.7) node[circle] [below] {$q_0$};
\draw (5,0.1) node[circle] [below] {$\hat{D}_1$};
\draw (5,1) node[circle] [below] {$\hat{D}_2$};
\draw [-, thick] (6.5,0) -- (7.5,2);
\draw [-, thick] (6.5,0) -- (7.5,-2);
\draw [-] (6.5,0) -- (6.3,1.67);
\draw [-] (6.5,0) -- (6.3,-1.67);
\draw [-] (5.3,1.4) -- (5.5,0);
\draw [-] (5.3,-1.4) -- (5.5,0);
\draw [-] (6.8,1.8) -- (6.9,0.78);
\draw [-] (6.8,-1.8) -- (6.9,-0.78);
\draw [-] (1.2,0) -- (1.1,0.55);
\draw [-] (1.2,0) -- (1.1,-0.55);
\draw [-] (2.25,0) -- (2.1,0.85);
\draw [-] (2.25,0) -- (2.1,-0.85);
\draw [-] (3.9,0) -- (3.7,1);
\draw [-] (3.9,0) -- (3.7,-1);
\draw [thick] (8,0) arc (6:20:8.2cm);
\draw [thick] (8,0) arc (-6:-20:8.2cm);
\filldraw [fill=gray, opacity=.2] 
(0,0) -- (6.5,0) -- (7.5,2) -- cycle;
\filldraw [fill=gray, opacity=.2] 
(0,0) -- (6.5,0) -- (7.5,-2) -- cycle;
\filldraw [fill=gray, opacity=.1] 
(0,0) -- (8,0) arc (6:20:8.2cm) (7.5,2) -- cycle;
\filldraw [fill=gray, opacity=.1] 
(0,0) -- (8,0) arc (-6:-20:8.2cm) (7.5,-2) -- cycle;
\end{tikzpicture}
\end{center}

\end{frame}

%%%%%%%%%%%%%%%%%%%%

%%%%%%%%%%%%%%%%%%%%

\begin{slem} \label{slem:psi:UtoD}
Each $\psi_k:\hat U_k \to \hat D_k$ is a $\tau_p(r)$-almost isometry.
\end{slem}
%%%%%%%%%%%%%%%%%%%
\begin{proof}
In the expression  
$\psi_k(x,s):=\psi_k\circ\zeta_k(x,s)=\zeta_k(x,s/t_k(x))$,
we have on $\Omega\times [0,1]$
\begin{align}  \label{eq:zeta-partial}
\frac{\pa\psi_k}{\pa s}=\frac{1}{t_k(x)}\frac{\pa\zeta_k}{\pa s},\quad
\frac{\pa\psi_k}{\pa x}=\frac{\pa\zeta_k}{\pa x} +
     \left(\frac{-s t'_k(x)}{t_k(x)^2}\right)\frac{\pa\zeta_k}{\pa s}.
\end{align}
It is easily checked that 
\beq\label{eq:angle-zeta-st}
\left\{
\begin{aligned} 
& 0<c_1<\left|\frac{\pa\zeta_k}{\pa x} \right|<c_2,\\
& 0<c_3<\angle\left( \frac{\pa\zeta_k}{\pa s}, \frac{\pa\zeta_k}{\pa x}\right) <\pi -c_4,
\end{aligned}
\right.
\eeq
for some uniform positive constants $c_1,\ldots,  c_4$.
Note also that 
\beq  \label{eq:partial-xs}
\text{$|t_k(x)-1|<\tau_p(r)$.}
\eeq
 By the property of $\varphi$, we see that
any tangent vector to $\varphi\circ \gamma_k$ is 
$\tau_p(r)$-almost parallel to the radial direction.
Now consider the curve 
$\eta_k(x)=\zeta_k(x, t_k(x))$ 
parametrizing $\varphi\circ \gamma_k$.
It follows from the expression 
\[
\frac{d\eta_k}{dx}(x)=\frac{\pa\zeta_k}{\pa x}(x, t_k(x))+
                      \frac{\pa\zeta_k}{\pa s}(x, t_k(x))  t_k'(x),
\]
that
\beq  \label{eq:partial-s}
\text{$\left|\frac{\pa\zeta_k}{\pa s}(x,s)  t_k'(x) \right|<\tau_p(r)$}.
\eeq
Let 
\begin{align*}
       & v:=\frac{\pa\zeta_k}{\pa x}, \quad  V:=d\psi_k(v), \\
       & w:=\frac{\pa\zeta_k}{\pa s}/\left|\frac{\pa\zeta_k}{\pa s}\right|,  \quad W:=d\psi_k(w).
\end{align*}
Combining \eqref{eq:zeta-partial}, \eqref{eq:angle-zeta-st}, \eqref{eq:partial-xs} and \eqref{eq:partial-s}, 
we have 
\begin{align*}
    &  ||V|-|v||<\tau_p(r), \quad ||W|-|w||<\tau_p(r), \\
    & |\langle V, W\rangle-\langle v,w\rangle|<\tau_p(r).
\end{align*}
Together with \eqref{eq:angle-zeta-st}, this implies 
$||d\psi_k(u)| -1|<\tau_p(r)$ for each unit tangent vector $u$
on $\Omega\times [0,1]$.
This completes the proof of Sublemma \ref{slem:psi:UtoD}.
\end{proof}

Obviously, the $\tau_p(r)$-almost isometry $\psi_k:\hat U_k \to \hat D_k$ extends to a $\tau_p(r)$-almost isometry
$\psi_k: U_k \to D_k$.
 Combining $\psi_1$ and $\psi_2$, 
we obtain  
a $\tau_p(r)$-almost isometry $\psi$ between the image of 
$\varphi$ and  $D_1\cup D_2$:
\[
   \psi:{\rm Im}(\varphi)\to  D_1\cup D_2 \subset\mathbb R^2.
\] 

\pmed\n
Step 2).\,
%\pmed\n
Finally we deform $D_1\cup D_2$ to the Euclidean  
sector $\Omega(L_1, L_2;r)$. 
Let $\varphi(C^X)$ be parametrized as 
$\varphi(C^X)=(r(t), \theta(t))$
\, $0\le t\le 1$. For every $0\le r'\le r(t)$,
let us define
\[
       \phi(r', \theta(t))= \left( \frac{r}{r(t)} r', \theta(t)\right),
\]
which defines a $\tau_p(r)$-almost isometry
\[
          \phi:  D_1\cup D_2 \to\Omega(L_1, L_2;r).
\]
Thus the composition $\phi\circ\psi\circ\varphi:\Omega(S,r)^X\to\Omega(L_1, L_2;r)$
is a $\tau_p(r)$-almost isometry. This completes the proof of Lemma \ref{lem:bilip-sector}.
\end{proof}

%%%%%%%%%%%%%%%%%%%%%%%%%%%%%%%%%%%%%

\pmed
\begin{lem} \label{lem:intSBp}
$S\cap B(p,r)$ is a $\CAT(\kappa)$-space with respect to the interior 
metric.
\end{lem}
\begin{proof} It suffices to show that every point $q\in S\cap S(p,r)$
has a neighborhood $U$ in $S\cap B(p,r)$ such that any $S$-geodesic triangle
region whose vertices are in $U$ is contained in $S\cap B(p,r)$.
To achieve this, we only have to show that $S\cap B(p,r)$ is boundary convex,
in the sense that  for arbitrary $x,y\in S\cap S(p,r)$,
any $S$-minimal geodesic $\gamma_{x,y}^S$ joining them is contained in 
$S\cap B(p,r)$. We may assume that $\gamma_{x,y}^S$ is vertical, and therefore
it is an $X$-geodesic (see also Lemma \ref{lem:x-geod-general2}).
Hence the conclusion follows from the
$X$-convexity of $B(p,r)$.
\end{proof}

\pbig\n
{\bf Filling ball}.\, 
Now we fill the ball $B(p,r)$ via properly embedded/branched immersed $\CAT(\kappa)$-disks.
For a vertex $v$ of $\Sigma_p$ of order $N$, let $\nu_1,\ldots,\nu_N$  and 
$a_1,\ldots, a_N$  be as in the beginning of Section  \ref{sec:fill}.
%in Section \ref{sec:sing}.
%
For every pair $(i,j)$ with $1\le i< j\le N$, 
we want to take a simple loop in $\Sigma_p(X)$ passing through 
$\nu_i$, $v$ and $\nu_j$.
Since this is not possible in general,  we consider the two cases.

\par\bigskip\n
Case I.\quad  There is a simple loop $\zeta$ in $\Sigma_p(X)$ through $\nu_i$, $v$ and $\nu_j$.
\par\bigskip

Consider the ruled surface $S(a_i,a_j)$ as well as the 
other ruled surfaces defined around other points of $\zeta$ which are
vertices of $\Sigma_p(X)$ (if they exist). By Lemma \ref{lem:bilip-sector},
considering the regular part of $\zeta$ as well,  we can define
a proper Lipschitz embedding $f_{ij}^v:D^2(\ell;r)\to B(p,r)$ with $f_{ij}^v(O)=p$ satisfying 
$\Sigma_p({\rm Im}(f_{ij}^v))=\zeta$, where $\ell$ is the length of $\zeta$.

\begin{proof}[Proof of Theorem \ref{thm:main}(1)
for embedded disks]
Lemma \ref{lem:intSBp} together with the gluing procedure 
as discussed after Lemma \ref{lem:g-complete}
implies 
that ${\rm Im} (f^v_{ij})$ is a 
$\CAT(\kappa)$-space. Note that $f_{ij}^v$ has bi-Lipschitz constant $<1+\tau_p(r)$.
\end{proof}

\par\bigskip\n
Case II.\quad  There are no simple loops in $\Sigma_p(X)$ containing $\nu_i$, $v$ and $\nu_j$.
\par\bigskip

\begin{claim}
There is an immersion $g:S^1\to \Sigma_p(X)$ such that
\benum
 \item if $W$ is the set of multiple points of $g$, then 
 $g^{-1}(W)$ consists of two arcs $W_1$, $W_2$
  (they may be points), and each 
 restriction $g|_{W_a}:W_a\to W$ \,$(a=1,2)$ is injective $;$
 \item there is an arc $I$ of $S^1$ such that $g(I)$ coincides with the arc between $\nu_i$ and $\nu_j$ containing $v$.
\eenum
\end{claim}
\begin{proof} In view of the present case, there are non-contractible loops $C_i$ and $C_j$  at $v$, freely homotopic to a circle, such that  
$\nu_i\in C_i$, $\nu_j\in C_j$, $\nu_j\notin C_i$, $\nu_i\notin C_j$.
If both $C_i$ and $C_j$ are simple, we can define
a desired immersion $g:S^1\to \Sigma_p(X)$ with $W=\{ v\}$.
Suppose $C_i$ is not simple. Then $C_i$ contains a 
simple loop $\tilde C_i$ at a point $u_i$ such that 
 $C_i$ is the union of $\tilde C_i$ and the arc $[v,u_i]$.
 If $C_j$ is also not simple, then we consider the union of 
 simple loops $\tilde C_i$, $\tilde C_j$ and the arc $[u_i, u_j]$.
 If only $C_i$ is not simple, then we consider the union of 
 simple loops $\tilde C_i$, $C_j$ and the arc $[u_i, v]$.
 This observation provides a desired immersion $g:S^1\to \Sigma_p(X)$ with $W=[u_i,u_j]$ or $W=[u_i, v]$. 
\end{proof}  
First suppose $W=\{ v\}$ and find $\nu_k\in C_i$ and 
$\nu_{\ell} \in C_j$, $1\le k, \ell\le N$, $k, \ell \neq i,j$.
Chasing on $g(I)$ in the order 
$$
\nu_i \to v \to \nu_j \to \nu_\ell \to v \to \nu_{k} \to \nu_i,
$$
we consider  the ruled surfaces $S(a_i,a_j)$, $S(a_k,a_{\ell})$ 
as well as the other ruled surfaces defined around other points of $g(I)$ which are vertices of $\Sigma_p(X)$ (if they exist).
 By Lemma \ref{lem:bilip-sector}, considering the regular part of $g(I)$ as well, we can define
a proper Lipschitz immersion $f_{ij}^v:D^2(\ell;r)\to B(p,r)$ with
branched point $(f_{ij}^v)^{-1}(p)=\{ O\}$
satisfying 
$\Sigma_p({\rm Im}(f_{ij}^v))=g(S^1)$,
 in a way similar to Case I.
Note that any multiple point $q\in {\rm Im} f_{ij}^v$
lies in a direction close to $v$.

Next suppose $W=[u_i, v]$ and find 
$\nu_{\ell} \in C_j$ with $1\le \ell\le N$, $\ell \neq i,j$.
Chasing on $g(I)$ in the order 
$$
\nu_i \to v \to \nu_j  \to \nu_\ell  \to v \to \nu_i
\to u_i \to \tilde C_i \to \nu_i,
$$ 
we similarly consider  the ruled surfaces $S(a_i,a_j)$, $S(a_j,a_{\ell})$ 
as well as the other ruled surfaces defined around other points of $g(I)$ which are vertices of $\Sigma_p(X)$ (if they exist).
In a way similar to the previous case, we can define
a desired proper Lipschitz immersion $f_{ij}^v:D^2(\ell;r)\to B(p,r)$ branched at the point $(f_{ij}^v)^{-1}(p)=\{ O\}$
 satisfying 
$\Sigma_p({\rm Im}(f_{ij}^v))=g(S^1)$.

The other case is similar, and hence omitted. 

\pmed
Note that $f_{ij}^v$ has bi-Lipschitz constant (resp. local bi-Lipschitz constant
except the origin) $<1+\tau_p(r)$ in Case I (resp. in Case II).
\pmed

\begin{lem} \label{lem:coincide}
\[
    B(p,r) = \bigcup_{v\in V(\Sigma_p(X))}\, \left( \bigcup_{1\le i<j\le N} {\rm Im}\, f_{ij}^v      \right).
\]
\end{lem}

\begin{proof} 
First note that from construction, $\Sigma_p(X)$ coincides with all the union of 
$\Sigma_p({\rm Im}\, f_{ij}^v)$.     
Suppose there is a point 
$x \in  B(p,r)$ which is not contained in any image ${\rm Im}\, f_{ij}^v$.
Let $\xi:=\uparrow_p^x$. Take some ${\rm Im}\, f_{ij}^v$
such that $\xi\in \Sigma_p({\rm Im}\, f_{ij}^v)$.
We may assume that $\xi$ is close to the vertex $v$, since
if $\xi$ is far from any vertex of $\Sigma_p(X)$, then 
$x=\gamma_\xi(|p,x|_X)$ is certainly contained in the union of all 
the images ${\rm Im}\, f_{ij}^v$, which is a contradiction.

Let $\gamma$ be a geodesic in ${\rm Im}\, f_{ij}^v$ starting from $p$ 
in the direction $\xi$.
Note that $\gamma$ reaches the metric sphere $S(p, r)$
(see also Sublemma \ref{slem:ratio-SX}). 
Let $x'$ be the point of $\gamma_{\xi}$
such that $|p,x'|_X=|p,x|_X$.
Consider the geodesic $\gamma_{x,x'}^X$.
If we extend $\gamma_{x,x'}^X$ through $x'$, it meets 
$\gamma_{p,a_k}$ for some $k$. Similarly,
if we extend $\gamma_{x,x'}^X$ through $x$, it meets 
$\gamma_{p,a_{\ell}}$ for some $\ell$.
Lemma \ref{lem:x-geod-general} yields that $x\in S(a_k, a_j)$.
This is a contradiction.
\end{proof}

Combining Lemma \ref{lem:coincide} and the above discussion, we complete the proof of Theorem \ref{thm:main}(1), (2), (3) except (1) for the branched 
immersed disks that occur from the above Case II.

The proof of Theorem \ref{thm:main}(1) 
for the branched 
immersed disks is deferred to Section \ref{sec:approx}.

%.

%\graph.tex 
\setcounter{equation}{0}

\section{Graph structure of singular set} \label{sec:graph}

Our next step is to characterize $\mathcal S(X)\cap B(p,r)$ 
as a union of finitely many Lipschitz curves.

For a subset $A$ of $X$, we denote by 
$\partial A$ the complement in $\bar A$ of the set of all points 
$a$ of $A$ such that there is a neighborhood of $a$ 
homeomorphic to an open disk and contained in $A$.
%
%From Theorem \ref{thm:ruled} and Lemma \ref{lem:coincide},
%we immediately have the following corollary.
%
%\begin{cor} \label{cor:sing}
%$\mathcal S(X)\cap B(p,r)$ is the union of all 
%$\partial (S(a_i,a_j)\cap S(a_k,a_{\ell}))\cap B(p,r)$ when 
%the vertex $v$, {\cred $a_i, a_j, a_k, a_\ell$} run over all the possibilities
%{\cred with $(i,j)\neq (k,\ell)$}.
%\end{cor}

For distinct $1\le i, j, k\le N$, we set 
\begin{equation*}
   C_{ij;k} := (\partial (S(a_i,a_j) - S(a_j,a_k))-\partial S(a_i,a_j))\cap
                 B(p,r).
\end{equation*}

\begin{lem} \label{lem:sing-arc}
$C_{ij;k}$ is a simple Lipschitz arc in $\mathcal S(X)$
such that 
\begin{enumerate}
 \item it starts from $p$ and reaches a point of $\partial B(p,r)$;
 \item its length is less than $(1+\tau_p(r))r$;
 \item each point of $\Sigma_x(C_{ij;k})$ is a vertex of 
 $\Sigma_x(X)$ for every $x\in C_{ij;k}$.  In particular, $C_{ij;k}$ has definite directions everywhere, and $\frac{|d_p(x)-d_p(y)|}{|x,y|_X}\ge 1-\tau_p(r)$ for all $x,y\in C_{ij;k}$.

\end{enumerate}
\end{lem}

\begin{proof}
For each $s\in [0, 2r]$, consider the ruling geodesic
$\lambda_s(t)$\,$(0\le t\le 1)$ of $S(a_k, a_j)$
joining $\gamma_{p,a_k}(s)$ to $\gamma_{p,a_j}(s)$ in $X$.
Let $t_0\in (0,1)$ be the first parameter at which $\lambda_s$ meet 
$S(a_i,a_j)$.

We claim that 
\begin{align} \label{eq:ruled-coincide}
   \lambda_s([t_0,1])\subset S(a_i,a_j).
\end{align}
Since $z_s:=\lambda_s(t_0)$ is a topological singular point of $X$,
by Lemma \ref{lem:S-circle}, we can take a direction $\xi_0\in\Sigma_{z_s}(S(a_i,a_j))$
with $\angle(\xi_0, \lambda_s'(t_0))=\pi$. A geodesic $\gamma_{\xi_0}$ in 
$S(a_i,a_j)$ with direction $\xi_0$ reaches $\gamma_{p,a_i}$. 
Take $\xi_1\in\Sigma_{z_s}(S(a_i,a_j))$ with 
$\angle(\xi_0,\xi_1)=\pi$. Similarly a geodesic $\gamma_{\xi_1}$ in 
$S(a_i,a_j)$ with direction $\xi_1$ reaches $\gamma_{p,a_j}$. 
It follows from Lemma \ref{lem:x-geod-general'}
that $\gamma_{\xi_0}$ and $\gamma_{\xi_1}$ form a geodesic
in $X$. In particular, $\gamma_{\xi_0}$  is a geodesic
in $X$, and therefore $\gamma_{\xi_0}$ and
$\lambda_s([t_0,1])$ form  a geodesic, say $\gamma$, in $X$, 
 Lemma \ref{lem:x-geod-general} implies that $\gamma$ is  contained in $S(a_i,a_j)$,
and so is $\lambda_s([t_0,1])$.

Since the curve $c(s):=z_s$ is continuous, its image coincides with $C_{ij;k}$.
By Corollary \ref{cor:vert}, we have 
\[
      \angle((\nabla d_p)(c(s)), \Sigma_{c(s),+}(C_{ij;k})) < \tau_p(r), \,\,
      \angle((-\nabla d_p)(c(s)), \Sigma_{c(s),-}(C_{ij;k})) < \tau_p(r),
\]
where $\Sigma_{c(s),+}(C_{ij;k})$ (resp. $\Sigma_{c(s),-}(C_{ij;k})$) denote 
the space of directions 
of $C_{ij;k}$ at $c(s)$ in the positive direction (resp. negative direction).
%{\cred In particular, $d_p$ is strictly increasing along
%$c(s)$.}

Now we take another parametrization $\varphi(s)$ of $C_{ij;k}$ defined as 
$\varphi(s)=C_{ij;k}\cap S(p, s)$, where $S(p,s)$ denotes the metric circle
of radius $s$ with respect to $d_X$.
If $s'$ is close enough to $s$, then we have 
$\angle(\uparrow_{\varphi(s)}^{\varphi(s')}, \nabla d_p(\varphi(s)))< \tau_p(r)$,
which implies that 
\begin{align}\label{eq:C-lipschitz}\displaystyle{\lim_{s'\to s}\,\frac{|\varphi(s), \varphi(s')|_X}{|s-s'|} \le 1+\tau_p(r)}.
\end{align}
Thus $\varphi$ is  Lipschitz 
with  Lipschitz constant $\le 1+\tau_p(r)$,
and therefore having length $L(\varphi)=L(C_{ij;k})\le (1+\tau_p(r))r$. \eqref{eq:C-lipschitz} also implies 
the inequality in (3).
\end{proof}

Lemma \ref{lem:sing-arc} claims that the closure of 
$S(a_j,a_k)-S(a_i,a_j)$ ``transversally'' intersects $S(a_i,a_j)$
with the Lipschitz curve $C_{ij;k}$.
In particular we have

\begin{lem} \label{lem:Cijk}
$C_{ij;k} = C_{ji;k} = C_{jk;i}$.
\end{lem} 

In view of Lemma \ref{lem:Cijk}, we use the notation
\[
      C_{ijk}:=C_{ij;k}.
\]

Using the discussion in the proof of Lemma \ref{lem:sing-arc}, we show the following refined version of Lemma \ref{lem:x-geod-general},  which is not a direct consequence of Lemma 
\ref{lem:x-geod-general'}.

\begin{lem}  \label{lem:x-geod-general2}
For arbitrary  $x, y\in S=S(a_i,a_j)$ such that
\[
    | |p,x|_X- |p,y|_X| < |x,y|_X/1000,
\]
the $X$-geodesic joining $x$ and $y$ is an $S$-geodesic.
\end{lem}
\begin{proof}
Consider the geodesic $\gamma_{x,y}^X$ and extend it in the both directions
until it reaches $\gamma_{p, a_k}$ and $\gamma_{p, a_{\ell}}$
for some $k, \ell$ at $w_k\in \gamma_{p, a_k}$ and $w_{\ell}\in \gamma_{p, a_{\ell}}$ 
respectively. That is, 
\[
 [w_k, w_{\ell}]_X=[w_k, x]_X\cup [x,y]_X\cup [y, w_{\ell}]_X.
\]
Let $z$ (resp. $u$) be the first point at which $[w_k,x]$ (resp. $[w_{\ell}, y]$)
 meets $R(a_i, a_j)$.
As in the proof of Lemma \ref{lem:sing-arc}, we have points $w_i\in\gamma_{p,a_i}$
and $w_j\in\gamma_{p,a_j}$
such that 
\[
     [w_i, w_{j}]_X=[w_i, z]_X\cup [z, x]_X\cup [x,y]_X\cup [y, w_{j}]_X.
\]
From the hypothesis, we have $||p,w_i|-|p,w_j||<|w_i,w_j|_X/100$.
Lemma \ref{lem:x-geod-general} then implies that $[w_i, w_{j}]_X$ is an $S$-geodesic.
Thus we conclude that $[x,y]_X$ is an $S$-geodesic as required.
\end{proof}

For a vertex $v$ of $\Sigma_p(X)$, 
let $a_1,\ldots,a_N\in S(p,2r)$ be as in Section \ref{sec:fill},
where $N=N_v$.
Let $S(a_1,\ldots,a_N;r)$ be the closed domain of $B(p, r)$ bounded by
$\gamma_{p,a_i}$\,$(1\le i\le N)$, and $S(p,r)$.
Note that $S(a_1,\ldots,a_N;r)$ is 
the union of all the ruled surfaces $S(a_i,a_j)$ and $B(p, r)$.

\begin{cor} \label{cor:sing-total}
For a vertex $v$ of $\Sigma_p(X)$, the union of all 
$C_{ijk}$ as above coincides with $\mathcal S(X)\cap S(a_1,\ldots,a_N;r)$.
\end{cor}

\begin{proof}
Since every element of $\mathcal S(X)\cap S(a_1,\ldots,a_N;r)$ comes 
from the intersection of some $S(a_i,a_j)$ and $S(a_k,a_{\ell})$, 
it suffices to show that 
$\partial (S(a_i,a_j)\cap S(a_k,a_{\ell}))\setminus S(p,r)$ is contained in 
$C_{ijk}\cup C_{ij\ell}$.
For every $x\in \partial(S(a_i,a_j)\cap S(a_k,a_{\ell}))$, 
take an $s$ such that the ruling geodesic $\lambda_s$ joining $\gamma_{p,a_i}(s)$ to 
$\gamma_{p,a_j}(s)$ goes through $x$, say at $\lambda_s(t_0)=x$.
Since $x\in \mathcal S(X)$, Lemma \ref{lem:base}(2),  
Theorem \ref{thm:ruled} and Corollary \ref{cor:vert} imply 
the existence of a direction 
$\xi\in\Sigma_x(S(a_k,a_{\ell}))$ such that 
$\angle(\xi,\lambda_s'(t_0))=\pi$.
Then a geodesic $\gamma_{\xi}$ in $S(a_k,a_{\ell})$ with direction $\xi$ 
must reach $\gamma_{p,a_k}$ or $\gamma_{p,a_{\ell}}$. 
Suppose it reaches $\gamma_{p,a_k}$ for instance:   
$\gamma_{\xi}(t_1)\in \gamma_{p,a_k}$ for some $t_1>0$.
An argument similar to that in the proof of Lemma \ref{lem:sing-arc}
then implies that $\gamma_{\xi}([0,t_0])$ does not meet $S(a_i,a_j)$
except for $x$, and that the union  $\gamma_{\xi}([0,t_0])\cup\lambda_s([t_0,1])$
forms a geodesic in $S(a_k,a_j)$. This shows $x\in C_{ijk}$.
\end{proof}

\begin{proof}[Proof of the second half of Theorem \ref{thm:main}]
It is now an immediate consequence of Lemma \ref{lem:sing-arc} and 
Corollary \ref{cor:sing-total}.
\end{proof}

We call a curve $C$ in $\ca S(X)$ {\it a singular curve}. 

\begin{rem} \upshape
Each singular curve $C$ contained in $S(a_1,\ldots,a_N;r)$ has the direction $v$ at $p$. 
From now on, we always consider the case when 
$d_p$ is strictly increasing along $C$. In that case,
for each interior point $q$ of $C$, $C$ has definite directions $\Sigma_q(C)$ consisting of  
two vertices of $\Sigma_q(X)$.
\end{rem}

\pmed\n
{\bf Structure of metric circles.}
Next we discuss the structure of $S(p,r)$.
\psmall
Let $b_i:=\gamma_{p,a_i}(r)$. For $0<t\le r$, set 
\[
    S(v;t):=\left(  \bigcup_{1\le i<j\le N} S(a_i, a_j)    \right)  \cap S^X(p,t).
\]
\psmall

\begin{lem}\label{lem:circle-tree}
For each $0<t\le r$, $S(v;t)$ is a tree with endpoints 
$\gamma_{p,a_i}(t)$\, $(1\le i\le N)$.
\end{lem}
\begin{proof}
For $3\le k\le N$, put 
\[
    S_k(v;t):=\left(  \bigcup_{1\le i<j\le k} S(a_i, a_j)    \right)  \cap S^X(p,t).
\]
Inductively we show that $S_k(v;t)$ is a tree for every $3\le k \le N$.  This is certainly true for $k=3$.
Assume that $S_{k-1}(v;t)$ is a tree. 
Set 
\[
      S(a_i,a_k)(t):= S(a_i,a_k)\cap S^X(p,t).
\]
Let $p_k(t):=\gamma_{p,a_k}(t)$.
Let $q(t)$ be the point of $S_{k-1}(v,t)$ where 
the arc starting from $p_k(t)$ in $S_k(v,t)$ first meets
$S_{k-1}(v,t)$. For every $1\le i \neq j\le k-1$ with 
$q(t)\in S(a_i, a_j)(t)$, \eqref{eq:ruled-coincide} implies that 
\begin{align}  \label{eq:circ-tree}
       S(a_i, a_k)(t) \setminus [p_k(t),q(t)] \subset S(a_i, a_j),
\end{align}
where $[p_k(t),q(t)]$ denotes the arc between $p_k(t)$ and 
$q(t)$ in $S_k(v,t)$. \eqref{eq:circ-tree} implies that 
$S_k(v,t)=S_{k-1}(v,t)\cup  [p_k(t),q(t)]$. Thus $S_k(v,t)$ is a tree.
\end{proof}

\pmed
\begin{proof}[Proof of Corollary \ref{cor:circle}]
Let $r_p\ge r >0$ be as in Theorem \ref{thm:main}.
%Consider the ruled surfaces $S(a_i,a_j)$, $1\le i,j\le N$,
%contained in $\mathcal R(v)$ for a vertex $v$ of $\Sigma_p(X)$.
By Lemma \ref{lem:circle-tree}, for every vertex $v$ of $\Sigma_p(X)$, $S(v;r)$ 
is a tree with endpoints $b_i$\, $(1\le i\le N)$.
Therefore $S(p,r)$ has the same homotopy type as $\Sigma_p(X)$.\par

Let $\alpha$ be any non-contractible simple closed loop in 
$S(p,r)$. From the discussion  in Case I and the proof of Theorem \ref{thm:main}(1) in Section \ref{sec:fill},
there is a non-contractible simple closed loop $\zeta$ in $\Sigma_p(X)$
of length, say $\ell\ge 2\pi$, 
and a properly embedded $\CAT(\kappa)$-disk
$f:D^2(\ell;r)\to B(p,r)$ associated with $\zeta$ such that
$\Sigma_p({\rm Im}(f))=\zeta$ and 
$f(\pa D^2(\ell;r))=\alpha$. 
For $v\in \zeta\cap V(\Sigma_p(X))$, let $\xi_i, \xi_j\in\zeta$
be points  nearby $v$ such that $v$ is the midpoint of the arc $[\xi_i,\xi_j]$.
Let $S_{ij}$ be the ruled surface defined by $\xi_i,\xi_j$.
Note that $B^{S_{ij}}(p,r)\subset S_{ij}\cap B^X(p,r)$.
Since we may assume $r<\pi/2\sqrt{\kappa}$,
the nearest point map 
$S_{ij}\cap S^X(p,r) \to S^{S_{ij}}(p,r)$
is distance non-increasing.
Let $\tilde S_{ij}$ be a sector in the model $M^2_\kappa$ with vertex $\tilde p$
bounded by two geodesics of length $r$ and $S(\tilde p,r)$
such that the sector angle at $\tilde p$ is equal to 
$\angle(\xi_1,\xi_2)$. From the curvature condition, we have
\[
  L(S^{S_{ij}}(p,r))\ge L(\tilde S_{ij}\cap S(\tilde p,r))
          = \angle(\xi_1,\xi_2)/\sqrt{\mu(\kappa,r)},
\]
yielding $L(S_{ij}\cap S^X(p,r))\ge  \angle(\xi_1,\xi_2)/\sqrt{\mu(\kappa,r)}$.
Applying a similar argument to the other parts of 
$\alpha$ and $\zeta$, we conclude that 
\[
   L(\alpha) \ge L(\zeta)/\sqrt{\mu(\kappa,r)}
        \ge 2\pi/\sqrt{\mu(\kappa,r)}.
\]
This completes the proof.
\end{proof} 

%%%%%%%%%%%%%%%%%%%%%%%%%%%%%%%%%%%%%%
% sturcture of $\mathcal S_{\rm top}(X)\cap B(p,r)$
\medskip
Now we define a metric graph structure of $\ca S(X)$ in a generalized sense as follows.
\begin{defn} \label{defn:metric-graph}\upshape
We consider the relative topology of $\ca S(X)$ with length metric. Let $I$ be an open set of $\ca S(X)$.
We call $I$ an {\it open arc} in $\ca S(X)$ if it is open in $\ca S(X)$ and is isometric to 
an open interval.
A maximal open arc $I$ with respect to the inclusion is called an {\it open edge} of $\ca S(X)$.
We denote by $E(\ca S(X))$ (resp. $|E(\ca S(X))|$) the set 
(resp. the union) of all open edges in $\ca S(X)$. 
We call each element of $\ca S(X)\setminus |E(\ca S(X))|$
a {\it vertex} of $\ca S(X)$.
We denote by $V(\ca S(X))$  the set of all vertices of $\ca S(X)$. 
Let us denote by $V_*(\ca S(X))\subset V(\ca S(X))$  the set of all accumulation points of $V(\ca S(X))$.
The case $V_*(\ca S(X))=V(\ca S(X))$ or $\ca H^1(V_*(\ca S(X)))>0$ may happen (see Example \ref{ex:uncountable}).
As usual, two vertices $v_1$ and $v_2$ of $\ca S(X)$ are 
{\it adjacent} if there is at least one open edge joining them.
The {\it order} of a vertex $v$ is defined as the limit of the 
number of components of 
$B^{\ca S(X)}(v,\epsilon)\setminus \{ v\}$ as $\epsilon\to 0$.
\end{defn}

\begin{proof}[Proof of Corollary \ref{cor:graph}]
First note that by Theorem \ref{thm:main}, 
$\ca S(X)$ is locally path-connected.
%%%
For a given point $p\in \ca S(X)$ and $v\in V(\Sigma_p(X))$, 
let $N=N_v$ be the 
branching number of $v$ in $\Sigma_p(X)$, and
take $r=r_p$ as in Theorem \ref{thm:main}.
For $\delta>0$ with $\delta\ll \min\{ \angle (v,v')\,|v\neq v'\in V(\Sigma_p(X))\}$, let  
$\gamma_1, \ldots, \gamma_{N}$ be geodesics
from $p$ with $\angle(\dot\gamma_i(0), v) = \delta$ and
$\angle(\dot\gamma_i(0), \dot\gamma_j(0)) = 2\delta$ for 
$1\le i\neq j\le N_v$.
By Corollary \ref{cor:sing-total}, we have 
\[
       \ca S(X)\cap U(v)=\bigcup_{1\le i<j<k\le N} C_{ijk},
\]
where $U(v):=C(v,\delta,r)$ is the cone neighborhood  around $v$ (see \eqref{eq:cone-neighborhood}).
%%%
By Lemma \ref{lem:sing-arc} (3), the distance 
function $d_p$ is strictly monotone on each $C_{ijk}$. It follows from Corollary \ref{cor:sing-total} 
that  $V(\ca S(X))$ has locally finite order.

In what follows, we give an explicit sharp bound on the orders
at the vertices in $\ca S(X)\cap B(p,r)$. 

\begin{slem} \label{slem:number-C}
$\ca S(X)\cap U(v)$ can be written as the union of
at most $N_v-2$ singular curves $C$
starting 
from $p$ in the direction $v$, and reaching $S(p,r)$
such that $d_p$ is strictly increasing along $C$.
\end{slem}
\begin{proof}
By Lemma \ref{cor:sing-total}, $\ca S(X)\cap U(v)$
coincides with the set of all topological singular
points resulting from the intersections of distinct
ruled surfaces $S_{ij}$ and $S_{i'j'}$ for all
$1\le i<j\le N$, $1\le i'<j'\le N$ with 
$(i,j)\neq (i',j')$. 
%%%%%
For $2\le k\le N$, let $E_k$ be the union of all $S_{ij}$ 
with $1\le i<j\le k$.
We inductively define singular curves $C_j$\, $(2\le j\le N-1)$ as the set of all points of $E_j$ where   
geodesics almost perpendicularly starting from points of 
$\gamma_{j+1}$ intersect $E_j$ for the first time. 
Then it is obvious to see that  
$\ca S(X)\cap U(v)=C_2\cup\cdots\cup C_{N-1}$.
From Lemma \ref{lem:sing-arc}, $d_p$ is strictly increasing along $C$.
\end{proof}

Let $\Gamma:=\ca S(X)\cap B(p,r)$.
It follows from Sublemma \ref{slem:number-C}
that 
\begin{itemize}
   \item the order at the vertex $p$ of the graph $\Gamma\cap U(v)$ is at most $N_v-2\,;$
   \item the order at any vertex $y$ in $\Gamma\cap U(v)\setminus \{ p\}$
    is at most $2(N_v-2)$.
\end{itemize}
Therefore the maximum of orders of vertices contained in 
$\Gamma$ is at most
\[
    \max\left\{   \sum_{v\in V(\Sigma_p(X))} (N_v-2),
           \max_{v\in V(\Sigma_p(X))} 2(N_v-2)\right\}.
\]
This completes the proof of Corollary \ref{cor:graph}.
\end{proof} 

%%%%%%%%%　Example %%%%%%%%%%%%%%%%%%%%%%%
\medskip
We exhibit the following example, which is  another version of Example \ref{ex:2}.   Here  we use the notion of $\e$-Cantor set (cf. \cite{AlBu})
to produce a two-dimensional $\CAT(0)$-space $X$
such that $V_*(\ca S(X))$ is one-dimensional.
Similar construction for a boundary singular set of a limit space
of manifolds with boundary was made in \cite{YZ:conv}.

\begin{ex}  \label{ex:uncountable}\upshape
For any $0<\e<1$, set $\delta:=1-\e$.
We define the so-called $\e$-Cantor set of $[0,1]$ inductively 
as follows: 
We start with $I_0 := [0, 1]$, 
and remove from $I_0$ the open interval of length $\delta/2$ around the center of $I_0$ .
We denote by $I_1$ be the rusult of this removing.
Note that $I_1$ consists of $2^1$ disjoint closed intervals $I_{1,j}$ \, $(j=1,2)$ having the same length 
%with $L(I_{1,j}) =1/2 -\delta/4\, \,(j=1,2)$.
and that $L(I_1) =1 -\delta/2$.
Suppose that we have constructed $I_k$ consisting of 
$2^k$ disjoint closed intervals $I_{k,j}$ \,$(1\le j\le 2^k)$  of the same length such that 
$L(I_k) = 1-\delta/2-\cdots -\delta/2^k$.
Remove from each $I_{k,j}$ the open interval of length $\delta/2^{2k+1}$ around the center of $I_{k,j}$.
We denote by $I_{k+1}$ the result of this removing.
Thus, inductively we have constructed $I_{n}$ for every $n$.
Finally we set
\[
         I_{\infty}:=\bigcap_{n=0}^{\infty}\, I_n, \quad 
          J_{n}:=[0,1]\setminus I_n,\quad
         J_{\infty}:=\bigcup_{n=0}^{\infty}\, J_n=[0,1]\setminus I_\infty.
\]
Note that $\ca H^1(I_{\infty}) =\lim_{n\to\infty} L(I_n)=1-\delta=\e$. The set $I_{\infty}$ is called an $\e$-Cantor set.

Next, inductively we define smooth functions $f_n:\mathbb R\to [0,1]$ \,$(n\in\mathbb N)$ such that 
\begin{itemize}
\item $\supp (f_n) = J_n;$
\item $f_n=f_{n-1}$ on $J_{n-1}\,;$
\item if we set $\hat J_n^\pm:=\{ (x, \pm f_n(x))\,|\, x\in J_n\}$, then the length $\ell_n$ 
and the maximum $\kappa_n$ of absolute geodesic curvature of $\hat J_n^\pm$
satisfy \eqref{eq:sum-ell-kappa}.
\end{itemize}
Now we define the limit $f:=\lim_{n\to\infty} f_n:\mathbb R\to [0,1]$, which satisfies
$\supp (f) = J_\infty$.
Using $f$, we define the closed subset $\Omega$ of $\mathbb R^2$ by 
$$
    \Omega:= \{ (x,y)\,|\,|y|\le f(x), x\in\R\},
$$
equipped with the length metric. Set
$$
    \pa_\pm \Omega:= \{ (x,y)\,|\,y=\pm f(x), x\in\R\}.
$$
%%%%%%%%%%%%%%%%%
Take closed concave domains $H_\pm$ in $\R^2$ homeomorphic to the half plane
such that for certain isometries $g_\pm:\pa_\pm \Omega\to \pa H_\pm$
the absolute geodesic curvature of $g_\pm(J_n^\pm)$ is greater than
$\kappa_n$.
Take two copies $H_\pm^1, H_\pm^2$ of $H_\pm$, and make a gluing 
of $H_+^1, H_+^2, H_-^1, H_-^2$ and $\Omega$ along their boundaries via $g_\pm$
as in Example \ref{ex:2} to get a 
two-dimensional locally compact, geodesically complete $\CAT(0)$-space $X$.
Note that $V_*(\ca S(X))=V(\ca S(X))=I_{\infty}$ and therefore
$\ca H^1(V_*(\ca S(X)))=\e>0$.
\end{ex}

%%%%%%%%%%%%%%%%%%%%%%%%%%%%%%%%%%%%%%%%%%%%%%%%%%%%

\pmed

%%%%%%% new section %%%%%%%%%%%
\setcounter{equation}{0}
\section{Approximations by polyhedral spaces} \label{sec:approx}

In this section, we give the proof of 
Theorem \ref{thm:main}(1) for branched immersed 
disks. We need to recall the notion of turn, which was 
first defined in the context of surfaces with 
bounded curvature in \cite{AZ:bddcurv}
 (see also \cite{Rsh:2mfd}).

%%%%%
\begin{defn} \label{defn:turn} \upshape
For a moment, let $X$ be a surface with bounded curvature.
 In X, we have the notion of angles between geodesics
starting from a point, and use the same 
notations for  spaces of directions, 
etc  (\cite[Theorem II.10]{AZ:bddcurv}).

Let $F$ be a domain in $X$ with boundary $C$.
For an open arc $e$ of $C$, we assume that
$e$ has definite directions at the endpoints $a, b$ and  
the spaces of directions of $F$ at $a$ and $b$ have  positive lengths.
Then  the {\it turn} (rotation) $\tau_F(e)$ of $e$ (\cite[Chapter VI]{AZ:bddcurv}) from the side of $F$ is 
defined as follows:
Let $\gamma_n$ be a broken geodesic in 
$F\setminus e$ except the endpoints joining $a$ and $b$ and converging to $e$
as $n\to\infty$. Let $\Gamma_n$ be the domain 
bounded by $e$ and $\gamma_n$.
We denote by $\alpha_n$ and $\beta_n$ the sector angle
of $\Gamma_n$ at $a$ and $b$ respectively.
Let $\theta_{ni}$, $(1\le i\le N_n)$, denote
the sector angle at the break points of $\gamma_n$, viewed from 
$F\setminus \Gamma_n$. Let 
\[
    \tilde \tau_F(\gamma_n) := \sum_{i=1}^{N_n} (\pi -\theta_{ni}) +\alpha_n+\beta_n.
\]
Then the turn $\tau_F(e)$ is defined as 
\[
     \tau_F(e) := \lim_{n\to\infty}\tilde \tau_F(\gamma_n), 
\]
where the existence of the above limit is shown in \cite[Theorem VI.2]{AZ:bddcurv}.

For an interior point $c$ of $e$ having definite two directions  $\Sigma_c(e)$, the turn of $e$ 
at $c$ from the side $F$ is defined as 
\[
      \tau_F(c):= \pi-L(\Sigma_c(F)).
\]
 %%%%
 
 We now assume the following additional conditions
 for all $c\in e$:
 \begin{enumerate}
  \item $L(\Sigma_c(F))>0\,;$
  \item $e$ has definite two directions
   $\Sigma_c(e)$.
 \end{enumerate}
Consider the constant $\mu_F(e)\in [0,\infty]$ defined
by
  \begin{align} \label{eq:bdd-turn}
  \mu_F(e):=\sup_{\{ a_i\}} \sum_{i=1}^{n-1}|\tau_{F}((a_i,a_{i+1}))|
   +\sum_{i=2}^{n-1}|\tau_{F}(a_i)|,
 \end{align}
where $\{ a_i\}=\{ a_i\}_{i=1,\ldots,n}$  runs over all the consecutive points on $e$.
The constant $\mu_F(e)$ is called the {\it turn variation} of $e$ from the side $F$, and $e$ has {\it finite turn variation} when $\mu_F(e)<\infty$.
For general treatments of curves with 
finite turn variation in $\CAT(\kappa)$-spaces, see \cite{ABG} 
and the references therein.
 %%%%

 Let $e$ be a simple arc on $X$. 
 One can define the notion of sides $F_+$ and $F_-$ of $e$.
Under the corresponding assumptions, 
 we define the turns $\tau_{F_+}(e)$, $\tau_{F_-}(e)$
 of $e$ from $F_+$ and $F_-$ respectively, as above.
Similarly, we define the turn variations $\mu_{F_+}(e)$, $\mu_{F_-}(e)$ from $F_+$ and $F_-$.    
We call $e$ to have finite turn variation
if $\mu_{F_+}(e)<\infty$ and $\mu_{F_-}(e)<\infty$
(actually both are finite if one is so
(\cite[Lemma IX.1]{AZ:bddcurv})).
 When $e$ has finite turn variation, $\tau_{F_+}$ and $\tau_{F_-}$ provide signed Borel measures on $e$
 (\cite[Theorem IX.1]{AZ:bddcurv}).
\end{defn}
%%%%
\pmed\n
{\bf  The structure of the union of ruled surfaces}.\,
Let $p\in \ca S(X)$, and $r=r_p>0$ be as in Theorem \ref{thm:main}. From now, we work on $B(p,r)$. Fix any $v\in V(\Sigma_p(X))$, and let $N=N_v$ be the branching number of $\Sigma_p(X)$ at $v$. 
For small enough $\delta_p>0$, let  
$\gamma_1, \ldots, \gamma_N$ be the geodesics
from $p$ with $\angle(\dot\gamma_i(0), v) = \delta_p$ and
$\angle(\dot\gamma_i(0), \dot\gamma_j(0)) = 2\delta_p$ for 
$1\le i\neq j\le N$.
For $2\le k\le N$, 
we define $E_k$ as the union of  ruled surfaces $S_{ij}$ determined 
by  $\gamma_i$ and $\gamma_j$ for all $1\le i\neq j \le k$. 

%%%%%
Let $C$ denote the union of all singular curves 
$C_{ij\ell}$ \,$(1\le i<j<\ell\le k)$.
By Lemma \ref{cor:sing-total},
$C$ coincides with the set of all topological singular
points resulting from the intersections of distinct
ruled surfaces $S_{ij}$ and $S_{i'j'}$ for all
$1\le i<j\le k$, $1\le i'<j'\le k$ with 
$(i,j)\neq (i',j')$. 
%%%%%

Note that a singular curve in the direction $v$ not 
included in $C$ might meet $E_k$.
We consider the graph structure of $C$ inherited from 
that of $\ca S(X)$, which is not the one of $C$
itself introduced as in Definition \ref{defn:metric-graph}.  Thus we set
\[
  E(C):=C\cap E(\ca S(X)),\quad
  V(C):=C\cap V(\ca S(X)),\quad
  V_*(C):=C\cap V_*(\ca S(X)),\quad
\]
and call $E(C)$ and $V(C)$ the set of egdes and the set of vertices of $C$ respectively.
Remember that all edges are assumed to be open
(Definition \ref{defn:metric-graph}).
%%%
\begin{defn} \label{defn:sing-vertex} \upshape
We say that a vertex point $x\in V(C)$ is {\it singular}
if either $x\in V_*(C)$ or there are two singular curves
$C_1$, $C_2$ in $C$ starting from $x$ such that 
\begin{enumerate}
\item $\angle_x(C_1,C_2)=0\,;$
\item $C_1$ and $C_2$ have no intersections near $x$ 
other than $x$.
\end{enumerate}
	The direction $v\in\Sigma_x(C)$ determined by the above $C_1$ and $C_2$ as well as  $v=\lim_{i\to\infty} \uparrow_x^{x_i}$ with $V(C)\ni x_i\to x$ is also called {\it singular}.
The set of singular vertices of $C$ is denoted by 
$V_{\rm sing}(C)$. 
The set of singular directions at $x\in V_{\rm sing}(C)$ is   denoted by 
$\Sigma_x^{\rm sing}(C)$.
\end{defn}
%%%

We set $r(x):=d_p(x)$ for simplicity.

For $x\in C\setminus \{ p\}$, let 
$\Sigma_{x,+}(C):=\Sigma_x(C\setminus {\rm int} B(p, r(x)))$ and 
$\Sigma_{x,-}(C):=\Sigma_x(C\cap B(p, r(x)))$.
%
%$C_{x,+}:=C\cap (E_k\setminus{\rm int}\,B(p, r(x)))$
%and 
%$C_{x,-}:=C\cap B(p, r(x))$.
%By Corollary \ref{cor:vert},
%$\angle(\pm\nabla d_p(x), \Sigma_x(C_{x,\pm}))
%<\tau_p(r(x))$.
By Lemma \ref{lem:near-vert}, we may assume that
for all $x\in C\setminus \{ p\}$,

\begin{align} \label{eq:10-10}
\begin{cases}
  & \angle(\nabla d_p(x),  \Sigma_{x,+}(C))<10^{-10}, \quad
  {\rm diam}(\Sigma_{x,+}(C))<10^{-10}, \\
&\angle(-\nabla d_p(x),  \Sigma_{x,-}(C))<10^{-10},\quad
   {\rm diam}(\Sigma_{x,-}(C))<10^{-10}.
\end{cases}
\end{align}

The following lemma is clear.

\begin{lem}\label{lem:E-Sigma}
For every $x\in C\setminus \{ p\}$, $\Sigma_x(E_k) (\subset \Sigma_x(X))$ coincides with the union
of all circles $\Sigma_x(S_{ij})$ such that 
$x\in S_{ij}\subset E_k$,
where the circles $\Sigma_x(S_{ij})$ are attached at the points of
$\Sigma_{x,\pm}(C)$.
\end{lem}

The following is a main result of this section.

\begin{thm} \label{thm:union-CAT}
 $E_k$ is a $\CAT(\kappa)$-space.
\end{thm}

\begin{rem}  \label{rem:LS-dim2} \upshape
Recently, we learned that Theorem \ref{thm:union-CAT}
is a direct consequence of the main result of 
Lytchak and Stadler \cite{LS:dim2}. 
However, in what follows, we present our original proof,
which provides deep insights on the local geometry  of $X$,
and will also be used in \cite{KNSYII} as one of key methods.
\end{rem}

 The basic strategy of the proof of Theorem \ref{thm:union-CAT} is to use 
the results \cite[Theorems 0.5 and 0.6]{BurBuy:upperII} 
on the characterizations for polyhedral spaces to be 
$\CAT(\kappa)$-spaces.

%%%%%%%%%%%

Let $\ca F_\kappa$ be the family of two-dimensional  polyhedral 
locally $\CAT(\kappa)$-spaces $F$ possibly with boundary
$\pa F$ such that any edge of $\pa F$ has finite turn variation. For a collection $\{ F_i\}$ of 
$\ca F_\kappa$, let $X$ be the polyhedron resulting from certain gluing of $\{ F_i\}$ along their edges. We always consider the 
intrinsic metric of $X$ induced from those of $F_i$.
We consider the following two conditions:
\pmed\n
(A)\, For any Borel subset $B$ of an arbitrary edge $e$ of
$X$, and arbitrary faces $F_i$, $F_j$ adjacent to $e$, we have
\[
    \tau_{F_i}(B)+\tau_{F_j}(B)\le 0\,;
\]
(B)\, For any vertex $x$ of $X$, $\Sigma_x(X)$ is $\CAT(1)$.
\psmall
\begin{thm} $($\cite[Theorem 0.5]{BurBuy:upperII}$)$ 
\label{thm:BB-gluing}
A polyhedron $X$ resulting from certain gluing of 
$\{ F_i\}\subset\ca F_\kappa$ along their edges belongs to $\ca F_\kappa$ if and only if the conditions $(A), (B)$
are satisfied.
\end{thm}

\begin{thm} $($\cite[Theorem 0.6]{BurBuy:upperII}$)$\label{thm:BB-character}
Each polyhedron $X$ in $\ca F_\kappa$ can be glued from the faces $\{ F_i\}$ contained in $\ca F_\kappa$ along their edges in such a way that the conditions $(A), (B)$ are satisfied.

In particular, each edge of $\ca S(X)$ has finite turn
variation.
\end{thm}

%%%%%%%%%%%%

Note that $E_k$ is not a polyhedral space in general.
Even in that case, we have some difficulty 
mentioned below. From these reasons, we shall do 
surgeries to get a polyhedral space $\tilde E_k$
which approximates $E_k$ in the Gromov-Hausdorff sense.
The point is, we can apply Theorems \ref{thm:BB-gluing}
and \ref{thm:BB-character} 
to $\tilde E_k$ to conclude that it is 
$\CAT(\kappa)$. Finally taking the limit, we will obtain 
the conclusion.  

%%%%
% We need some lemmas on the relations between the local geometries of $E_k$ and $X$.
% 
 From now on, we set $E:=E_k$ for simplicity.
 We need some preliminary argument on the local geometry of $E$.

\begin{lem} \label{lem:E-int=ext}
For every $x\in E$, $\Sigma_x(E)$ is isometric to the intrinsic space of 
directions $\Sigma_x(E^{\rm int})$
 in the sense of Definition \ref{defn:intrinsic-direct}.
%{\cred \$ 4  should be modified too.}
\end{lem} 

\begin{proof}
The basic idea of the proof is the same as that of 
Lemma \ref{lem:ext=int}.
% we outline the argument of the proof.
Obviously, we may assume $x\in C$. We only consider the 
case $x\neq p$.
%%%%%%
We first show that each component $\Sigma$ of 
$\Sigma_x(E)\setminus \Sigma_x(C)$
is isometrically embedded in $\Sigma_x(E^{\rm int})$.
For $\xi_1, \xi_2\in \Sigma$ with $|\xi_1,\xi_2|<\pi$,
let $\mu_n$ be an $X$-geodesic with
$\dot\mu_n(0)=\xi_n$ \,$(n=1,2)$. Then for small $\e$, we have  %from Lemma \ref{lem:regular-dir} 
 $\mu_1([0,\e])\subset S_{ij}$ and $\mu_2([0,\e])\subset S_{k\ell}$ for some $S_{ij}, S_{k\ell}$ in $E$.
 Note that the $X$-geodesic $\gamma^X_{\mu_1(t),\mu_2(t)}$ joining $\mu_1(t)$ and $\mu_2(t)$ does not meet $C$,
 and hence $\gamma^X_{\mu_1(t),\mu_2(t)}$
  is contained in the same ruled surface 
  $S_{ij}=S_{k\ell}\subset E$.
 This implies that $\angle^X(\xi_1,\xi_2)= \angle^{S_{ij}}(\xi_1,\xi_2)$. From $\angle^X\le\angle^E\le\angle^{S_{ij}}$, 
 we conclude that $\angle^X(\xi_1,\xi_2)= \angle^E(\xi_1,\xi_2)$ and the existence of an isometric 
 embedding $\iota:\Sigma\to\Sigma_x(E^{\rm int})$. 
 %%%%%
 
 Next, for any $v\in\Sigma_x(C)$, take 
$\xi_3, \xi_4\in \Sigma_x(E)\setminus \Sigma_x(C)$ close to $v$ such that the segment $[\xi_3,\xi_4]$ in $\Sigma_x(E)$ meets 
$\Sigma_x(C)$ only at $v$.
Take $X$-geodesics  $\alpha_3,\alpha_4$ in the direction $\xi_3,\xi_4$, and choose $S_{ij}, S_{i'j'}$ in $E$
such that $\alpha_3(t)\subset S_{ij}$ and 
$\alpha_4(t)\subset S_{i'j'}$.
Let $\alpha:[0,1]\to X$ be the $X$-geodesic from $\alpha_3(t)$ to $\alpha_4(t)$.
By  \eqref{eq:10-10},  $\alpha$ is vertical.
We extend  $\alpha$ until it reaches $\pa E$. 
We can choose such an extension that 
$\alpha(t_-)\in\gamma_\ell$ and 
$\alpha(t_+)\in\gamma_\ell'$ with 
$\ell\in\{ i, j\}$ and $\ell'\in\{ i', j'\}$ for some 
$t_-<0<1<t_+$.
Thus, we have $\alpha([t_-,t_+])\subset S_{\ell\ell'}$.
%%%%%%

By Lemma \ref{lem:pass}, we can find a sequence $s_n$ 
such that  the ruling geodesics $\lambda_{s_n}$ of 
$S_{\ell \ell'}$ meets both $\alpha_3$ and $\alpha_4$
and $\lambda_{s_n}$ converges to a ruling 
geodesic through $x$ as $n\to\infty$.
This implies that $\angle^X(\xi_3,\xi_4)= \angle^{S_{ij}}(\xi_3,\xi_4)$, and hence $\angle^X(\xi_3,\xi_4)= \angle^{E}(\xi_3,\xi_4)$.
% as well as
%$\angle(\iota(\xi_3),\iota(\xi_4))=\angle(\xi_3,\xi_4)$.
This completes the proof.
\end{proof}

\begin{lem}\label{lem:E-geodesic}
 For arbitrary $x,y\in E$, let 
 $\gamma:=\gamma^E_{x,y}:[0, |x,y|_E]\to E$
 be an $E$-shortest curve between $x$ and $y$.
 Suppose that the set of accumulation points of $\gamma\cap \ca S(X)$ is finite. 
  Then $\gamma$ is an $X$-geodesic.
\end{lem}
\begin{proof} 
%%%
Set $\Gamma:=\gamma\cap\ca S(X)$.
We only have to consider the case when $\Gamma$ has a unique 
accumulation point $\gamma(u)$ with 
$\Gamma=\{ \gamma(t_i),\gamma(s_j), \gamma(u)\,|\,
i,j=1,2,\ldots \}$ with 
$0\le t_1<t_2<\cdots<t_i<\cdots <u< \cdots< s_j<\cdots<s_2<s_1\le |x,y|_E$ and 
$\lim_{i\to\infty} t_i=\lim_{j\to\infty} s_j=u$.

Note that 
%the surface $E\setminus \ca S(X)$ is open 
%in the surface $X\setminus \ca S(X)$. 
%Suppose that $\gamma$ 
% %transversally 
% meets $\ca S(X)$ with 
% finitely many points $\gamma(t_i)$, where 
% $0\le t_1<\cdots t_i<\cdots <t_T\le |x,y|_E$,
% in the sense that $\angle(\dot\gamma(t_i),
% \Sigma_{\gamma(t_i)}(C_0))>0$ for all 
% singular curve $C_0$ in $\ca S(X)$ containing $\gamma(t_i)$.
%%%%
for each $i$ and 
 any small enough  $\e>0$, $\gamma([t_{i-1}+\e, t_i-\e])$ 
 is contained in the surface $X\setminus \ca S(X)$. Therefore $\gamma|_{[t_{i-1}+\e, t_i-\e]}$ is locally $X$-minimizing, and hence 
 $X$-minimizing. Thus 
 $\gamma|_{[t_{i-1}, t_i]}$ is $X$-minimizing.
 
 By Lemma \ref{lem:E-int=ext}, we have  
 \[
 \angle^X(\dot\gamma^E_{\gamma(t_{i}),\gamma(t_{i-1})},  
 \dot\gamma^E_{\gamma(t_{i}), \gamma(t_{i+1})})=
 \angle^E(\dot\gamma^E_{\gamma(t_{i}),\gamma(t_{i-1})},  
 \dot\gamma^E_{\gamma(t_{i}), \gamma(t_{i+1})})=
 \pi.
 \] 
 Therefore $\gamma|_{[t_{i-1},t_{i+1}]}$ is an $X$-geodesic,
  %%%%%
 which implies that $\gamma|_{[0,u]}$ is an $X$-geodesic. Similarly, $\gamma|_{[u,|x,y|_E]}$ is an $X$-geodesic.
 In a way similar to the above, we have
 %%%
  $\angle^X(\dot\gamma^E_{\gamma(u),\gamma(0)}(0),  
 \dot\gamma^E_{\gamma(u), \gamma(|x,y|_E)}(0))=
% \angle^E(\dot\gamma^E_{\gamma(t_{i}),\gamma(t_{i-1})},  
% \dot\gamma^E_{\gamma(t_{i}), \gamma(t_{i+1})})=
 \pi$. 
  It follows that $\gamma$ is an $X$-geodesic.
% \gamma(t_{i+1})= \angle^X\gamma(t_{i-1} \gamma(t_i)\gamma(t_{i+1})=\pi$.
% $\angle^E\gamma(t_{i-1}) \gamma(t_i)\gamma(t_{i+1})= \angle^X\gamma(t_{i-1} \gamma(t_i)\gamma(t_{i+1})=\pi$.
\end{proof} 

\begin{rem} \label{rem:accumlation-EX}\upshape
Lemma \ref{lem:E-geodesic} does not hold 
in case a subarc of $\gamma$ is contained in 
$\ca S(X)$ (see Example \ref{ex:2}).
\end{rem}
 
\begin{lem} \label{lem:SE-geom}
For a fixed  $x\in E$, we have for every $y\, (\neq x)\in E$ %distinct from $x$,
%\begin{align*}
\[
 \frac{|x,y|_{E}}{|x,y|_X} <1+\tau_x(|x,y|_X).
\]
%    & (2)\,\,  \angle^X \left(\dot\gamma^{E}_{x,y}(0), 
%         \dot\gamma^X_{x,y}(0)\right) <\tau_x(|x,y|_X),\\
%   & (3)\,\,  \angle^X \left(\dot\gamma^{E}_{y,x}(0), 
%         \dot\gamma^X_{y,x}(0)\right) <\tau_x(|x,y|_X).
%\end{align*}
\end{lem}
\begin{proof}
We may assume $x\in E\cap \ca S(X)$.
Suppose the conclusion does not hold. Then we have 
a sequence $y_n\in E$ converging to $x$ 
such that 
$\frac{|x,y_n|_{E}}{|x,y_n|_X} >1+c$ for some 
positive constant $c$. Passing to a subsequence, we may 
assume that all $y_n\in S_{ij}$ for some $S_{ij}$.
This is a contradiction to Lemma \ref{slem:ratio-SX}
since $|x,y_n|_{E}\le |x,y_n|_{S_{ij}}$.
\end{proof}

%%%%%
%\begin{lem} \label{lem:SE-geom2}
%For a fixed  $x\in E$ and $0<\theta<\pi/2$, let 
%$y, z\, (\neq x)\in E$ be arbitrary distinct points  satisfying
%$\theta\le \angle yxz\le \pi-\theta$. Then 
%we have
%\begin{align*}
%     \frac{|y,z|_{E}}{|y,z|_X} <1+\tau_{x,\theta}(|x,y|_X).
%\end{align*}
%\end{lem}
%\begin{proof}
%In a way similar to Lemma \ref{lem:ext=int},
%we can show that 
%$\Sigma_x(E)$ is isometric to the intrinsic 
%space of directions $\Sigma_x(E)^{\rm int}$.
%Since 
%\begin{align*}
%(K(\Sigma_x(E)), o_x)&=\lim_{\e\to 0}
%      \left( \frac{1}{\e} E,x\right),\\
%(K(\Sigma_x(E)^{\rm int}), o_x)&=\lim_{\e\to 0}
%      \left( \frac{1}{\e} E^{\rm int},x\right).

%%%%%%
\pmed
\begin{proof}[Proof of Theorem \ref{thm:union-CAT}]
For each edge $e\in E(C)$, let 
$D_i$\,$(1\le i\le m(e))$ be open half-disks in $X$ with
$\pa D_i=e$ such that 
\begin{align} \label{eq:topological-suspen}
\text{$\textstyle{\bigcup_{i=1}^{m(e)} D_i}$ is an open neighborhood of $e$ in $X$.}
%\begin{cases}
%&{\text{the union
%$\textstyle{\bigcup_{i=1}^{m(e)} D_i}$ is an open neighborhood of $e$ in $X$,}}\\
%& \textstyle{\bigcup_{i=1}^{m(e)} D_i}\cap \ca S(X)=e.
%\end{cases}
\end{align}
Let $\tau_{D_i}$ be the turn of $e$ from the side 
$D_i$. 
We want to apply Theorem \ref{thm:BB-character}
to the completion of the components
of $E\setminus C$. Let $A$ be such a completion 
containing some $D_i$. However here are some difficulties: The domain $A$ might be too thin to define 
$\tau_{D_i}(e)$ because of the presence of singular vertices. In particular, we do not know if $e$ has finite turn variation in $A$. We also have to care about $V_*(C)$.
To overcome these difficulties, we do surgeries
around points of $V_{\sing}(C)$.
At this moment, we can apply Theorem \ref{thm:BB-character} to $e$ locally.
 Each point of $e$ has a convex neighborhood $P$ in 
$\textstyle{\bigcup_{i=1}^{m(e)} D_i}$ 
such that $\pa P$ consists of broken geodesics
joining the endpoints of $e\cap P$.
It follows from   Theorem \ref{thm:BB-character} that 
we have for all $1\le i\neq j\le m(e)$,
\begin{align} \label{eq:turn+turn}
   \tau_{D_i}(e\cap P) + 
      \tau_{D_j}(e\cap P)\le 0 %\quad \text{ locally on \,$e$,}
\end{align} 
and $e$ has locally finite turn variation in $D_i$.

%\begin{defn}\label{defn:sing-vertex} \upshape
%For $x\in V_{\rm sing}(C)$, we say that $v\in\Sigma_x(C)$
%\end{defn}

Let $\e_0$ be any positive number.
For $x\in V_{\sing}(C)$, we assume that the singularity of $x$ occurs from 
the positive direction. Namely, there is $v\in \Sigma_{x,+}^{\rm sing}(C)$.
% for which either there is a sequence 
%$x_n\in V(C)$ with $r(x_n)>r(x)$ and $\lim \uparrow_x^{x_n}=v$ converging to $x$,
%or there are singular curves $C_1, C_2$ starting from $x$
%in the the direction $v$  
%$\Sigma_x^+(C_1)=\Sigma_x^+(C_2)$. 
The other case $v\in \Sigma_{x,-}^{\rm sing}(C)$ is similarly discussed.
Let $C(v)$ denote the union of singular curves in $C$  starting at $x$ in the direction $v$.

Choose $\delta=\delta_x>0$  and 
$\e=\e_x>0$ with $\delta, \e\le \e_0$ and $\e\ll \delta$  satisfying 
\begin{align} \label{eq:non-meeting-VC}
 &\text{$\{ B^{\Sigma_x(X)}(v,2\delta)\}$\, 
 $(v\in\Sigma_x^{\sing}(C))$ is mutually disjoint\,$;$} \\
 &\text{$C(v,\delta, 2\e)$
 (see \eqref{eq:cone-neighborhood}) covers 
 $C(v)\cap B_+(x,2\e)\,;$} \label{eq:non-meeting-VC2}\\
% &  \text{in Lemma \ref{lem:SE-geom}, 
%  $\tau_x(|x,y|_X)<\e_0$ for all $y\in B^E(x,2\e)\,;$}
%          \label{eq:inLemma}\\
 &\text{$E\cap S(p,r(x)+\e)$ does not meet $V(C)$,}     
      \label{eq:non-meeting-VC3} \hspace{3.5cm}
\end{align}
where $B_+(x,\e):=B(x,\e)\setminus {\rm int}\,B(p,r(x))$.
%(resp. $B_-(x,2\e)=B(x,2\e)\cap B(p, r(x))$).
%%%
By Lemma \ref{lem:circle-tree},  $C(v,\delta,\e)\cap E\cap S(p, \delta(x)+\e)$ is a tree, say $\hat T(x,v)$. 
Replacing each edge of $\hat T(x,v)$ by the $X$-geodesic between the endpoints,  we obtain a geodesic tree $T(x,v)$. 
By Lemma \ref{lem:verticalS}, we have 
$T(x,v)\subset E$.
Let $K(x,v)$ be a closed domain of $E$  bounded by 
$T(x,v)$ and the $X$-geodesic segments
between $x$ and the endpoints of the tree $T(x,v)$.
Note that such $X$-geodesics between $x$ and the endpoints of $T(x,v)$ are contained in $E$. 
Taking smaller $\delta, \e$ if necessary, we may assume 
\begin{align} \label{eq:tildeXyxy'}  
 \text{$\tilde\angle^X yxy' <\e_0$ for all $y,y'\in
    T(x,v)$.}  
\end{align}

For each vertex  $y\in V(T(x,v))$, take the $X$-geodesic
$\gamma^X_{x,y}$ between $x$ and $y$.
For each edge $e\in E(T(x,v))$ with endpoints $y,y'$, let
 $\triangle^X_e$ denote the $X$-geodesic triangle 
 consisting of $\gamma^X_{x,y}\cup \gamma^X_{x,y'} 
 \cup \gamma^X_{y,y'}$.

%\begin{lem} \label{lem:K=ETD}
%\[
%   K(x,v)=\sum_{e\in E(T(x,v))} \,\blacktriangle_{e}^E.
%\]
%\end{lem}
%\begin{proof}
%We denote by $K_\Delta(x,v)$ the RHS of the above .
%For every $s>0$, $K(x,v)\cap S(p,s)$ is a tree with 
%endpoints contained in $K_\Delta(x,v)$. It suffices to show that 
%$K_{\Delta}(x,v)\cap S(p,s)$ is connected.
%This is easily seen by using the $d_p$-gradient flow.
%\end{proof}

%%%%
%{\cred 
%Let $\Delta^X_e$ denote the $X$-geodesic triangle
%$\gamma^X_{x,y}\cup \gamma^X_{x,y'} 
% \cup \gamma^X_{y,y'}$.
%}
%Taking  small enough $\e_x$, by Lemma \ref{lem:comparison} and \eqref{eq:inLemma},
%%\ref{lem:verticalS} and \ref{lem:sing-arc},
% we may assume that 
%%%%%
%
%\begin{align} \label{eq:close-angle}
%\begin{cases}
%& \text{the angles of $\triangle^X_e$ are $\e_0$-close to the corresponding angles} \\
%&\text{of the comparison triangle 
%$\tilde\triangle^X_{e}$ in $M^2_\kappa\,;$} %\label{eq:close-angle2}
%\end{cases} 
%\end{align}
%\vspace{-5mm}
%\begin{align} \label{eq:biLipschitz-close}
%  \text{$\left| \frac{|x,y|_X}{|x,y|_{E}}-1 \right| <\e_0$
%      for all interior vertices $y$ in $V(T(x,v))$}.
%\end{align}

%%%%
Let $\tilde\blacktriangle^X_{e}$ be the triangular  region
bounded by $\tilde\triangle^X_{e}$.
Gluing 
$\{ \tilde\blacktriangle^X_{e}\,|\,e\in E(T(x,v))\}$ properly, we obtain 
a polyhedral space $\tilde K(x,v)$ corresponding to 
$K(x,v)$. 

%%%%%%%
 
%%%%%tikz picture%%%%%p.48
\begin{center}
\begin{tikzpicture}
[scale = 1]
\fill (0,0) coordinate (A) circle (1.2pt) node [below] {$x$};
\fill (4.5,0.4) coordinate (A) circle (1.2pt) node [right] {$y$};
\draw (4.3,2.2) node[circle] {$T(x,v)$};
\draw (1.6,-1.6) node[circle] {$K(x,v)$};
\draw [-, very thick] (0,0)--(0.1,0);
\draw [-, very thick] (4.5,0.4)--(4.8,1.6);
\draw [-, very thick] (4.5,-0.4)--(4.8,-1.6);
\draw [-, very thick] (0,0)--(4.8,1.6);
\draw [-, very thick] (0,0)--(4.8,-1.6);
\draw [dotted, very thick] (3.95,1.315)--(4.5,0.4);
\draw [-, very thick] (3.95,1.315)--(3.6,1.9);
\draw [-, very thick] (0,0)--(3.6,1.9);
\draw [-, very thick] (0,0)--(3.6,-1.9);
\draw [dotted, very thick] (3.95,-1.315)--(4.5,-0.4);
\draw [-, very thick] (3.95,-1.315)--(3.6,-1.9);
\coordinate (P1) at (3.7,0);
\coordinate (P2) at (2,0);
\coordinate (P3) at (1,0);
\coordinate (P4) at (0.3,0);
\coordinate (P5) at (0.1,0);
\coordinate (P6) at (4.5,0.4);
\coordinate (P7) at (4.5,-0.4);
\draw [very thick]
(P1) .. controls +(30:1cm) and +(180:0cm) ..
(P6) .. controls +(270:0.5cm) and +(90:0.5cm) ..
(P7) .. controls +(180:0cm) and +(330:1cm) .. (P1); 
\filldraw [fill=gray, opacity=.1]
(P1) .. controls +(30:1cm) and +(180:0cm) ..
(P6) .. controls +(270:0.5cm) and +(90:0.5cm) ..
(P7) .. controls +(180:0cm) and +(330:1cm) .. (P1); 
\draw [very thick]
(P1) .. controls +(160:1cm) and +(20:1cm) ..
(P2) .. controls +(340:1cm) and +(200:1cm) .. (P1); 
\filldraw [fill=gray, opacity=.1]
(P1) .. controls +(160:1cm) and +(20:1cm) ..
(P2) .. controls +(340:1cm) and +(200:1cm) .. (P1); 
\draw [very thick]
(P2) .. controls +(160:0.6cm) and +(20:0.6cm) ..
(P3) .. controls +(340:0.6cm) and +(200:0.6cm) .. (P2); 
\filldraw [fill=gray, opacity=.1]
(P2) .. controls +(160:0.6cm) and +(20:0.6cm) ..
(P3) .. controls +(340:0.6cm) and +(200:0.6cm) .. (P2); 
\draw [very thick]
(P3) .. controls +(160:0.4cm) and +(20:0.4cm) ..
(P4) .. controls +(340:0.4cm) and +(200:0.4cm) .. (P3); 
\filldraw [fill=gray, opacity=.1]
(P3) .. controls +(160:0.4cm) and +(20:0.4cm) ..
(P4) .. controls +(340:0.4cm) and +(200:0.4cm) .. (P3); 
\draw [very thick]
(P4) .. controls +(160:0.15cm) and +(20:0.15cm) ..
(P5) .. controls +(340:0.15cm) and +(200:0.15cm) .. (P4); 
\filldraw [fill=gray, opacity=.1]
(P4) .. controls +(160:0.15cm) and +(20:0.15cm) ..
(P5) .. controls +(340:0.15cm) and +(200:0.15cm) .. (P4); 
\draw [-, very thick] (5.5,0)--(5.6,0);
\fill (5.5,0) coordinate (A) circle (1.2pt) node [below] {$\tilde{x}$};
\fill (10,0.4) coordinate (A) circle (1.2pt) node [right] {$\tilde{y}$};
\draw (9.8,2.2) node[circle] {$\tilde{T}(x,v)$};
\draw (7.1,-1.6) node[circle] {$\tilde{K}(x,v)$};
\draw [-, very thick] (5.5,0)--(10,0.4);
\draw [-, very thick] (5.5,0)--(10,-0.4);
\draw [-, very thick] (10,0.4)--(10,-0.4);
\draw [-, very thick] (10,0.4)--(10.3,1.6);
\draw [-, very thick] (10,-0.4)--(10.3,-1.6);
\draw [-, very thick] (5.5,0)--(10.3,1.6);
\draw [-, very thick] (5.5,0)--(10.3,-1.6);
\draw [dotted, very thick] (9.45,1.315)--(10,0.4);
\draw [-, very thick] (9.45,1.315)--(9.1,1.9);
\draw [-, very thick] (5.5,0)--(9.1,1.9);
\draw [-, very thick] (5.5,0)--(9.1,-1.9);
\draw [dotted, very thick] (9.45,-1.315)--(10,-0.4);
\draw [-, very thick] (9.45,-1.315)--(9.1,-1.9);
\end{tikzpicture}
\end{center}
%%%%%%%%%%

%%%%%%%%%

We provide a relation between $K(x,v)$ and $\tilde K(x,v)$.

\begin{lem} \label{lem:thin-est}
$(1)$\,Let $y, y'$ be arbitrary endpoints of $T(x,v)$.
For arbitrary $z\in\gamma^X_{x,y}=\gamma^E_{x,y}$ and $z' \in\gamma^X_{x,y'}=\gamma^E_{x,y'}$,
%Let $\gamma:[0,|z,z'|_E]\to E$ be an $E$-shortest curve     joining $z$ to $z'$. 
assuming $|x,z|_X\le |x,z'|_X$, we have 
   \[
     ||z,z'|_{E}-|\tilde z, \tilde z'||< \tau(\e_0)|x,z|_X,
   \]
where $\tilde z, \tilde z'\in\pa\tilde K(x,v)$ are the points corresponding to $z,z'$ respectively. 
\par\n
$(2)$\, For arbitrary $y\in T(x,v)$
and $z\in \pa K(x,v)$, we have 
   \[
     ||y,z|_{E}-|\tilde y, \tilde z||< \tau(\e_0)|x,y|_X.
   \]
where $\tilde y, \tilde z$ are the points of 
$\pa\tilde K(x,v)$ corresponding to $y, z$.
\end{lem}
%\psmall
%%%%%%%
\begin{proof}
(1)\, 
Let $\hat z\in\gamma^{X}_{x,z'}$ be the point such 
that $|x,\hat z|_{X}=|x,z|_{X}$. Note that 
$\gamma^X_{z,\hat z}$ is vertical, and therefore
contained in $E$.
By triangle inequality, we have
$||z,z'|_{E}-|z',\hat z|_{E}|\le |\hat z,z|_{E}=|\hat z,z|_{X}$.
\eqref{eq:tildeXyxy'} implies 
$\tilde\angle^X zx\hat z<\e_0$, and hence
$|z,\hat z|_X<\tau(\e_0)|x,z|_X$.
%%%
In view of $\gamma_{x,y}^X=\gamma_{x,y}^E$ and
 $\gamma_{x,y'}^X=\gamma_{x,y'}^E$, we have 
$$
||z,z'|_{E}-(|x,z'|_{X}-|x,z|_{X})|\le \tau(\e_0)|x,z|_{X}.
$$
Since $\angle \tilde z\tilde x \tilde z'<\tau(\e_0)$, 
similarly we have 
$||\tilde z,\tilde z'|-(|x,z'|_{X}-|x,z|_{X})|\le \tau(\e_0)|x,z|_{X}$.
Combining the last two inequalities, we obtain the 
required inequality.
\par\n
(2)\, 
Choose $w\in\pa T(x,v)$ such that 
$z\in\gamma^X_{x,w}$.
From \eqref{eq:tildeXyxy'}, %and  \eqref{eq:inLemma},
we have $|y,w|_E=|y,w|_X<\tau(\e_0)|x,y|_X$.
Therefore by triangle inequality, we obtain 
$||y,z|_E -|z,w|_E|\le |y,w|_E<\tau(\e_0)|x,y|_X$.
Since $\angle\tilde y\tilde x\tilde z<\e_0$, 
by a similar consideration on the triangle
$\triangle \tilde y\tilde x\tilde z$ in $M^2_\kappa$, we have 
$||\tilde y,\tilde z|-|\tilde z,\tilde w|| <\tau(\e_0)|x,y|_{X}$,
where $\tilde w$ is the point of $\pa\tilde T(x,v)$ corresponding to $w$. Since $|z,w|_E=|z,w|_X=|\tilde z,\tilde w|$, combining the last two inequalities, we obtain the 
required inequality.
\end{proof}
%%%%%%%%%%%

In $E$, we do surgeries by removing 
$K(x,v)$ from $E$, and gluing 
$E\setminus K(x,v)$ and $\tilde K(x,v)$ along their isometric boundaries to get a new space, say 
$\tilde E_{x,v}$.

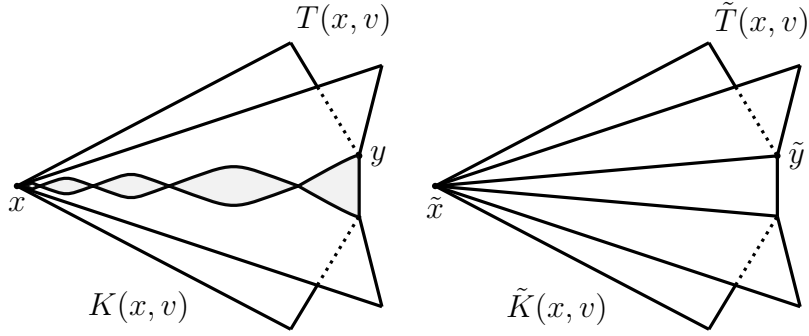
\begin{prop}\label{prop:CAT(1)-vertexK}
For each vertex $\tilde y$ of $\tilde K(x,v)$,
$\Sigma_{\tilde y}(\tilde E_{x,v})$ is $\CAT(1)$.
\end{prop}

We begin with 

\begin{lem} \label{lem:E-CAT}
For each vertex $y\in V(T(x,v))$,
$\Sigma_{\tilde y}(\tilde E_{x,v})$ is $\CAT(1)$,
where $\tilde y\in V(\tilde T(x,v))$ is the vertex corresponding to $y$.
\end{lem}
\begin{proof}
The lemma is clear when $y$ is an endpoint 
of $T(x,v)$. From now, we assume that $y$ is an interior vertex of $T(x,v)$.
%%%
Let us consider 
\begin{align*}
&\Sigma_y^-:=\Sigma_y(K(x,v))\subset \Sigma_y(X), \\
&\Sigma_y^+:=\Sigma_y(E\setminus {\rm int}\,K(x,v))\subset \Sigma_y(X),\\
& \tilde \Sigma_y^-:=\Sigma_{\tilde y}(\tilde K(x,v)).
\end{align*}
Note that $\Sigma_y(E)=\Sigma_y^-\cup\Sigma_y^+$
is a subgraph of $\Sigma_y(X)$ without endpoints,
and hence it is $\CAT(1)$. 
Since 
$\Sigma_{\tilde y}(\tilde E_{x,v})=\tilde\Sigma_y^-\cup\Sigma_y^+$, it suffices to show 
\begin{claim} \label{clm:expanding}
There is an expanding map 
$\Sigma_y^-\to \tilde\Sigma_{y}^-$.
\end{claim}
\begin{proof}
Let $\gamma:=\gamma_{y,x}^X:[0,|y,x|_X]\to X$, and set 
$u:=\dot\gamma(0)$. Choose any
$\xi\in \Sigma_y(T(x,v))\subset\Sigma_y^-$,
and let $w\in\Sigma_y^-$ be the direction of $C$.
Let $\tilde\xi,\tilde u$ be the directions in 
$\tilde\Sigma_y^-$ corresponding to $\xi, u$ respectively.

\pmed\n
Case i)\, $u\in \Sigma_y^-$.
\psmall
Since $X$ is locally $\CAT(\kappa)$, we have
$\angle^X(\xi,u)\le \angle(\tilde\xi,\tilde u)$.
Therefore  the correspondence
$\xi\to\tilde\xi$, $u\to \tilde u$ gives rise to 
the desired expanding map
$\Sigma_y^-\to \tilde\Sigma_{\tilde y}^-$.
\pmed\n
Case ii)\, $u\notin \Sigma_y^-$.
\psmall
This is the case when $\gamma$ leaves 
$E$ after $y=\gamma(0)$ at least for a short time.
From  \eqref{eq:non-meeting-VC3}, $y$ is contained in an open edge in $E(C)$. Therefore, for small enough $t>0$, the $X$-geodesic starting from $\gamma(t)$ to $\gamma_{\xi}(t)$ must meet $C$.
This implies
\begin{align} \label{eq:angle-expanding}
\angle^X(\xi,w)\le\angle^X(\xi,u)\le \angle(\tilde\xi,\tilde u).
\end{align}
Therefore the correspondence
$\xi\to\tilde\xi$, $w\to \tilde u$ gives rise to 
the desired expanding map
$\Sigma_y^-\to \tilde\Sigma_{\tilde y}^-$.
 Note that by  \eqref{eq:topological-suspen},
 $\Sigma_y(X)$ is homeomorphic to 
the suspension  with vertices $\Sigma_x(C)$,
from which \eqref{eq:angle-expanding} also follows.
\end{proof}
This completes the proof of Lemma \ref{lem:E-CAT}.
\end{proof}

For the proof of Proposition \ref{prop:CAT(1)-vertexK},
it suffices to show the following.

\begin{lem}       \label{lem:CATx}
$\Sigma_{\tilde x}(\tilde E_{x,v})$ is $\CAT(1)$.
\end{lem}

%For each $1\le i\le m$, let $F_i=F_i(x,v)$ be the connected component of  $K(x,v)\setminus \ca S(X)$ containing $\sigma_i$.
%%%%%
%Let $W_*$ be the {\it core} of $K(x,v)$ defined as 
%the complement of 
%$F_1\cup  \cdots \cup F_m$ in $K(x,v)$.
%
%The following sublemma is needed in the proof of Lemma \ref{lem:CATx}.
%
%\begin{slem} \label{lem:yx-dot=v}
%Taking smaller $\e>0$ if necessary, 
%we may assume that for each $y\in V_{\rm int}(T(x,v))$,
%either $\dot\gamma^X_{x,y}(0) =v$ or 
%$\gamma^X_{x,y}$ does not meet $W_*$
%except the endpoints.
%\end{slem}
%\begin{proof}
%If there is a sequence $t_i>0$ converging to $0$ such that  $\gamma^X_{x,y}(t_i)\in W_*$, then 
%$\dot\gamma^X_{x,y}(0) =v$.
%Otherwise, taking smaller $\e>0$, 
%we may assume that $\gamma^X_{x,y}$ does not meet $W_*$ except the endpoints.
%\end{proof}

\begin{proof}
Let 
$\sigma_i$ \,$(1\le i\le m)$  be the $X$-geodesics
joining $x$ to the points of $\pa T(x,v)$, and set
$\nu_i:=\dot\sigma_i(0)$\,$(1\le i\le m)$.
Remember that $\Sigma_x(K(x,v))$ consists of
$m$ segments from the vertex $v$ to $\nu_i$
of length $\delta$.
Since $\Sigma_x(X)$ is $\CAT(1)$, it suffices to show
\begin{align}\label{eq:angle(nui,nuj)}
\text{$\angle(\tilde\nu_i,\tilde \nu_j)\ge 2\delta$\, for all\, $1\le i\neq j\le m$,}
\end{align}
where $\tilde \nu_i$ denotes the direction at $\tilde x$
corresponding to $\nu_i$.

For arbitrary $y,y'\in V_{\rm int}(T(x,v))$ adjacent to $\pa T(x,v)$,
let us assume that $z_1,\ldots, z_\ell\in \pa T(x,v)$ 
(resp.  $z_{1'},\ldots, z_{n'} \in \pa T(x,v)$
with $1'<\cdots <n'$)
are the set of 
$\pa T(x,v)$ adjacent to $y$ (resp. to $y'$) with
$z_i  \in\sigma_i$ (resp. $z_{i'} \in \sigma_{i'}$).
Set $v_y:=\dot\gamma^{X}_{x,y}(0)$ and 
$\tilde v_y:=\dot\gamma_{\tilde x,\tilde y}(0)
\in V(\Sigma_{\tilde x}(\tilde K(x,v)))$.
Using the angle comparison for  
$\triangle_e^X$, we have for any  $1\le i \neq  j\le \ell$
\begin{align*}
\angle(\tilde \nu_i,\tilde \nu_j)&=
     \angle(\tilde \nu_i,\tilde v_y)+
       \angle(\tilde v_y, \tilde \nu_j)\\
  &\ge 
  \angle^X(\nu_i, v_y)+
       \angle^X(v_y, \nu_j)=2\delta.
\end{align*}
Let $\bigcup_{\alpha=1}^k [y_{\alpha-1},y_\alpha]$ be the 
shortest path from $y$ to $y'$ in $T(x,v)$ with 
$y=y_0$, $y'=y_k$, $y_\alpha\in V_{\rm int}(T(x,v))$
and $[y_{\alpha-1},y_\alpha]\in E(T(x,v))$.
Then for arbitrary $1\le i\le \ell$ and $1'\le j'\le n'$,
we have 
\begin{align*}
\angle(\tilde \nu_i,\tilde \nu_{j'})&=
     \angle(\tilde \nu_i,\tilde v_y)+
     \sum_{\alpha=1}^k\angle
     (\tilde v_{y_{\alpha-1}},\tilde v_{y_\alpha}) +     
       \angle(\tilde v_{y'}, \tilde \nu_{j'})\\
  &\ge 
  \angle^X(\nu_i, v_y)+
   \sum_{\alpha=1}^k\angle^X
     (v_{y_{\alpha-1}}, v_{y_\alpha}) +  
       \angle^X(v_{y'}, \nu_{j'})\ge 2\delta.
\end{align*}
This completes the proof of Lemma \ref{lem:CATx}.  
\end{proof}

Note that in  $\tilde E_{x,v}$,  the subarc $[x,y]$
of $C$ is replaced by the geodesic 
$[\tilde x,\tilde y]:=\gamma_{\tilde x,\tilde y}$.
%%%
On the singular locus $\tilde C(x,v)$ 
of $\tilde E_{x,v}$, we consider the graph structure
inherited from $C$ (and hence from $\ca S(X)$),
except that $\tilde x, \tilde y\in V(\tilde C(x,v))$
and $(\tilde x,\tilde y)\in E(\tilde C(x,v))$.

%%%%%
After all the surgeries at $x$ possibly in the both 
positive and negative singular directions, 
we obtain a new space, denoted by $\tilde E_x$.
Note that the point $\tilde x\in\tilde E_x$
replacing $x$ is no longer singular in the graph
structure of the new singular locus
$\tilde C(x)\subset \tilde E_x$.

%%%%%
In what follows, we shall perform such surgeries finitely many times consistently in the directions of 
$\Sigma_x^{\rm sing}(X)$ at points $x\in V_{\sing}(C)$ so that the surgery parts
cover $V_{\sing}(C)$.  

%%%
First take $\e=\e_{p}>0$ satisfying 
\eqref{eq:non-meeting-VC},
\eqref{eq:non-meeting-VC2},
\eqref{eq:non-meeting-VC3}
and 
%\eqref{eq:inLemma},
\eqref{eq:tildeXyxy'}
%
%\eqref{eq:close-angle}  
%and \eqref{eq:biLipschitz-close} 
for $x=p$, and set $\delta_0=\e_{p}$.
Remember that $S(p,\delta_0)$ does not meet $V(C)$.
We enlarge the radius of the ball $B(p, \delta_0)$, and choose $r_1>\delta_0$ such that during the enlarging,
$S(p,r_1)$ first meets $V_{\sing}(C)$, say at $x$, after
$S(p, \delta_0)$. 
%%%
We call $r_1$ a {\it critical radius} in the surgeries.
%%%
Now we do the above surgery at $x$,
either in the negative direction $-\nabla d_p(x)$, where the surgeries should be carried out inside the annulus 
$A(p,\delta_0,r_1)=B(p,r_1)\setminus {\rm int}\,B(p,\delta_0)$, or in the positive direction $\nabla d_p(x)$
to resolve the singularity at $x$.

We again perform such surgeries at all points 
$x\in S(p,r_1)\cap V_{\sing}(C)$. 
%%%
Here, taking the smallest constant $\e=\e_x$
among all $x$ and all singular directions there, 
we may assume that  those surgeries are carried out 
based on a common metric sphere around $p$.
More precisely, for some $0<\delta_1<r_1-\delta_0$, 
we have 
$V(T(x,v))\subset S(p, r_1+\delta_1)$
(resp. $V(T(x,v))\subset S(p, r_1-\delta_1$))
for all $x\in S(p,r_1)\cap V_{\sing}(C)$
and $v\in \Sigma^{\sing}_{x,+} (C)$ 
(resp. $v\in \Sigma^{\sing}_{x,-} (C)$).
%%% 
We call $\delta_1$ (resp. $\delta_0$)  the 
{\it surgery radius} at $S(p,r_1)$ (resp. at $p$).

%%%
Then we again enlarge the radius of 
$B(p, r_1+\delta_1)$ until the next critical radius 
$r_2$.
Repeating this procedure, we have a possibly infinite 
sequence of critical radii $r_i$,
$$
0<r_1< r_2< \cdots <r_i<\cdots,
$$
and surgery radii $\delta_i$ at $S(p,r_i)$ 
with  $r_i+\delta_i<r_{i+1}-\delta_{i+1}$
such that the $X$-annulus 
$A^X(p, r_i+\delta_i, r_{i+1}-\delta_{i+1})$
does not meet $V_{\sing}(C)$.
Note also that the number of surgeries at 
points of $S(p, \delta_i)$ is bounded by 
the uniform constant $N_v-2$.
%%%%%%%

We show that one can cover $V_{\sing}(C)$
after performing surgeries as above finitely many times. 
Suppose $r_*=\lim_{i\to\infty} r_i<r$. 
From construction, $S(p,r_*)$ meets $V_{\sing}(C)$.
We again do surgeries at points of $S(p,r_*)\cap V_{\sing}(C)$. 
For the surgeries in the negative direction at those points,
we can make them consistent with the previous surgeries
since our procedure is done based on metric spheres
around $p$.
%%%%
This shows that after finitely many such surgeries, 
we can resolve all singular vertices in $V(C)$.
%%%%
Let $0<r_1< r_2< \cdots <r_J<r$ be critical radii, and 
$\delta_i$\, $(0\le i\le J)$  surgery radii, where we may assume that $A(p,r_J+\delta_J,r)$
does not meet $V_{\sing}(C)$
by taking slightly larger $r$ if necessary.
%%%

Let $\ca K:=\{ K_{n,i}:=K(x_n,v_{n,i})\,|\, 1\le n\le M, 1\le i\le L_n\}$ be the set of all cone-like domains in $E$ constructed as above for 
$x_n\in V_{\sing}(C)$ and $v_{n,i}\in\Sigma_x^{\sing}(C)$
which arise in the course of 
the surgeries.
%%%%
%Now we have a finite sequence of pairwise disjoint 
%closed subintervals 
Set $I_0:=[0, \delta_0]$,
$I_j:=[r_j-\delta_j, r_j+\delta_j]$ and
$A_j:=d_p^{-1}(I_j)$ \,$(1\le j\le J)$.
%of $[0,r]$ such that any $K_n\in\ca K$ is contained
%in the annulus $A_j:=d_p^{-1}(I_j)$ for some $0\le j\le J$.
From construction, we have the following for 
 every $K_n\in\ca K$. 
\begin{itemize}
\item  $K_n\in\ca K$ is convex in $E\,;$
\item  $K_n$ is contained in some $A_j\,;$
\item $K_n$ and $K_{n'}$ do not have intersection in their interiors for all $n\neq n'\,;$
\item the number of $K_n$ contained in $A_j$ 
is at most $N_v-2$ for each $1\le j\le J$.
\end{itemize}

Let $\tilde E$ be the result of those surjeries,
and let $\tilde C$ be the singular locus of 
$\tilde E$, with graph structure $V(\tilde C)$,
$E(\tilde C)$ defined as above.
Note that $V(\tilde C)$ is finite and $V_{\sing}(\tilde C)$ is empty.
%%%%%%

\begin{lem} \label{lem:tildeE-CAT}
$\tilde E$ is a $\CAT(\kappa)$-space.
\end{lem}

\begin{proof}
From the construction  and Proposition \ref{prop:CAT(1)-vertexK}, we have 
\begin{itemize}
 \item for every edge $e$ of 
 $E(\tilde C)$, the condition (A) holds and $e$ has finite turn variation\,$;$
  \item $\Sigma_{\tilde y}(\tilde E)$ is $\CAT(1)$ for 
    every $\tilde y\in V(\tilde C)$ . 
\end{itemize}
Consider any triangulation of $\tilde E$ extending
$V(\tilde C)$ and $E(\tilde C)$ by adding 
geodesic edges if necessary.
Now, we are ready to apply  Theorem \ref{thm:BB-gluing} to this triangulation to conclude that  $\tilde E$ is $\CAT(\kappa)$.
\end{proof}

Now we are going to show the Gromov-Hausdorff convergence $\tilde E\to E$ as $\e_0\to 0$.
%%%%

For each $K(x_n,v_{n,i})\in \ca K$, we fix any element 
$y_{n,i} \in V(T(x_n, v_{n,i}))$, and let 
$\gamma_{n,i}:[0,|x_n,y_{n,i}|_E]\to E$ be an $E$-geodesic
from $x_n$ to $y_{n,i}$. 

%%%%%
Define $\varphi:\tilde E\to E$ as follows.
Let $\varphi$ be identical outside the surgery part.
For every $\tilde z\in {\rm int}\,\tilde K(x_n,v_{n,i})$,
we let 
$$
\varphi(\tilde z):=\gamma_{n,i}(|\tilde x_n, \tilde z|).
$$
 Since ${\rm diam}(K(x_n,v_{n,i}))<\tau(\e_0)$,
the image of $\varphi$ is $\tau(\e_0)$-dense in 
$E$.

For arbitrary $\tilde z,\tilde z'\in \tilde E$, set
$z=\varphi(\tilde z)$, $z'=\varphi(\tilde z')$, and 
choose an $E$-shortest
curve $\gamma:[0, |z, z'|_E]\to E$ between $z$ and $z'$.
Suppose first that $d_p(\gamma(t))$ takes a local 
minimum or local maximum. Then we see that 
$\gamma$ is vertical, and hence an $X$-geodesic.
Moreover, $\gamma$ intersects $C$ almost perpendicularly with at most $N_v-2$ points
(Sublemma \ref{slem:number-C}). 
This implies that $\gamma$ meets at most
$N_v-2$ elements of $\ca K$.
Therefore from Lemma \ref{lem:thin-est}, we have 
 $$
 ||z,z'|_{E}-|\tilde z, \tilde z'||<\tau(\e_0)(r +N_v-2).
 $$

Now we assume that $d_p(\gamma(t))$ is strictly monotone.
Let $\ca K_\gamma$
be set of all $K(x_n, v_{n,i})\in \ca K$  meeting $\gamma$.
% with $\pa K_{ni}$.
For simplicity, we renumber elements of $\ca K_\gamma$ as 
$\ca K_\gamma=\{ K_i\,|\, 1\le i\le I\}$.
%$\x_ny_n$ and $x_ny_n'$.
%
%%%%%%%
Let $\ca K_j$ be the set of all $K_n\in\ca K_\gamma$ contained in 
$A_j$. If $\gamma$ meets $K_n\in \ca K_j$
with $\{ z_n,z_n'\}=\gamma\cap\pa K_n$, then 
from Lemma \ref{lem:thin-est} we have 
$$
|\varphi^{-1}(z_n),\varphi^{-1}(z_n')|_{\tilde E}
-|z_n,z_n'|_E|<2\tau(\e_0)\delta_j.
$$
%%%%%%%
%
%
%First fix any $K_{i_1}\in\ca K_\gamma$ containing $p$.
%Let $\ca K_1$ be the set of $K_n\in \ca K_\gamma$ meeting
%$K_{i_1}$, and choose $K_{n_1}\in\ca K_1$ with
%maximal diameter among $\ca K_1$.
%Fix any $K_{i_2}\in\ca K_\gamma$ that is closest to $p$
%among $\ca K_\gamma\setminus\ca K_1$.
%Let $\ca K_2$ be the set of $K_n\in \ca K_\gamma\setminus \ca K_1$ meeting
%$K_{i_2}$, and choose $K_{n_2}\in\ca K_2$ with
%maximal diameter among $\ca K_2$.
%In this way, we select $K_{n_1}, K_{n_2},\ldots,K_{n_L}$ together with the decomposition $\ca K_1,\ca K_2,\ldots,\ca K_L$ of $\ca K_\gamma$ with 
%$K_{ni}\in\ca K_i$\, $(1\le i\le L)$.
%Now we simply write as
%%%% 
%$K_{ni}=K(x_{ni},v_{ni})$ and choose any $y_{ni}\in V(T(x_{ni},v_{ni}))$.
%Let $r_{ni}:=|r(y_{ni})-r(x_{ni})|$. 
%%%
%From the construction, we have for each $1\le i\le L$,
%\[
%  \sum_{K_n\in \ca K_{i}}{\rm diam}(K_n)
%      \le 2r_{n_i}(N_v-2).
%\]
It follows  that 
$$
 ||z,z'|_{E}-|\tilde z, \tilde z'||<2r(N_v-1)\tau(\e_0).
 $$
In this way, we conclude that 
$\tilde E$ converges to $E$ as 
$\e_0\to 0$ with respect to the Gromov-Hausdorff
distance, which yields that 
$E$ is a $\CAT(\kappa)$-space.
This completes the proof of Theorem \ref{thm:union-CAT}.
\end{proof}

\begin{proof}[Proof of Theorem \ref{thm:main}(1) 
in Case II]
We consider Case II in the subsection of filling ball of  Section \ref{sec:fill}.
We only have to apply Theorem \ref{thm:union-CAT} for $k=4$ to Case II.
The rest of the argument is similar to that in Case I
given in Section \ref{sec:fill}, and hence omitted.
\end{proof}

 The proof of Corollary \ref{cor:main-zeta} is similar to  
 that of  Theorem \ref{thm:main}(1) in Case II, and hence omitted.

Using Theorem \ref{thm:union-CAT}, we also have 
the following. 

\begin{thm} \label{thm-rem:Thm1.1(1)} 
In Theorem \ref{thm:main}, every union 
${\rm Im} f_{i_1}\cup\cdots\cup {\rm Im} f_{i_k}$ is 
a $\CAT(\kappa)$-space.
\end{thm}
\begin{proof}
The basic idea of the proof of Theorem  
\ref{thm-rem:Thm1.1(1)} is the same as that of Theorem \ref{thm:main}(1) for branched immersed disks.
Set $\Sigma_{p,i_j}:=\Sigma_p({\rm Im}(f_{i_j}))$ $(1\le j\le k)$,
and consider 
$\Sigma:=\Sigma_{p,i_1}\cup\cdots\cup\Sigma_{p,i_k}$.
For each $v\in V(\Sigma_p(X))$ contained in $\Sigma$,
we construct a ruled surface $S$ for which we may assume 
$\CAT(\kappa)$ by taking smaller $r$.
Let $S(v)$ denote the union of all such ruled surfaces $S$.
By Theorem \ref{thm:union-CAT}, $S(v)$ is $\CAT(\kappa)$.
The rest of the argument is the same as before,
and hence omitted.
\end{proof}
%
%\pmed\n
%{\bf Turn of singular locus}.\,
%Theorem \ref{thm:union-CAT} enables us to describe
%the boundedness of total variation of turn of
%the singular locus $\ca S(X)$.
%
%
%\begin{defn} \label{defn:turn-C} \upshape
%We say that $\ca S(X)$ has {\it locally finite turn variation} if  
%for every $p\in\ca S(X)$,
%each singular curve $C$ starting from $p$
%contained in a thin ruled surface $S$ with vertex 
%$p$ as in Section \ref{sec:ruled}
%has finite turn variation in $S$
%for every such an $S$.
%\end{defn}
%
%\begin{proof}[Proof of Theorem \ref{thm:bdd-turn-C}]
%We use the notations in the beginning of this section. 
%Suppose $C$ is a singular curve contained in 
%a ruled surface $S=S_{ij}$ with vertex $p$. 
%%%%
%Consider the graph structure $V(C)$ and $E(C)$
%inherited from that of $\ca S(X)$ as before.
%For each $e\in E(C)$, $\bar e$ is contained in the singular curve
%$C_{ij\ell}$ for some $\ell$.
%%%%
%Thus we have a 
%finite decomposition $C=B_1\cup\cdots \cup B_{N-2}$
%of $C$ into Borel subsets such that each $B_m$ is contained in the singular curve
%$C_{ij\ell_m}=S_{ij}\cap S_{j\ell_m}\cap S_{\ell_m i}$ for 
%some $\ell_m\neq i, j$.
%Applying Theorem \ref{thm:union-CAT} for $k=3$,
%we see that $E_3:=S_{ij}\cup S_{j\ell_m}\cup S_{\ell_m i}$
%is 
%$\CAT(\kappa)$. Theorem \ref{thm:BB-character} then 
%concludes that $C_{ij\ell_m}$ has finite turn variation, and so does $B_m$.
%\end{proof}
% 

\appendix
%%%%%%%%%%
%%%%%%%%%%

\setcounter{equation}{0}

\section{Alexandrov's result on ruled surfaces} \label{sec:append}

Following the ideas of Alexandrov in \cite{alexandrov-ruled},
we prove Theorem \ref{thm:alex-ruled2}.
 As mentioned in Section \ref{sec:intro}, it also follows from 
\cite{PS} in the CAT(0)-setting.

We denote by $D_{\kappa}$
the diameter of $M_{\kappa}^2$.
Recall that a $\CAT(\kappa)$-space
is defined as a $D_{\kappa}$-geodesic space
in which every triangle with perimeter $< 2D_{\kappa}$
is not thicker than its comparison triangle in $M_{\kappa}^2$
with the same side lengths,
where a $D_{\kappa}$-geodesic space
means a metric space in which 
any two points with distance $< D_{\kappa}$
can be joined by a minimal geodesic.
Throughout this appendix, let $X$ be a $\CAT(\kappa)$-space.
%%%%%%%%%%
\subsection{Finite sequences of ruling geodesics}

Let $S$ be a ruled surface in $X$ 
with parametrization $\sigma \colon R \to X$,
where $R = [0,\ell] \times [0,1]$.
 Let $\pi:R\to R_*$ and $p_1 \colon R \to [0,\ell]$ 
be as in Section \ref{sec:ruled-alex}.
%%%%
%%%%%

We give an explicite formulation of the pullback metric $e_\sigma$.
% (see also \cite{alexandrov-ruled}).
%
For $u=(s_0,t_0)$ and $u'=(s_0',t_0')$ with $s_0 <s_0'$ in $R$, let 
$\Delta: s_0\le s_1\le \cdots\le s_n=s_0'$ be a decomposition of $[s_0,s_0']$, 
and set 
$|\Delta|=\max \{|s_i-s_{i-1}|\,|\, 1\le i\le n\}$.
We consider 
%\beq
\begin{align*}    \label{eq:e-sigma-Delta}
e_\sigma^\Delta(\pi(u), &\pi(u')) \\
     & := \inf\left\{ \sum_{i=1}^n |x_{i-1}, x_i|\,|\,x_0=\sigma(u), x_n=\sigma(u'), x_i\in\lambda_{s_i}\right\}.
\end{align*}
%\eeq 
%%%%%
%%%
%For $u, u' \in R$ with $\pi(u)\neq \pi(u')$,
%choose a finite decomposition $\Delta = \{ s_i \}_{i = 0, 1, \dots, n}$
%of $[p_1(u),p_1(u')]$, 
%so that $s_0 = p_1(u)$, $s_n = p_1(u')$, and $s_{i-1} \le s_i$
%for all $i \in \{1, \dots, n\}$.
Choose a sequence $\{ x_i \}_{i = 0, 1, \dots, n}$ in $X$
such that $x_0 = \sigma(u)$, $x_n = \sigma(u')$, 
$x_i \in \lambda_{s_i}$ for all $i \in \{ 1, \dots, n-1 \}$,
and 
$$
e_{\sigma}^{\Delta}(\pi(u),\pi(u')) = \sum_{i=1}^n | x_{i-1}, x_i |.
$$
We call such a sequence $\{ x_i \}_{i = 0, 1, \dots, n}$
a \emph{$\Delta$-minimizing chain along $S$ from $\sigma(u)$ to $\sigma(u')$}.
Notice that
possibly we have $x_{i-1} = x_i$
for some $i \in \{ 1, \dots, n \}$.
We set $\gamma^\Delta:=\bigcup x_{i-1}x_i$, and call it a
{\it $\Delta$-minimizing broken geodesic} 
in $X$ from $\sigma(u)$ to $\sigma(u')$, which  realizes 
$L(\gamma^\Delta)=e_\sigma^\Delta(\pi(u),\pi(u'))$.

%%%%%%%%

\begin{lem} \label{lem:esigma-decomp}
Under the above situation, we have the following $:$
% for all $\pi(u), \pi(u')\in R_*:$
\begin{enumerate}
 \item \[         e_\sigma(\pi(u), \pi(u'))=\sup_{\Delta} e_\sigma^\Delta(\pi(u), \pi(u')),
           \]
 where $\Delta$ runs over all decompositions of $[s_0,s_0']\,;$
\item For any sequence $\Delta_n$ of decompositions of
 $[s_0,s_0']$ satisfying 
$\lim_{n\to\infty}|\Delta_n|=0$, we have
\[
       e_\sigma(\pi(u), \pi(u'))=\lim_{n\to\infty} e_\sigma^{\Delta_n}(\pi(u), \pi(u'))\,;
\]
%\item For any sublimit $\gamma$ of $\gamma^{\Delta_n}$,
%$\gamma_*:=\sigma_*^{-1}(\gamma)$ provides a
%shortest curve on $(R_*, e_\sigma)$ from $\pi(u)$ to $\pi(u')$, which has a lift in $R$.
\end{enumerate}
\end{lem}
\begin{proof}
By Proposition \ref{prop:e-sigma}, there is a shortest curve
$c_{0*}:[0,1]\to (R_*,e_\sigma)$ from $\pi(u)$ to $\pi(u')$ together with its lift $c_0$.
Set $\gamma_0(t):=\sigma_*\circ c_{0*}(t)$.
%and $s_0:=p_1(u), s_0':=p_1(u')$.
For any decompositon $\Delta=\{ s_i\}_{i=1}^N$ of $[s_0,s_0']$,  
take $t_i\in [0,1]$ such that $\gamma_0(t_i)\in \lambda_{s_i}$.
Then in view of Proposition \ref{prop:e-sigma-shortest}, 
we have 
$$
 e_\sigma^\Delta(\pi(u),\pi(u'))\le\sum_{i=1}^N 
    |\gamma_0(t_{i-1}), \gamma_0(t_i)|\le L(\gamma_0)=L(c_{0*})= e_\sigma(\pi(u),\pi(u')).
$$
Thus we have $\sup_{\Delta} e_\sigma^\Delta(\pi(u),\pi(u'))\le e_\sigma(\pi(u),\pi(u'))$.
\psmall\n

%%%%
Let  $\{ \Delta_n\} $ be a sequence of decompositions
 of $[s_0,s_0']$ with 
$\lim_{n\to\infty}|\Delta_n|=0$ such that
\[
\lim_{n\to\infty} e_\sigma^{\Delta_{n}}(\pi(u),\pi(u'))=
\liminf_{|\Delta|\to 0} e_\sigma^\Delta(\pi(u), \pi(u')).
\]
%where we may assume that $\Delta_{n+1}$ is a refinement 
%of $\Delta_n$ for each $n$
Let $\gamma_n:[0,1]\to X$ be a $\Delta_n$-minimizing 
broken geodesic in $X$. 
%Take a subseqeunce $\{ n_k\}$
%with 
%$$
%\lim_{k\to\infty} e_\sigma^{\Delta_{n_k}}(\pi(u),\pi(u'))=\liminf_{n\to\infty} e_\sigma^{\Delta_{n}}(\pi(u),\pi(u')).
%$$
Passing to a subseqeunce, we may assume that 
$\gamma_{n}$ converges to a curve $\gamma:[0,1]\to X$.
From $|\Delta_n|\to 0$, it follows that 
$\gamma([0,1])\subset S$.

\begin{slem} \label{slem:lift-appendix}
$\gamma$ has a lift in $R$ from $u$ to $u'$.
\end{slem}
\begin{proof}
We may assume ${\rm Sing}(\sigma)$ is empty.
Let 
\[
\Delta_n: s_0=s_{n,0}\le s_{n,1}\le\cdots\le s_{n,k_n}=s_0',
\]
and  $\gamma_n=\gamma^{\Delta_n}:=\bigcup_{i=1}^{k_n} x_{n,i-1}x_{n,i}$  with $x_{n,i}\in\lambda_{s_{n,i}}$. 
Choose $a_{n,i}\in I_{s_{n,i}}$ with $\sigma(a_{n,i})=x_{n,i}$
\,$(1\le i\le k_n-1)$, and 
consider the Euclidean broken geodesic 
$c_n:=\bigcup_{i=1}^{k_n} a_{n,i-1} a_{n,i}$.
Note that $c_n$ is monotone, and $\sigma\circ c_n$ also
converges to $\gamma$ as $n\to\infty$.
We show that a subsequence of $c_n$ converges to a curve $c$, which is a lift of $\gamma$.
We do not know if $L(\sigma\circ c_n)$ is uniformly bounded or even if it is finite, which is the only difference
from Proposition \ref{prop:e-sigma}.

Since the basic strategy is the same as the proof of Proposition \ref{prop:e-sigma}, we present only an 
outline.
Let $J_0$ be a countable dense subset of $J=[s_0,s_0']$. 
For each $s\in J$, choose a point $c_n(t_n(s))$ of 
$c_n$ with $c_n(t_n(s))\in I_s$.
Now we have a subsequence $\{ m\}$ of $\{ n\}$
such that $c_m(t_m(s))$ converges to a point
$x(s)\in I_s$ for every $s\in J_0$.
We consider the limit set, say $LS(\{ c_m\})$, 
 of the sequence $\{ {\rm Im}(c_m)\}_m$,
and set
\[
    E_s:=LS(\{ c_m\})\cap I_s,
\]
as in the proof of Proposition \ref{prop:e-sigma}.
Then we have the decomposition 
$J=J_1\cup J_2$, %depending on the carginality of 
%$E_s$, 
where 
\[
    J_1= \{ s\in J\,| \text{$E_s$ is a single point} \}, \quad  
    J_{2}=J\setminus  J_{1}.
\] 
 In the same way, we have the conclusions $(1)\sim (4)$ 
in the proof of Proposition \ref{prop:e-sigma}.
Here it should be remarked that the following 
holds as well,
\[
     \sum_{s\in J_2} \sigma(E_s) \le L(\gamma).
\]
Thus as before, we obtain a monotone continuous parametrization 
on the union of points and segments
$\{ E_s\,|\,s\in J\}$, which provides a lift of $\gamma$
from $u$ to $u'$. 
\end{proof}

By Sublemma \ref{slem:lift-appendix}, we conclude
\[
e_\sigma(\pi(u),\pi(u'))\le L(\gamma)\le \lim_{n\to\infty} L(\gamma_{n})
=\liminf_{n\to\infty} e^{\Delta_n}_\sigma(\pi(u),\pi(u')).
\]
This completes the proof.
%%%%%%
\end{proof}  
%%%%%%%%

%%%%%%%%%%%%%%%%%%%%%%%%%%%%%%
From the choice of a $\Delta$-minimizing chain along $S$,
%and from the first variation formula of distance functions
%on $\CAT(\kappa)$-spaces
we derive the following:

\begin{lem}\label{lem:overpi}
In the setting discussed above,
let $\{ x_i \}_{i = 0, 1, \dots, n}$ be a 
$\Delta$-minimizing chain along $S$ from $\sigma(u)$ to $\sigma(u')$.
Then
for each $i \in \{ 1, \dots, n-1 \}$ and
for each $t \in \{ 0, 1 \}$
we have
\[
\angle x_{i-1}x_i\lambda_{s_i}(t) + \angle \lambda_{s_i}(t)x_ix_{i+1} \ge \pi,
\]
whenever
$| x_{i-1}, x_i |, | x_i, x_{i+1} | < D_{\kappa}$,
and the angles
$\angle x_{i-1}x_i\lambda_{s_i}(t)$ and $\angle \lambda_{s_i}(t)x_ix_{i+1}$
can be defined.
\end{lem}
\begin{proof}
First we show the conclusion in the case $t=0$.
Set
$\theta_i^{-} := \angle x_{i-1}x_i\lambda_{s_i}(0)$
and
$\theta_i^{+} := \angle \lambda_{s_i}(0)x_ix_{i+1}$.
Take $t_i \in (0,1]$ with $x_i = \lambda_{s_i}(t_i)$,
where we may assume $t_i\neq 0$.
If we put 
$h(\epsilon) := | \lambda_{s_i}(t_i-\epsilon), x_{i-1} |
+| \lambda_{s_i}(t_i-\epsilon), x_{i+1} |$ for small $\e>0$,
then by the first variation formula (see e.g., \cite[Corollary II.3.6]{bridson-haefliger})
together with the $\Delta$-minimizing property of 
$\{ x_i \}_{i = 0, 1, \dots, n}$,
we have
\[
0 \le \lim_{\epsilon \to 0+}
\frac{h(\epsilon) - h(0)}{\epsilon}
= - (\cos \theta_i^- + \cos \theta_i^+).
\]
This implies $\theta_i^- + \theta_i^+ \ge \pi$.
Similarly, we see the inequality for $t = 1$.
\end{proof}

%%%%%%%%%%%%%%%%%%%%%%%%%%%%%

Let $u_*:=\pi(u), v_*:=\pi(v), w_*:=\pi(w)$ be distinct points in $R_*$.
Assume for a while that
\[
p_1(u) \le p_1(v) \le p_1(w),
\]
and choose a decomposition
$\Delta = \{ s_i \}_{i = 0, 1, \dots, n}$
of $[p_1(u),p_1(w)]$ such that
for some $m \in \{ 1, \dots, n-1 \}$ we have $p_1(v) = s_m$.
Let $\Delta' := \{ s_i \}_{i = 0, 1, \dots, m}$
be the decomposition of $[p_1(u),p_1(v)]$,
and $\Delta'' := \{ s_{m+i} \}_{i = 0, 1, \dots, n-m}$
the decomposition of $[p_1(v),p_1(w)]$.
Take a $\Delta'$-minimizing chain $\{ y_i \}_{i = 0, 1, \dots, m}$
along $S$ from $\sigma(u)$ to $\sigma(v)$,
and a $\Delta''$-minimizing chain $\{ y_{m+i} \}_{i = 0, 1, \dots, n-m}$
along $S$ from $\sigma(v)$ to $\sigma(w)$,
and a $\Delta$-minimizing chain $\{ z_i \}_{i = 0, 1, \dots, n}$
along $S$ from $\sigma(u)$ to $\sigma(w)$.
Assume in addition that we have
\[
e_{\sigma}^{\Delta'}(u_*, v_*) + e_{\sigma}^{\Delta''}(v_*,w_*) + e_{\sigma}^{\Delta}(w_*,u_*)
< 2D_{\kappa}.
\]
Set $x:=\sigma(u)$, $y:=\sigma(v)$ and $z:=\sigma(w)$. 
Let $B^{\Delta}(xy)$ be
the broken geodesic 
$\bigcup_{i=1}^m y_{i-1}y_i$ in $X$ joining $x$ and $y$,
$B^{\Delta}(yz)$
the broken geodesic 
$\bigcup_{i=1}^{n-m} y_{m+i-1}y_{m+i}$ in $X$ joining $y$ and $z$,
$B^{\Delta}(zx)$
the broken geodesic 
$\bigcup_{i=1}^{n} z_{i-1}z_i$ in $X$ joining $z$ and $x$.
We denote by $P^{\Delta}(xyz)$
the polygon in $X$ defined by
\[
P^{\Delta}(xyz) :=
B^{\Delta}(xy) \cup B^{\Delta}(yz) \cup B^{\Delta}(zx),
\]
and we call $P^{\Delta}(xyz)$ the 
\emph{$\Delta$-minimizing chain triple along $S$}.
We denote by $\theta_x^{\Delta}(y,z)$
the angle at $x$ in $X$ between $B^{\Delta}(xy)$ and $B^{\Delta}(zx)$,
by $\theta_y^{\Delta}(z,x)$
the angle at $y$ in $X$ between $B^{\Delta}(yz)$ and $B^{\Delta}(xy)$,
by $\theta_z^{\Delta}(x,y)$
the angle at $z$ in $X$ between $B^{\Delta}(zx)$ and $B^{\Delta}(yz)$.

%%%%%tikz picture%%%%%
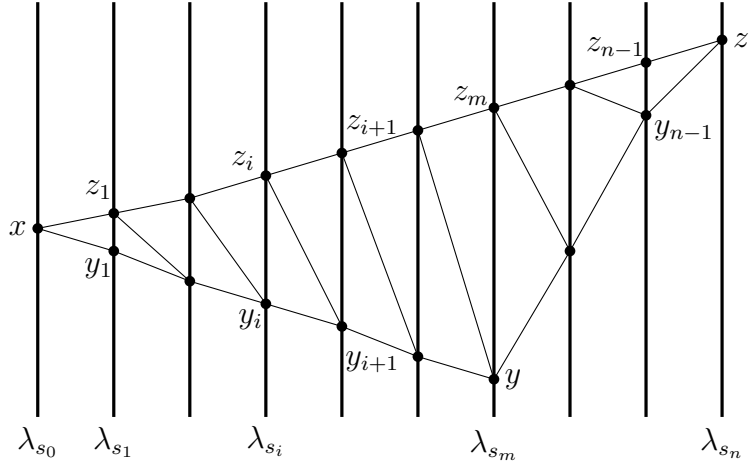
\begin{center}
\begin{tikzpicture}
[scale = 1]
\draw [-, very thick] (-5,3)--(-5,-2.5);
\draw [-, very thick] (-4,3)--(-4,-2.5);
\draw [-, very thick] (-3,3)--(-3,-2.5);
\draw [-, very thick] (-2,3)--(-2,-2.5);
\draw [-, very thick] (-1,3)--(-1,-2.5);
\draw [-, very thick] (0,3)--(0,-2.5);
\draw [-, very thick] (1,3)--(1,-2.5);
\draw [-, very thick] (2,3)--(2,-2.5);
\draw [-, very thick] (3,3)--(3,-2.5);
\draw [-, very thick] (4,3)--(4,-2.5);
\draw [-] (-5,0)--(-4,-0.3);
\draw [-] (-4,-0.3)--(-3,-0.7);
\draw [-] (-3,-0.7)--(-2,-1);
\draw [-] (-2,-1)--(-1,-1.3);
\draw [-] (-1,-1.3)--(0,-1.7);
\draw [-] (0,-1.7)--(1,-2);
\draw [-] (1,-2)--(2,-0.3);
\draw [-] (2,-0.3)--(3,1.5);
\draw [-] (3,1.5)--(4,2.5);
\draw [-] (4,2.5)--(3,2.2);
\draw [-] (3,2.2)--(2,1.9);
\draw [-] (2,1.9)--(1,1.6);
\draw [-] (1,1.6)--(0,1.3);
\draw [-] (0,1.3)--(-1,1);
\draw [-] (-1,1)--(-2,0.7);
\draw [-] (-2,0.7)--(-3,0.4);
\draw [-] (-3,0.4)--(-4,0.2);
\draw [-] (-4,0.2)--(-5,0);
\draw [-] (-4,0.2)--(-3,-0.7);
\draw [-] (-3,0.4)--(-2,-1);
\draw [-] (-2,0.7)--(-1,-1.3);
\draw [-] (-1,1)--(0,-1.7);
\draw [-] (0,1.3)--(1,-2);
\draw [-] (1,1.6)--(2,-0.3);
\draw [-] (2,1.9)--(3,1.5);
\draw (-5,-2.3) node[circle] [below] {$\lambda_{s_0}$};
\draw (-4,-2.3) node[circle] [below] {$\lambda_{s_1}$};
\draw (-2,-2.3) node[circle] [below] {$\lambda_{s_i}$};
\draw (1,-2.3) node[circle] [below] {$\lambda_{s_m}$};
\draw (4,-2.3) node[circle] [below] {$\lambda_{s_n}$};
\fill (-5,0) circle (2pt) node [left] {$x$};
\fill (1,-2) circle (2pt) node [right] {$y$};
\fill (4,2.5) circle (2pt) node [right] {$z$};
\fill (-4,-0.3) circle (2pt) node [left] {};
\fill (-3,-0.7) circle (2pt) node [left] {};
\fill (-2,-1) circle (2pt) node [left] {};
\fill (-1,-1.3) circle (2pt) node [left] {};
\fill (0,-1.7) circle (2pt) node [left] {};
\fill (2,-0.3) circle (2pt) node [left] {};
\fill (3,1.5) circle (2pt) node [left] {};
\fill (3,2.2) circle (2pt) node [left] {};
\fill (2,1.9) circle (2pt) node [left] {};
\fill (1,1.6) circle (2pt) node [left] {};
\fill (0,1.3) circle (2pt) node [left] {};
\fill (-1,1) circle (2pt) node [left] {};
\fill (-2,0.7) circle (2pt) node [left] {};
\fill (-3,0.4) circle (2pt) node [left] {};
\fill (-4,0.2) circle (2pt) node [left] {};
\draw (-4.2,-0.1) node[circle] [below] {$y_1$};
\draw (-4.2,0.1) node[circle] [above] {$z_1$};
\draw (-2.2,-0.8) node[circle] [below] {$y_i$};
\draw (-2.3,0.5) node[circle] [above] {$z_i$};
\draw (-0.6,-1.2) node[circle] [below] {$y_{i+1}$};
\draw (-0.6,0.8) node[circle] [above] {$z_{i+1}$};
\draw (3.5,1.9) node[circle] [below] {$y_{n-1}$};
\draw (2.6,1.8) node[circle] [above] {$z_{n-1}$};
\draw (0.7,1.3) node[circle] [above] {$z_m$};
\end{tikzpicture}
\end{center}
%%%%%%%%%%

In the model surface $M_{\kappa}^2$,
we define a comparison polygon $\tilde{P}^{\Delta}(xyz)$ 
for $P^{\Delta}(xyz)$ as follows:
Let $\triangle \tilde{x}\tilde{y}_1\tilde{z}_1$ 
and $\triangle \tilde{y}_{n-1}\tilde{z}_{n-1}\tilde{z}$ be 
comparison triangles in $M_{\kappa}^2$
for $\triangle xy_1z_1$ and for $\triangle y_{n-1}z_{n-1}z$,
respectively.
For each $i \in \{ 1, \dots, n-1 \}$,
take comparison triangles 
$\triangle \tilde{y}_i\tilde{y}_{i+1}\tilde{z}_i$ and 
$\triangle \tilde{y}_{i+1}\tilde{z}_i\tilde{z}_{i+1}$ 
in $M_{\kappa}^2$
for $\triangle y_iy_{i+1}z_i$ and for $\triangle y_{i+1}z_iz_{i+1}$, respectively,
and then glue all the comparison triangles in $M_{\kappa}^2$ along
$\tilde{y}_i\tilde{z}_i$,
and along $\tilde{y}_{i+1}\tilde{z}_i$,
for all $i \in \{ 1, \dots, n-1 \}$.
Let $\tilde{B}^{\Delta}(xy)$ be
the broken geodesic 
$\bigcup_{i=1}^m \tilde{y}_{i-1}\tilde{y}_i$ in $M_{\kappa}^2$ 
joining $\tilde{x}$ and $\tilde{y}$,
$\tilde{B}^{\Delta}(yz)$
the broken geodesic 
$\bigcup_{i=1}^{n-m} \tilde{y}_{m+i-1}\tilde{y}_{m+i}$ in $M_{\kappa}^2$
joining $\tilde{y}$ and $\tilde{z}$,
$\tilde{B}^{\Delta}(zx)$
the broken geodesic 
$\bigcup_{i=1}^{n} \tilde{z}_{i-1}\tilde{z}_i$ in $M_{\kappa}^2$
joining $\tilde{z}$ and $\tilde{x}$.
Then we put
\[
\tilde{P}^{\Delta}(xyz) :=
\tilde{B}^{\Delta}(xy) \cup \tilde{B}^{\Delta}(yz) \cup \tilde{B}^{\Delta}(zx),
\]
and we call $\tilde{P}^{\Delta}(xyz)$ a
\emph{comparison $\Delta$-minimizing chain triple in $M_{\kappa}^2$
for $P^{\Delta}(xyz)$}.
We denote by $\tilde{\theta}_x^{\Delta}(y,z)$
the angle at $\tilde{x}$ in $M_{\kappa}^2$ 
between $\tilde{B}^{\Delta}(xy)$ and $\tilde{B}^{\Delta}(zx)$,
by $\tilde{\theta}_y^{\Delta}(z,x)$
the angle at $\tilde{y}$ in $M_{\kappa}^2$ 
between $\tilde{B}^{\Delta}(yz)$ and $\tilde{B}^{\Delta}(xy)$,
by $\tilde{\theta}_z^{\Delta}(x,y)$
the angle at $\tilde{z}$ in $M_{\kappa}^2$ 
between $\tilde{B}^{\Delta}(zx)$ and $\tilde{B}^{\Delta}(yz)$.
Note that 
\[
\theta_x^{\Delta}(y,z) \le \tilde{\theta}_x^{\Delta}(y,z),
\quad
\theta_y^{\Delta}(z,x) \le \tilde{\theta}_y^{\Delta}(z,x),
\quad
\theta_z^{\Delta}(x,y) \le \tilde{\theta}_z^{\Delta}(x,y).
\]
From Lemma \ref{lem:overpi}
we derive the following concavity of $\tilde{P}^{\Delta}(xyz)$
except the vertices $\tilde{x}, \tilde{y}, \tilde{z}$:
Namely,
for each $i \in \{ 1, \dots, n-1 \}$ with $i \neq m$
the inner angle at $\tilde{y}_i$ in $\tilde{P}^{\Delta}(xyz)$
is at least $\pi$;
moreover,
for each $i \in \{ 1, \dots, n-1 \}$
the inner angle at $\tilde{z}_i$ in $\tilde{P}^{\Delta}(xyz)$
is at least $\pi$.

By stretching the comparison $\Delta$-minimizing chain triple
$\tilde{P}^{\Delta}(xyz)$ at the concave vertices,
we obtain a triangle $\triangle \bar{x}\bar{y}\bar{z}$ in $M_{\kappa}^2$
whose side-lengths satisfy
\[
| \bar{x}, \bar{y} | = e_{\sigma}^{\Delta'}(u_*,v_*),
\quad
| \bar{y}, \bar{z} | = e_{\sigma}^{\Delta''}(v_*,w_*),
\quad
| \bar{z}, \bar{x} | = e_{\sigma}^{\Delta}(w_*,u_*).
\]
We call $\triangle \bar{x}\bar{y}\bar{z}$ a
\emph{comparison $\Delta$-minimizing stretched triangle in $M_{\kappa}^2$
for $P^{\Delta}(xyz)$},
and we denote it by $\bar{P}^{\Delta}(xyz)$.
We denote by $\bar{\theta}_x^{\Delta}(y,z)$
the angle $\angle \bar{y}\bar{x}\bar{z}$ at $\bar{x}$ in $M_{\kappa}^2$ 
between $\bar{x}\bar{y}$ and $\bar{z}\bar{x}$,
by $\bar{\theta}_y^{\Delta}(z,x)$
the angle $\angle \bar{z}\bar{y}\bar{x}$ at $\bar{y}$ in $M_{\kappa}^2$ 
between $\bar{y}\bar{z}$ and $\bar{x}\bar{y}$,
and by $\bar{\theta}_z^{\Delta}(x,y)$
the angle $\angle \bar{x}\bar{z}\bar{y}$ at $\bar{z}$ in $M_{\kappa}^2$ 
between $\bar{z}\bar{x}$ and $\bar{y}\bar{z}$.
Let $\bar{y}_i \in \bar{x}\bar{y}$
and $\bar{z}_i \in \bar{x}\bar{z}$,
$i \in \{ 1, \dots, n-1 \}$,
be the points corresponding to $\tilde{y}_i$ and to $\tilde{z}_i$,
respectively.
Since $\tilde{P}^{\Delta}(xyz)$ is concave except the vertices,
the Alexandrov stretching lemma
(see e.g., \cite[Lemma I.2.16]{bridson-haefliger})
leads to the following:

\begin{lem}\label{lem:chain-comparison}
Under the setting discussed above,
we have
\[
\tilde\theta_x^{\Delta}(y,z) \le \bar{\theta}_x^{\Delta}(y,z),
\quad
\tilde\theta_y^{\Delta}(z,x) \le \bar{\theta}_y^{\Delta}(z,x),
\quad
\tilde\theta_z^{\Delta}(x,y) \le \bar{\theta}_z^{\Delta}(x,y).
\]
Moreover, 
for all $i \in \{ 1, \dots, n-1 \}$
we have
$| y_i, z_i | \le | \bar{y}_i, \bar{z}_i |$.
\end{lem}

Let $y_j \in B^{\Delta}(xy) \setminus \{ x, y \}$ be a broken point
for $j \in \{ 1, \dots, m-1 \}$,
$y_k \in B^{\Delta}(yz) \setminus \{ y, z \}$ a broken point
for $k \in \{ m+1, \dots, n-1 \}$,
and
$z_l \in B^{\Delta}(zx) \setminus \{ z, x \}$ a broken point
for $l \in \{ 1, \dots, n-1 \}$.
Assume that the broken points
$y_j$, $y_k$, and $z_l$ are distinct to each other.
Choose four $\Delta$-minimizing chain triples
$P^{\Delta}(xy_jz_l)$,
$P^{\Delta}(y_jyy_k)$,
$P^{\Delta}(z_ly_kz)$,
and
$P^{\Delta}(y_jy_kz_l)$ along $S$,
and take
comparison $\Delta$-minimizing stretched triangles
$\bar{P}^{\Delta}(xy_jz_l)$,
$\bar{P}^{\Delta}(y_jyy_k)$,
$\bar{P}^{\Delta}(z_ly_kz)$,
and
$\bar{P}^{\Delta}(y_jy_kz_l)$
in $M_{\kappa}^2$
for $P^{\Delta}(xy_jz_l)$,
$P^{\Delta}(y_jyy_k)$,
$P^{\Delta}(z_ly_kz)$,
and
$P^{\Delta}(y_jy_kz_l)$,
respectively.

From Lemma \ref{lem:chain-comparison} we derive the following monotonicity:

\begin{lem}\label{lem:delta-monotone}
Under the setting discussed above,
we
have
\[
\bar{\theta}_x^{\Delta}(y_j,z_l)
\le \bar{\theta}_x^{\Delta}(y,z),
\quad
\bar{\theta}_y^{\Delta}(y_k,y_j)
\le \bar{\theta}_y^{\Delta}(z,x),
\quad
\bar{\theta}_z^{\Delta}(z_l,y_k)
\le \bar{\theta}_z^{\Delta}(x,y).
\]
\end{lem}

\begin{proof}
Gluing the triangles
$\bar{P}^{\Delta}(xy_jz_l) = \triangle \bar{x}\bar{y}_j\bar{z}_l$,
$\bar{P}^{\Delta}(y_jyy_k) = \triangle \bar{y}_j\bar{y}\bar{y}_k$,
$\bar{P}^{\Delta}(z_ly_kz) = \triangle \bar{z}_l\bar{y}_k\bar{z}$,
and
$\bar{P}^{\Delta}(y_jy_kz_l) = \triangle \bar{y}_j\bar{y}_k\bar{z}_l$
in $M_{\kappa}^2$
along the edges 
$\bar{y}_j\bar{y}_k$, $\bar{y}_k\bar{z}_l$, and $\bar{z}_l\bar{y}_j$,
we obtain a hexagon
$\bar{x}\bar{y}_j\bar{y}\bar{y}_k\bar{z}\bar{z}_l$ in $M_{\kappa}^2$
whose side-lengths satisfy
$| \bar{x}, \bar{y}_j | + | \bar{y}_j, \bar{y} | = e_{\sigma}^{\Delta'}(u_*,v_*)$,
$| \bar{y}, \bar{y}_k | + | \bar{y}_k, \bar{z} | = e_{\sigma}^{\Delta''}(v_*,w_*)$,
and
$| \bar{z}, \bar{z}_l | + | \bar{z}_l, \bar{x} | = e_{\sigma}^{\Delta}(w_*,u_*)$.
By Lemmas \ref{lem:overpi} and \ref{lem:chain-comparison},
we have
\begin{align*}
\pi &\le \theta_{y_j}^{\Delta}(x,z_l) + \theta_{y_j}^{\Delta}(z_l,y_k) +
\theta_{y_j}^{\Delta}(y_k,y) \\
&\le \bar{\theta}_{y_j}^{\Delta}(x,z_l) + \bar{\theta}_{y_j}^{\Delta}(z_l,y_k) +
\bar{\theta}_{y_j}^{\Delta}(y_k,y).
\end{align*}
Similarly,
we have
\begin{align*}
\pi &\le \bar{\theta}_{y_k}^{\Delta}(y,y_j) + \bar{\theta}_{y_k}^{\Delta}(y_j,z_l) +
\bar{\theta}_{y_k}^{\Delta}(z_l,z), \\
\pi &\le \bar{\theta}_{z_l}^{\Delta}(z,y_k) + \bar{\theta}_{z_l}^{\Delta}(y_k,y_j) +
\bar{\theta}_{z_l}^{\Delta}(y_j,x).
\end{align*}
By stretching the hexagon $\bar{x}\bar{y}_j\bar{y}\bar{y}_k\bar{z}\bar{z}_l$
at the concave vertices $\bar{y}_j$, $\bar{y}_k$, and $\bar{z}_l$,
we obtain a comparison $\Delta$-minimizing stretched triangle $\bar{P}^{\Delta}(xyz)$
in $M_{\kappa}^2$ for $P^{\Delta}(xyz)$. 
The Alexandrov stretching lemma
(see e.g., \cite[Lemma I.2.16]{bridson-haefliger})
leads to the desired inequalities.
\end{proof}

From Lemma \ref{lem:chain-comparison} we also derive the following:

\begin{lem}\label{lem:uniquechain}
Let $u_*, u_*' \in R_*$ be distinct points.
%Take $u, u' \in R$ with $\sigma(u) = x$, $\sigma(u') = x'$.
Assuming $p_1(u) \le p_1(u')$,
we choose a decomposition $\Delta = \{ s_i \}_{i = 0, 1, \dots, n}$
of $[p_1(u),p_1(u')]$.
 If $e_{\sigma}^{\Delta}(u_*, u_*') < D_{\kappa}$,
then a $\Delta$-minimizing chain $\{ x_i \}_{i = 0, 1, \dots, n}$ 
along $S$ from $\sigma(u)$ to $\sigma(u')$ is uniquely determined.
\end{lem}

\begin{proof} Let $x:=\sigma(u), x':=\sigma(u')$, and 
suppose that
two distinct 
$\Delta$-minimizing chains $\{ x_i \}_{i = 0, 1, \dots, n}$ and 
$\{ y_i \}_{i = 0, 1, \dots, n}$
along $S$ from $x$ to $x'$
satisfy $x_m \neq y_m$ for $m \in \{ 1, \dots, n-1 \}$.
Then for the $\Delta$-minimizing chain triple $P^{\Delta}(xy_mx')$ along $S$
we see that
a comparison $\Delta$-minimizing stretched triangle $\bar{P}^{\Delta}(xy_mx')$
degenerates in $M_{\kappa}^2$.
Hence we have $\bar{x}_m = \bar{y}_m$.
On the other hand,
Lemma \ref{lem:chain-comparison} implies 
 $| x_m, y_m | \le | \bar{x}_m, \bar{y}_m |$.
This is a contradiction.
\end{proof}

%%%%%%%%%%
\subsection{Curvature bounds on ruled surfaces}

Let $\hat{\triangle} u_*v_*w_*$ be a geodesic triangle in $(R_*,e_\sigma)$ with distinct vertices 
and with perimeter $< 2D_{\kappa}$
determined by
$\hat{\triangle}u_*v_*w_* = \widehat{u_*v_*} \cup \widehat{v_*w_*} \cup \widehat{w_*u_*}$,
where $\widehat{u_*v_*}$, $\widehat{v_*w_*}$, and $\widehat{w_*u_*}$
are the edges of $\hat{\triangle}u_*v_*w_*$.
%Let $w_*' \in \widehat{u_*v_*} \setminus \{ u_*, v_* \}$,
%$u_*' \in \widehat{v_*w_*} \setminus \{ v_*, w_* \}$,
%and $v_*' \in \widehat{w_*u_*} \setminus \{ w_*, u_* \}$
%be arbitrary.
%We may assume that
%$u_*'$, $v_*'$, and $w_*'$ are distinct to each other.

We denote by $\triangle \tilde{u_*}\tilde{v_*}\tilde{w_*}$ a comparison triangle
in $M_{\kappa}^2$ for $\hat\triangle u_*v_*w_*$
with the same side-lengths,
and by $\tilde{\theta}_{u_*}(v_*,w_*)$ the angle 
$\angle \tilde{v_*}\tilde{u_*}\tilde{w_*}$
at $\tilde{u_*}$ between $\tilde{u_*}\tilde{v_*}$ and $\tilde{u_*}\tilde{w_*}$.

In order to complete the proof of Theorem \ref{thm:alex-ruled2},
it suffices to show the following
(see e.g., \cite[Proposition II.1.7]{bridson-haefliger}).

\begin{lem}\label{lem:angle-comparison}
Every geodesic triangle 
$\hat{\triangle} u_*v_*w_*$ in $(R_*,e_\sigma)$ as above 
%with distinct vertices and with perimeter $< 2D_{\kappa}$
satisfies the convexity of angle $\kappa$-comparison$:$
Namely for all $w_*' \in \widehat{u_*v_*} \setminus \{ u_*, v_* \}$,
$u_*' \in \widehat{v_*w_*} \setminus \{ v_*, w_* \}$,
and $v_*' \in\widehat{w_*u_*} \setminus \{ w_*, u_* \}$,
we have the following monotonicity$\,:$
\begin{gather*}
\tilde{\theta}_{u_*}(v_*',w_*') \le \tilde{\theta}_{u_*}(v_*, w_*),
\quad
\tilde{\theta}_{v_*}(w_*',u_*') \le \tilde{\theta}_{v_*}(w_*,u_*),
\quad\\
\tilde{\theta}_{w_*}(u_*',v_*') \le \tilde{\theta}_{w_*}(u_*, v_*).
\end{gather*}
%\[
%\tilde{\theta}_{u_*}(v_*',w_*') \le \tilde{\theta}_{u_*}(v_*, w_*),
%\quad
%\tilde{\theta}_{v_*}(w_*',u_*') \le \tilde{\theta}_{v_*}(w_*,u_*),
%\quad
%\tilde{\theta}_{w_*}(u_*',v_*') \le \tilde{\theta}_{w_*}(u_*, v_*).
%\]
\end{lem}

Before proving Lemma \ref{lem:angle-comparison},
we show the following sublemma.
By Proposition \ref{prop:e-sigma-shortest}, 
for every minimal geodesic $c_*$ in $(R_*,e_\sigma)$
there exists a monotone curve $c$ in $R$ with $\pi \circ c = c_*$
up to monotone parametrization.

\begin{slem}\label{slem:geodesics-ruled-surface-approximation}
In the same setting as in Lemma \ref{lem:angle-comparison},
let $u_*, u_*' \in (R_*,e_\sigma)$ be distinct points with $e_{\sigma}(u_*,u_*') < D_{\kappa}$,
and let $c_*$ be a minimal geodesic in $(R_*,e_\sigma)$
from $u_*$ to $u_*'$.
Assume $p_1(u) \le p_1(u')$ and 
choose a sequence $\{ \Delta_n \}_{n \in \N}$
of decompositions $\Delta_n = \{ s_i \}_{i = 0, 1, \dots, n}$
of $[p_1(u),p_1(u')]$
with $\lim_{n \to \infty} | \Delta_n | = 0$.
For $n \in \N$,
let $\{ x_i \}_{i = 0, 1, \dots, n }$
be the $\Delta_n$-minimizing chain along $S$ from $x:=\sigma(u)$ to $x':=\sigma(u')$,
and take a sequence $\{ y_i \}_{i = 0, 1, \dots, n }$ 
in the image of $\gamma:=\sigma_*\circ c_*$
in such a way that $y_0 = x$, $y_n = x'$,
and $y_i \in \lambda_{s_i}$ for all $i \in \{ 1, \dots, n-1 \}$.
Then we have the following $:$
\begin{enumerate}
\item
\[
e_{\sigma}(u_*, u_*') = \lim_{n \to \infty} \sum_{i=1}^n | y_{i-1}, y_i | ;
\]
\item For every $s \in [p_1(u),p_1(u')]$,
and for every sequence $\{ s_{i_n} \}_{n \in \N}$ 
converging to $s$
with $s_{i_n} \in \Delta_n$,
we have
\[
\lim_{n \to \infty} | x_{i_n}, y_{i_n} | = 0.
\]
\end{enumerate}
\end{slem}

\begin{proof}
(1)\,From Lemma \ref{lem:non-metric},
we derive 
$e_{\sigma}(u_*,u_*') = \lim_{n \to \infty} \sum_{i=1}^n | x_{i-1}, x_i |$;
moreover,
$e_{\sigma}(u_*,u_*') = \lim_{n \to \infty} \sum_{i=1}^n | y_{i-1}, y_i |$.
Indeed, we have
\begin{align*}
e_{\sigma}(u_*,u_*') &= \lim_{n \to \infty} \sum_{i=1}^n | x_{i-1}, x_i |
\le \liminf_{n \to \infty} \sum_{i=1}^n | y_{i-1}, y_i | \\
&\le \limsup_{n \to \infty} \sum_{i=1}^n | y_{i-1}, y_i |
\le \limsup_{n \to \infty} \sum_{i=1}^n e_{\sigma}(y_{i-1}, y_i)
= e_{\sigma}(u_*,u_*').
\end{align*}
(2)\,
For $n \in \N$,
let
$P_n = 
\left( \bigcup_{i=1}^n x_{i-1}x_i \right) \cup 
\left( \bigcup_{i=1}^n y_{i-1}y_i \right)$
be the polygon in $X$.
In the model surface $M_{\kappa}^2$,
we construct a comparison $(n+1)$-gon $\bar{P}_n$ 
for $P_n$ as follows:
Let $\triangle \tilde{x}\tilde{x}_1\tilde{y}_1$ 
and $\triangle \tilde{x}_{n-1}\tilde{y}_{n-1}\tilde{x}'$ be 
comparison triangles in $M_{\kappa}^2$
for $\triangle xx_1y_1$ and $\triangle x_{n-1}y_{n-1}x'$,
respectively.
For each $i \in \{ 1, \dots, n-1 \}$,
take comparison triangles 
$\triangle \tilde{x}_i\tilde{x}_{i+1}\tilde{y}_i$ and 
$\triangle \tilde{x}_{i+1}\tilde{y}_i\tilde{y}_{i+1}$ 
in $M_{\kappa}^2$
for $\triangle x_ix_{i+1}y_i$ and $\triangle x_{i+1}y_iy_{i+1}$, respectively,
and then glue all the comparison triangles in $M_{\kappa}^2$ along
$\tilde{x}_i\tilde{y}_i$,
and along $\tilde{x}_{i+1}\tilde{y}_i$,
for all $i \in \{ 1, \dots, n-1 \}$.
Then we put
$\tilde{P}_n :=
\left( \bigcup_{i=1}^n \tilde{x}_{i-1}\tilde{x}_i \right) \cup 
\left( \bigcup_{i=1}^n \tilde{y}_{i-1}\tilde{y}_i \right)$.
From Lemma \ref{lem:overpi}
it follows that
for each $i \in \{ 1, \dots, n-1 \}$
the inner angle at $\tilde{x}_i$ in $\tilde{P}_n$
is at least $\pi$.
By stretching the polygon $\tilde{P}_n$ at the concave vertices,
we obtain an $(n+1)$-gon 
$\bar{P}_n = \bar{x}\bar{x}' \cup 
\left( \bigcup_{i=1}^n \bar{y}_{i-1}\bar{y}_i \right)$ in $M_{\kappa}^2$
whose side-lengths satisfy
$| \bar{x}, \bar{x}' | = e_{\sigma}^{\Delta_n}(u_*,u_*')$
and
$| \bar{y}_{i-1}, \bar{y}_i | = | y_{i-1}, y_i |$
for all $i \in \{ 1, \dots, n \}$.
Let $\bar{x}_i \in \bar{x}\bar{x}'$,
$i \in \{ 1, \dots, n-1 \}$,
be the points corresponding to $\tilde{x}_i$.
The Alexandrov stretching lemma
(see e.g., \cite[Lemma I.2.16]{bridson-haefliger})
leads to that 
$| x_i, y_i | \le
| \bar{x}_i, \bar{y}_i |$ for all $i \in \{ 1, \dots, n-1 \}$.

Suppose that the second half of the sublemma is false.
Then we find $s \in (p_1(u),p_1(u'))$,
and a sequence $\{ s_{i_n} \}_{n \in \N}$ 
converging to $s$
such that for all $n \in \N$
we have $s_{i_n} \in \Delta_n$,
and we have
$| x_{i_n}, y_{i_n} | \ge C$ for some $C>0$.
Then
for the points $\bar{x}_{i_n}, \bar{y}_{i_n}$ 
on the comparison $(n+1)$-gon $\bar{P}_n$ for $P_n$
we have
\[
C \le \liminf_{n \to \infty} | x_{i_n}, y_{i_n} | \le 
\liminf_{n \to \infty} | \bar{x}_{i_n}, \bar{y}_{i_n} |.
\]
On the other hand,
since $e_{\sigma}(u_*,u_*') = \lim_{n \to \infty} \sum_{i=1}^n | x_{i-1}, x_i |$,
and 
since $e_{\sigma}(u_*,u_*') = \lim_{n \to \infty} \sum_{i=1}^n | y_{i-1}, y_i |$,
the comparison $(n+1)$-gon $\bar{P}_n$ degenerates in $M_{\kappa}^2$
as $n \to \infty$.
This yields a contradiction.
\end{proof}

\begin{proof}[Proof of Lemma \ref{lem:angle-comparison}]
%Let $\hat{\triangle} xyz$ be a triangle with distinct vertices 
%and with perimeter $< 2D_{\kappa}$
%in $(S,d_{\sigma})$ determined by
%$\hat{\triangle}xyz = \widehat{xy} \cup \widehat{yz} \cup \widehat{zx}$,
%where $\widehat{xy}$, $\widehat{yz}$, and $\widehat{zx}$
%are the edges of $\hat{\triangle}xyz$.
%Let $z' \in \widehat{xy} \setminus \{ x, y \}$,
%$x' \in \widehat{yz} \setminus \{ y, z \}$,
%and $y' \in \widehat{zx} \setminus \{ z, x \}$
%be arbitrary.
%We may assume that
%$x'$, $y'$, and $z'$ are distinct to each other.
%
%Take $u, v, w \in R$ with $\sigma(u) = x$, $\sigma(v) = y$, $\sigma(w) = z$.
Without loss of generality,
we may assume that 
\[
p_1(u) \le p_1(v) \le p_1(w).
\]
For each $n \in \N$,
let us choose a decomposition 
$\Delta_n = \{ s_i \}_{i = 0, 1, \dots, n}$
of $[p_1(u),p_1(w)]$ 
with $\lim_{n \to \infty} | \Delta_n | = 0$
such that $p_1(v) = s_m$ for some $m \in \{ 1, \dots, n-1 \}$.
Let $\Delta_n' := \{ s_i \}_{i = 0, 1, \dots, m}$
be the decomposition of $[p_1(u),p_1(v)]$,
and let $\Delta_n'' := \{ s_{m+i} \}_{i = 0, 1, \dots, n-m}$
be the decomposition of $[p_1(v),p_1(w)]$.
Set $x:=\sigma(u)$, $y:=\sigma(v)$, $z:=\sigma(w)$ and
take the (unique) 
$\Delta_n'$-minimizing chain $\{ y_i \}_{i = 0, 1, \dots, m}$
along $S$ from $x$ to $y$,
and the $\Delta_n''$-minimizing chain $\{ y_{m+i} \}_{i = 0, 1, \dots, n-m}$
along $S$ from $y$ to $z$,
and the $\Delta_n$-minimizing chain $\{ z_i \}_{i = 0, 1, \dots, n}$
along $S$ from $x$ to $z$.

Let $P^{\Delta_n}(xyz)$ be the $\Delta_n$-minimizing chain triple along $S$ defined by
\[
P^{\Delta_n}(xyz) :=
B^{\Delta_n}(xy) \cup B^{\Delta_n}(yz) \cup B^{\Delta_n}(zx),
\]
where
$B^{\Delta_n}(xy)$ is
the broken geodesic 
$\bigcup_{i=1}^m y_{i-1}y_i$ in $X$ joining $x$ and $y$,
$B^{\Delta_n}(yz)$ is
the broken geodesic 
$\bigcup_{i=1}^{n-m} y_{m+i-1}y_{m+i}$ in $X$ joining $y$ and $z$,
and $B^{\Delta_n}(zx)$ is
the broken geodesic 
$\bigcup_{i=1}^{n} z_{i-1}z_i$ in $X$ joining $z$ and $x$.
Set $x':=\sigma(u')$, $y':=\sigma(v')$ and  $z':=\sigma(w')$. 
By Sublemma \ref{slem:geodesics-ruled-surface-approximation},
we can take sequences 
$\{ y_{j_n} \}_{n \in \N}$,
$\{ y_{k_n} \}_{n \in \N}$,
$\{ z_{l_n} \}_{n \in \N}$
of broken points on $P^{\Delta_n}(xyz) \setminus \{ x, y, z \}$
satisfying
\[
\lim_{n \to \infty}
|y_{j_n},z'| = 0,
\quad
\lim_{n \to \infty}
|y_{k_n},x'| = 0,
\quad
\lim_{n \to \infty}
|z_{l_n},y'| = 0,
\]
where
$j_n \in \{ 1, \dots, m-1 \}$,
$k_n \in \{ m+1, \dots, n-1 \}$,
$l_n \in \{ 1, \dots, n-1 \}$.

Let $\bar{P}^{\Delta_n}(xyz) = \triangle \bar{x}\bar{y}\bar{z}$ be 
a comparison $\Delta_n$-minimizing stretched triangle
in $M_{\kappa}^2$ for $P^{\Delta_n}(xyz)$
whose side-lengths satisfy
\[
| \bar{x}, \bar{y} | = e_{\sigma}^{\Delta_n'}(u_*,v_*),
\quad
| \bar{y}, \bar{z} | = e_{\sigma}^{\Delta_n''}(v_*,w_*),
\quad
| \bar{z}, \bar{x} | = e_{\sigma}^{\Delta_n}(w_*,u_*).
\]
Set $\bar{\theta}_x^{\Delta_n}(y,z) := \angle \bar{y}\bar{x}\bar{z}$,
$\bar{\theta}_y^{\Delta_n}(z,x) := \angle \bar{z}\bar{y}\bar{x}$,
and $\bar{\theta}_z^{\Delta_n}(x,y) := \angle \bar{x}\bar{z}\bar{y}$.
Choose the three $\Delta_n$-minimizing chain triples
$P^{\Delta_n}(xy_jz_l)$,
$P^{\Delta_n}(y_jyy_k)$,
and $P^{\Delta_n}(z_ly_kz)$ along $S$,
and take
comparison $\Delta_n$-minimizing chain triangles
$\bar{P}^{\Delta_n}(xy_jz_l)$,
$\bar{P}^{\Delta_n}(y_jyy_k)$,
and $\bar{P}^{\Delta_n}(z_ly_kz)$
in $M_{\kappa}^2$
for $P^{\Delta_n}(xy_jz_l)$,
$P^{\Delta_n}(y_jyy_k)$,
and $P^{\Delta_n}(z_ly_kz)$,
respectively.
As shown in Lemma \ref{lem:delta-monotone},
we have the monotonicity
$\bar{\theta}_x^{\Delta_n}(y_j,z_l)
\le \bar{\theta}_x^{\Delta_n}(y,z)$,
$\bar{\theta}_y^{\Delta_n}(y_k,y_j)
\le \bar{\theta}_y^{\Delta_n}(z,x)$,
and
$\bar{\theta}_z^{\Delta_n}(z_l,y_k)
\le \bar{\theta}_z^{\Delta_n}(x,y)$.
From the choices of the sequences 
$\{ y_{j_n} \}_{n \in \N}$,
$\{ y_{k_n} \}_{n \in \N}$,
and $\{ z_{l_n} \}_{n \in \N}$,
it follows that
%%%%%%%
$\bar{P}^{\Delta_n}(xy_{j_n}z_{l_n})$,
$\bar{P}^{\Delta_n}(y_{j_n}yy_{k_n})$,
and
$\bar{P}^{\Delta_n}(z_{l_n}y_{k_n}z)$
%%%%%%
converge to comparison triangles
in $M_{\kappa}^2$
for triangles
$\hat{\triangle} u_*w_*'v_*'$,
$\hat{\triangle} w_*'v_*u_*'$,
and $\hat{\triangle} v_*'u_*'w_*$ in $(R_*, e_{\sigma})$,
respectively.
Notice that 
$\bar{P}^{\Delta_n}(xyz)$ converges to a comparison triangle in $M_{\kappa}^2$
for the triangle $\hat{\triangle} u_*v_*w_*$.
Therefore we obtain
\[
\tilde{\theta}_{u_*}(w_*', v_*') 
= \lim_{n \to \infty} \bar{\theta}_x^{\Delta_n}(y_{j_n},z_{l_n})
\le \lim_{n \to \infty} \bar{\theta}_x^{\Delta_n}(y,z)
= \tilde{\theta}_{u_*}(v_*, w_*).
\]
Similarly,
we see
$\tilde{\theta}_{v_*}(u_*',w_*') \le \tilde{\theta}_{v_*}(w_*, u_*)$,
and
$\tilde{\theta}_{w_*}(v_*',u_*') \le \tilde{\theta}_{w_*}(u_*,v_*)$.
Thus $\hat{\triangle} u_*v_*w_*$ satisfies the convexity
of angle $\kappa$-comparison.
\end{proof}

From Lemma \ref{lem:angle-comparison}
we conclude that
$(R_*,e_\sigma)$ is a $\CAT(\kappa)$-space.
This completes the proof of Theorem \ref{thm:alex-ruled2}.
\qed

%%%%start of the bibliography

\end{document}